\documentclass[11pt,reqno]{amsart}
\usepackage{hyperref}

\allowdisplaybreaks[4] 

\usepackage{amsmath}
\usepackage{amssymb}
\usepackage{mathrsfs}
\usepackage{amsthm}
\usepackage{bm}

\usepackage{booktabs}
\usepackage{float}
\usepackage{graphicx,epstopdf}
\usepackage[caption=false]{subfig}
\allowdisplaybreaks

\graphicspath{{Fig/}}
\usepackage{geometry}
\usepackage{color}
\theoremstyle{plain} 
\newtheorem{lemma}{Lemma}[section]
\newtheorem{thm}[lemma]{Theorem}

\newtheorem{remark}[lemma]{Remark}
\newtheorem{prop}[lemma]{Proposition}

\newcommand{\rmd}{\mathrm d}

\begin{document}
\title[Random attractor and SRB measure for stochastic Hopf bifurcation]{Random attractor and SRB measure for stochastic Hopf bifurcation under discretization}
\author[Chuchu Chen, Jialin Hong, Yibo Wang]{Chuchu Chen, Jialin Hong, Yibo Wang*}
\address{LSEC, ICMSEC, Academy of Mathematics and Systems Science, Chinese Academy of Sciences, Beijing 100190, China,
\and 
School of Mathematical Sciences, University of Chinese Academy of Sciences, Beijing 100049, China}
\email{chenchuchu@lsec.cc.ac.cn; hjl@lsec.cc.ac.cn; wangyb@amss.ac.cn}
\thanks{This work is funded by the National Key R\&D Program of China under Grant (No. 2024YFA1015900 and No. 2020YFA0713701),  by the National Natural Science Foundation of China (No. 12031020, No. 12461160278, and No. 12471386), and by Youth Innovation Promotion Association CAS}
\subjclass{37M22, 37M25, 60H35}
\thanks{*Corresponding author}
\begin{abstract}
	Chaotic phases in stochastic differential equations are characterized by two essential long-time dynamical features: a random attractor capturing asymptotic geometry and a Sinai--Ruelle--Bowen (SRB) measure describing statistical information. 
	This paper investigates whether the stochastic Hopf bifurcation under discretization could inherit both features. 
	We establish that the stochastic Hopf bifurcation under discretization induces a discrete random dynamical system. 
	Further, we prove that this discrete system possesses a random attractor, and then derive the existence of an SRB measure by demonstrating a strictly positive numerical Lyapunov exponent.
	Numerical experiments visualize the retained random attractor and SRB measure for the discrete random dynamical system, revealing structures consistent with the theoretical chaotic phase. 
\end{abstract}
\keywords{Random attractor $\cdot$ SRB measure $\cdot$ Lyapunov exponent $\cdot$ Discrete random dynamical system}
\maketitle
\section{Introduction}

The interplay between noise and nonlinearity generates rich dynamical phenomena that are absent in deterministic systems. Prominent examples include stochastic bifurcations and shear-induced chaos, where random perturbations fundamentally alter qualitative behavior \cite{Lin2008,Wieczorek2009}. 
Understanding these phenomena, particularly in their chaotic phases, requires capturing two essential long-time dynamical features: the random attractor, which provides the geometric structure about asymptotic regime of the system as $t\rightarrow\infty$, and the Sinai--Ruelle--Bowen (SRB) measure, which is supported on this attractor and describes the statistical information of the system. 
Importantly, the SRB measure ensures the equality of temporal and spatial averages on the sets with positive Lebesgue measure (i.e., observable events in physical experiments and numerical simulations), making it indispensable for quantifying chaotic dynamics  \cite{Blumenthal2019,Young2002}.

This paper investigates whether these dynamical features can be preserved under discretization for the stochastic Hopf bifurcation, modeled by the following two-dimensional stochastic differential equation (SDE): 
\begin{equation}\label{SDE}
	\rmd X_t  = \left( \begin{bmatrix}
		\alpha & -\beta \\ \beta & \alpha 
	\end{bmatrix} X_t 
	+ \|X_t\|^2 \begin{bmatrix}
		-a & -b \\ b & -a 
	\end{bmatrix} X_t \right) \rmd t + \sigma \begin{bmatrix}
		\rmd W_t^1 \\ \rmd W_t^2
	\end{bmatrix}, 
\end{equation}
where $b\in\mathbb{R}$ quantifies shear strength, $\sigma>0$ is the diffusion coefficient, $a>0$, $\alpha, \beta \in \mathbb{R}$, and $\{W_t^j\}_{t\geq0}$ for $j=1, 2$ are independent standard Brownian motions. The continuous random dynamical system (RDS) induced by \eqref{SDE} is known to undergo a Hopf bifurcation (see, e.g., \cite{Baxendale1994,Deville2011}) and exhibits two distinct dynamical phases governed by the shear strength $b$:  a synchronization phase for small shear $b$ where all trajectories converge to a random equilibrium, and a chaotic phase for large shear $b$ induced by dynamical instabilities. 
Our primary focus is on the chaotic phase of \eqref{SDE}. 
It has been established in \cite{Doan2018} that the continuous RDS induced by \eqref{SDE} possesses a random attractor.
Moreover, recent works \cite{Baxendale2024,Chemnitz2023} have provided rigorous proof of strictly positive lower bounds for the top Lyapunov exponent of SDE \eqref{SDE}. 
Based on these results, it can be shown that there exists an SRB measure for the continuous system of \eqref{SDE} (see Proposition \ref{Prop:SRB}). 

%
%Therefore, the theory in \cite{Chekroun2011,Ledrappier1988}, which establishes the sufficient conditions for an SRB measure, together with the proven existence of a random attractor and a positive Lyapunov exponent, yields the existence of SRB measure for the continuous system. 

Since analytical solutions of nonlinear SDEs are generally unavailable, the random attractor and the SRB measure cannot be characterized explicitly, which necessitates the development of numerical methods to capture these dynamical features. While significant progress has been made in the numerical analysis of SDEs, particularly regarding the convergence, stability, and ergodicity of various schemes (see, e.g., \cite{Chen20252,Chen2025,Chen2023,Higham2002,Mao2013}), research into numerical studies %approximation and preservation 
of the random attractor and the SRB measure remains in its infancy \cite{Arnaudon2017,Arnaudon2018,Chekroun2011}. 
This motivates a fundamental question: Can numerical methods inherit the essential dynamical features from the underlying continuous RDS?

This paper gives an answer to this question and demonstrates that the key dynamical features, specifically the random attractor and the SRB measure, of the original RDS induced by \eqref{SDE} can be preserved under the discrete RDS generated by backward Euler method. 
In order to study the inheritance of the random attractor, we construct an appropriate random absorbing set by transforming the numerical approximation into a discrete random differential equation. 
This random absorbing set provides a sufficient condition to guarantee the existence of a random attractor for the discrete RDS. 
Furthermore, the existence of an SRB measure for the discrete RDS can be established provided that the corresponding Lyapunov exponent is positive. 
Proving this positivity is essential yet challenging, as the direct proof is nontrivial and typically involves complicated evaluations of numerical approximations and their variational properties. 
To overcome this difficulty, we relate the Lyapunov exponent of discrete RDS to that of continuous system via the numerical approximation error. 
Thus, by establishing the strong convergence of backward Euler method, we prove that the Lyapunov exponent of the discrete RDS remains positive provided that the continuous system exhibits a positive Lyapunov exponent, which in turn ensures the existence of an SRB measure for discrete RDS. 
In practical computations, the random attractor and the SRB measure are obtained by evolving a large set of initial conditions under fixed noise realization to approximate the asymptotic state (see, e.g., \cite{Doan2018,Arnold2000,Ochs2001,Keller1999}).
This approach leverages the fact that SRB measures, by ensuring the equivalence of temporal and spatial averages for observable sets (i.e., the sets with positive Lebesgue measure), are the natural candidates for quantifying statistics from such simulations \cite{Eckmann1985}. 
Crucially, our theoretical analysis on the existence of random attractor and SRB measure for discrete RDS offers a rigorous justification previously lacking for this computational approach in RDSs, thus supporting more reliable numerical exploration of chaotic phenomena in applications such as physics and biology.

Our theoretical results are complemented by numerical experiments that visualize the random attractor and the SRB measure. 
These computations confirm that the discrete RDS reproduces the Hopf bifurcation exhibited in the continuous system. 
In particular, under large shear strength, the numerical results clearly demonstrate a positive Lyapunov exponent along with the emergence of geometrically complicated random attractor and associated SRB measure, which provides the numerical evidence for the chaotic phase of discrete RDS.

This paper is organized as follows. Section \ref{Sec:preliminary} begins by reviewing essential concepts in the RDS theory, and then presents the existence of a random attractor and an SRB measure for continuous RDS. In Section \ref{Sec:discrete RDS}, we establish the existence of a random attractor for the discrete RDS induced by backward Euler method. We also demonstrate that the discrete RDS possesses an SRB measure by proving the positivity of numerical Lyapunov exponent. 
Section \ref{Sec:proof} provides detailed proofs of propositions used in Section \ref{Sec:discrete RDS}, including the ergodicity of numerical approximation and the strong convergence of backward Euler method. In Section \ref{Sec:numerical example}, we present several numerical experiments that validate our analytical findings and offer insightful visualizations of the dynamical behaviors predicted by the theory.

\section{Preliminary}\label{Sec:preliminary}
\subsection{Notation and concepts}
In this paper we denote $\|\cdot\|$ the Euclidean norm of vectors in $\mathbb{R}^2$, and denote $X\cdot Y = x_1y_1 + x_2y_2$ for $X=(x_1,x_2)^{\top}$ and $Y=(y_1,y_2)^{\top}$, where $X,Y\in\mathbb{R}^2$ and the superscript $\top$ represents the transpose of vectors. 
We use $C$ to denote an unspecified positive and finite constant, which may vary from one line to another but is always independent of the discretization parameters. 
Constants depending on
certain parameters $a,b,...$ are numbered as $C(a,b,...)$.

In order to demonstrate the random attractor and the SRB measure, we will briefly introduce those concepts in the theory of  RDS which are relevant for our case.  
We refer the reader to \cite{arnoldRDS} for a more detailed and systematic treatment.  
In the sequel $\mathbb{T}=\mathbb{R}$ or $\mathbb{Z}$ stands for the time. 
An RDS $\varphi$ on the state space $\mathbb{R}^2$ with time $\mathbb{T}$ is a pair of the following objects: 

(i) \textit{Model of the noise}: Let $(\Omega,\mathcal{F},\mathbb{P},\theta)$ be a metric dynamical system, i.e., $(\Omega,\mathcal{F},\mathbb{P})$ is a probability space, and $\theta$ is a flow of mappings $\{\theta_{t}\}_{t\in\mathbb{T}}$ on $\Omega$ (i.e., $\theta_{0}=\text{id}_{\Omega}$, $\theta_{t+s} = \theta_{t}\circ\theta_{s}$ for all $t, s\in\mathbb{T}$), which leaves the measure $\mathbb{P}$ invariant.  For simplicity we assume that $\theta$ is ergodic.

(ii) \textit{Model of the system perturbed by noise}: Let $\varphi$ be a cocycle over $\theta$, i.e., $\varphi$ is a measurable mapping: $\mathbb{T}\times\Omega\times\mathbb{R}^{2} \rightarrow \mathbb{R}^{2}$, $(t,\omega,x)\mapsto\varphi(t,\omega,x)$, such that $(t,x)\mapsto\varphi(t,\omega,x)$ is continuous for all $\omega\in\Omega$ and $\varphi$ satisfies the cocycle property: 
\begin{equation*}
	\varphi(0,\omega,\cdot) = \text{id}_{\mathbb{R}^2}(\cdot), \quad \varphi(t+s,\omega,\cdot) = \varphi(t,\theta_{s}\omega,\cdot)\circ\varphi(s,\omega,\cdot), \quad \text{for all} \ t, s \in\mathbb{T} \ \text{and} \ \omega\in\Omega, 
\end{equation*}
where $\circ$ means the composition of mappings. 
Then $\varphi$ is said to be an RDS (resp. a discrete RDS) on the state space $\mathbb{R}^2$ over a metric dynamical system $(\Omega,\mathcal{F},\mathbb{P},\theta)$ with time $\mathbb{T}=\mathbb{R}$ (resp. $\mathbb{T}=\mathbb{Z}$). 

Next we introduce the concept of random attractor. Suppose that there exists a random compact set $A(\omega)$ satisfying $\varphi(t,\omega,A(\omega))=A(\theta_{t}\omega)$ for all $t>0$, $\omega\in\Omega$. If $A$ attracts every bounded deterministic set $B\subset \mathbb{R}^{2}$, i.e., 
\begin{equation*}
	\lim\limits_{t\rightarrow\infty} d\left(\varphi(t,\theta_{-t}\omega,B) , A(\omega)\right) = 0, \quad \text{for all} \ \omega \in \Omega, 
\end{equation*}
then $A$ is said to be a random attractor for $\varphi$. Here $d(B_1,B_2)=\sup\{\inf\{d(x,y):y\in B_1\}: x\in B_2\}$ denotes the Hausdorff semi-distance between subsets $B_1$ and $B_2$ in $\mathbb{R}^2$ (see \cite{Crauel1994}). 
The random attractor provides the geometric information about asymptotic regime of RDS as $t\rightarrow\infty$. 

The SRB measure, supported on this random attractor, describes the statistical information of RDS. 
In the sequel we introduce the concept of SRB measure. 
Suppose that $\mu$ is a probability measure on $\Omega\times\mathbb{R}^2$ and
symbolically has the disintegration $\mu(\rmd\omega,\rmd x) = \mu_{\omega}(\rmd x)\mathbb{P}(\rmd\omega)$, where $\mu_{\omega}$ is the sample measure of $\mu$. 
If $\varphi^{*}(t,\omega,\cdot)\mu_{\omega} = \mu_{\theta_{t}\omega}$ for all $t\in\mathbb{R}$, $\omega\in\Omega$, then $\mu$ (resp. $\mu_{\omega}$) is said to be the invariant measure  (resp. invariant sample measure) of $\varphi$, where $\varphi^{*}(t,\omega,\cdot)\mu_{\omega}$ 
denotes the push-forward of $\mu_{\omega}$ under $\varphi(t,\omega,\cdot)$ (see \cite{arnoldRDS}).  
An invariant sample measure $\mu_{\omega}$ is called the SRB measure if $\varphi$ has a positive Lyapunov exponent, $\mu_{\omega}$ has absolutely continuous conditional measure on unstable manifold, and $\mu_{\omega}$ satisfies the following identity for all continuous functions $G$: 
\begin{equation}\label{physical measure}
	\lim\limits_{t\rightarrow\infty} \frac{1}{t} \int_{0}^{t} G\circ\varphi(s,\theta_{-s}\omega,x) \rmd s = \int_{A(\omega)} G(z) \mu_{\omega}(\rmd z), \quad \text{for all} \  \omega\in\Omega, \ x\in U, 
\end{equation}
where $U$ is a subset of $\mathbb{R}^2$ with positive Lebesgue measure (see \cite{Blumenthal2019,Young2017,Ochs2001}). 
The identity \eqref{physical measure} indicates that the spatial and temporal averages of observables coincide for a set of initial conditions having positive Lebesgue measure. 
Consequently, the SRB measure is also known as the natural or the physically relevant invariant measure.

Before concluding this subsection, we would like to mention another concept that we refer to as the stationary measure in the following sections. 
It is a fixed point of the Markov semigroup of associated SDE (denoted it by $\rho$) and is connected to the invariant sample measure via the relationship 
\begin{equation*}
	\rho(B)=\int_{\Omega} \mu_{\omega}(B) \mathbb{P}(\rmd \omega), \quad \text{for all} \ B\in\mathcal{B}(\mathbb{R}^2). 
\end{equation*}
The explicit expression of the stationary measure of SDE \eqref{SDE} can be found in, e.g., \cite{Deville2011,Doan2018}.

%Before proceeding, it is worth noting that another concept there is also a stochastic approach for studying the statistical information about dynamics of \eqref{SDE}. This approach is to seek the fixed point of the associated Markov semigroup of \eqref{SDE}, which we refer to as the stationary measure of the one-point Markov process in the following sections. 
%However, such an approach only models a single time series probabilistically and fails to capture many relevant dynamical properties, such as the comparison of evolution of nearby initial points under the same noise. 
%This is one reason why we work with the RDS approach in this paper. Its fundamental object is the sample measure of RDS rather than the stationary measure of Markov process. 

\subsection{Generation of an RDS}
In this subsection we address the generation of an RDS. 
We first rewrite \eqref{SDE} as 
\begin{equation}\label{compact SDE}
	\rmd X_{t} = F(X_{t}) \rmd t + \sigma \rmd W_t,
\end{equation} 
where $X_{t} = (x_{t}, y_{t})^{\top}\in\mathbb{R}^2$, $W_{t}$ is a $2$-dimensional Brownian motion, and 
\begin{equation} \label{nonlinearity}
	F: \mathbb{R}^2 \rightarrow \mathbb{R}^2, \quad 
	X \mapsto  \begin{bmatrix}
		\alpha & -\beta \\ \beta & \alpha 
	\end{bmatrix} X
	+ \|X\|^2 \begin{bmatrix}
		-a & -b \\ b & -a 
	\end{bmatrix}X.  
\end{equation} 
Without loss of generality we assume that $W_t$ is a two-sided Brownian motion. 
For our applications, we will consider a special metric dynamical system by choosing $(\Omega,\mathcal{F},\mathbb{P})$ to be the Wiener space. Specifically, $\Omega = C_{0}(\mathbb{R}; \mathbb{R}^2)$ is the space of all continuous functions $\omega: \mathbb{R} \rightarrow \mathbb{R}^2$ satisfying $\omega(0)=0$ endowed with the compact open topology, $\mathcal{F}$ is the Borel $\sigma$-algebra of $\Omega$, and $\mathbb{P}$ is the Wiener measure. We define the shift mappings on $\Omega$ as $\theta_t \omega(\cdot) := \omega(t+\cdot) - \omega(t)$. Then $(\Omega,\mathcal{F},\mathbb{P},\theta)$ is an ergodic metric dynamical system that drives the SDE \eqref{compact SDE}, with $W_{t}(\omega)=\omega(t)$. See \cite[Appendix]{arnoldRDS} for more details. 

Following standard procedures, we next state that the SDE \eqref{compact SDE} induces an RDS through its conjugation to a random differential equation via a suitable transformation. 
To this end, for $\gamma>0$, we introduce the linear equation 
\begin{equation}\label{OU1}
	\rmd Z = -\gamma Z \rmd t + \rmd\omega(t), 
\end{equation}
where $Z \in \mathbb{R}^2$, and define $Z^*(\omega) := -\int_{-\infty}^{0} e^{\gamma s} \omega(s)\rmd s$. Then $t \mapsto Z^*(\theta_t \omega)$ solves \eqref{OU1}, i.e., 
\begin{equation}\label{OU2}
	Z^*(\theta_t \omega) = Z^*(\omega) - \gamma \int_{0}^{t} Z^*(\theta_s \omega) \rmd s + \omega(t). 
\end{equation}
From \cite[Lemma 2.1]{Duan2003} we know that there exists a $\{\theta_t\}_{t\in\mathbb{R}}$-invariant subset $\Omega_1 \subset \Omega$ of full measure such that $\lim_{t\rightarrow\pm\infty} \frac{|\omega(t)|}{|t|}=0$ for all $\omega\in\Omega_{1}$. Moreover, for $\delta\in(\frac{1}{2},1)$ and $\omega \in \Omega_1$, there exist two positive constants $L_{\delta,\omega}$ and $\tilde{L}_{\delta,\omega}$ such that 
\begin{equation}\label{Z*}
	\|Z^*(\theta_t \omega)\|^2  \leq L_{\delta,\omega} + \tilde{L}_{\delta,\omega}|t|^{\delta}, \quad \forall \, t \in \mathbb{R}. 
\end{equation}
Replacing $\mathcal{F}$ by $\mathcal{F}_{1} = \{\Omega_1\cap A, A\in \mathcal{F}\}$ and considering the probability measure $\mathbb{P}_1$ which is the restriction of Wiener measure $\mathbb{P}$ to this new $\sigma$-algebra,  we obtain a metric dynamical system $(\Omega_1,\mathcal{F}_1,\mathbb{P}_1,\theta)$. 
In the following we focus on this metric dynamical system, and for convenience, still denote it by $(\Omega,\mathcal{F},\mathbb{P},\theta)$.

We introduce the mappings $\mathcal{T},\mathcal{T}^{-1}: \Omega\times\mathbb{R}^2\rightarrow\mathbb{R}^2$ defined by
\begin{equation}\label{mathcalT}
	\mathcal{T}(\omega,\bm{x}) := \bm{x} - \sigma Z^*(\omega), \quad \mathcal{T}^{-1}(\omega,\bm{x}) := \bm{x} + \sigma Z^*(\omega). 
\end{equation}
Then under the transformation $\hat{X}_{t} = \mathcal{T}(\theta_{t}\omega,X_{t})$, the SDE \eqref{compact SDE} can be transformed into the following random differential equation 
\begin{equation}\label{RDE1}
	\rmd \hat{X}_{t} = F\left(\mathcal{T}^{-1}(\theta_{t}\omega,\hat{X}_{t})\right) \rmd t + \gamma\sigma Z^*(\theta_t \omega) \rmd t. 
\end{equation}
Doan et al. \cite{Doan2018} have demonstrated that for any $\omega\in\Omega$ and $\bm{x}\in\mathbb{R}^2$, \eqref{RDE1} possesses a unique solution given by
\begin{equation*}
	\hat{\varphi}(t,\omega,\bm{x}) = \bm{x} + \int_{0}^{t} F\left(\mathcal{T}^{-1}(\theta_{s}\omega,\hat{\varphi}(s,\omega,\bm{x}))\right) \rmd s + \gamma\sigma\int_{0}^{t} Z^*(\theta_s\omega) \rmd s, \quad \forall \,  t\geq0, 
\end{equation*}
and $\hat{\varphi}$ forms an RDS. 
Moreover, it can be verified that the mapping $\varphi$ determined by the conjugacy relation 
\begin{equation}\label{conjugacy}
	\varphi(t, \omega, \bm{x}) = \mathcal{T}^{-1}(\theta_{t}\omega, \cdot) \circ \hat{\varphi}(t, \omega, \cdot) \circ \mathcal{T}(\omega,\bm{x})
\end{equation}
solves the SDE \eqref{compact SDE}, i.e., for any  $\omega\in\Omega$ and $\bm{x}\in\mathbb{R}^2$, $\varphi$ satisfies the integral equation
\begin{equation*}
	\varphi(t, \omega, \bm{x}) = \bm{x} + \int_{0}^{t} F(\varphi(s, \omega, \bm{x})) \rmd s + \sigma \omega(t), \quad \forall \,  t\geq0. 
\end{equation*}
It follows from the cocycle property of $\hat{\varphi}$ and \eqref{conjugacy} that the mapping $\varphi$ induced by \eqref{compact SDE} is a cocycle, and thus forms an RDS over $(\Omega,\mathcal{F},\mathbb{P},\theta)$.

\subsection{Random attractor and SRB measure for the RDS}
We first present the results concerning random attractor and invariant sample measure for the RDS $\varphi$.  
Following \cite{Doan2018}, it suffices to construct a random absorbing set for $\varphi$. In fact, one can show that for all $\bm{x}\in\mathbb{R}^2$ and $\omega\in\Omega$, 
\begin{equation*}
	\|\varphi(t, \theta_{-t}\omega, \bm{x})\|^2 \leq e^{-\sqrt{a}t} \|\bm{x}-\sigma Z^{*}(\theta_{-t}\omega)\|^2 + C\int_{0}^{t} e^{-\sqrt{a}r} \left( \|Z^{*}(\theta_{-r}\omega)\|^4 + 1 \right) \rmd r + 2\sigma^2\|Z^{*}(\omega)\|^2, 
\end{equation*}
where the positive constant $C$ depends on the coefficients of SDE \eqref{compact SDE}. 
Define 
\begin{equation*}
	\mathcal{A}(\omega) := 1+C\int_{0}^{\infty} e^{-\sqrt{a}r} \left( \|Z^{*}(\theta_{-r}\omega)\|^4 + 1 \right) \rmd r + 2\sigma^2\|Z^{*}(\omega)\|^2. 
\end{equation*}
Then it follows from $\|Z^*(\omega)\|<\infty$ and \eqref{Z*} that $|\mathcal{A}(\omega)|<\infty$ for all $\omega\in\Omega$. Thus the random set defined by
\begin{equation*}
	K(\omega) := \left\{ \bm{x}\in\mathbb{R}^2: \|\bm{x}\|^2 \leq \mathcal{A}(\omega) \right\}
\end{equation*}
is a random absorbing set. From this, one can derive the following results. 
\begin{prop}[\cite{Doan2018}]\label{Prop:attractor}
	There exists a random attractor $A(\omega)$ for the RDS $\varphi$, i.e., for every bounded deterministic set $B\subset \mathbb{R}^{2}$, 
	\begin{equation*}
		\lim\limits_{t\rightarrow\infty} d\left(\varphi(t,\theta_{-t}\omega,B) , A(\omega)\right) = 0, \quad \text{for all} \ \omega \in \Omega. 
	\end{equation*} 
	In addition, the sample measure $\mu_{\omega}:=\delta_{A(\omega)}$ is invariant under $\varphi$.  
\end{prop}

Next, we introduce the concept of the Lyapunov exponent, the sign of which is crucial in determining whether the sample measure $\mu_{\omega}$ is an SRB measure. 
Linearizing \eqref{SDE} along a trajectory $\{X_t: t \geq 0\}$ gives the process $\{U_t: t \geq 0\}$ satisfying 
\begin{equation}\label{variational SDE}
	\rmd U_t  = \left( \begin{bmatrix}
		\alpha & -\beta \\ \beta & \alpha 
	\end{bmatrix} U_t 
	+ \|X_t\|^2 \begin{bmatrix}
		-a & -b \\ b & -a 
	\end{bmatrix} U_t + 2 X_t \cdot U_t  \begin{bmatrix}
		-a & -b \\ b & -a 
	\end{bmatrix} X_t \right) \rmd t. 
\end{equation}
Here $U_{t}\in\mathbb{R}^2$ should be regarded as a vector at the point $X_{t}$. 
The \textit{(top) Lyapunov exponent}, which characterizes the long-time rate of exponential growth or decay of the process $\{U_t: t\geq0\}$, is defined by
\begin{equation}\label{lyp1}
	\lambda = \lambda(\alpha, a, b, \sigma) := \lim\limits_{t \rightarrow \infty} \frac{1}{t} \log \|U_t\|, 
\end{equation}
which depends on the coefficients of SDE \eqref{compact SDE}. 
Denote $\gamma_t = \|X_t\|$. Then we quote from \cite[Theorem 3.1]{Baxendale2024} the following theorem which addresses the positivity of $\lambda$. 
\begin{prop}\label{Prop:positive lyp}
	Given $\alpha\in\mathbb{R}$, $a>0$ and $\sigma>0$, the process $\gamma_{t}$ has a unique stationary measure $\tilde{\nu}$. Moreover, 
	\begin{equation*}
		\lambda(\alpha,a,b,\sigma) \sim (2b\sigma)^{\frac{2}{3}} \lambda_{0} \int_{\mathbb{R}^{+}} \gamma^{\frac{2}{3}} \tilde{\nu}(\rmd \gamma) \qquad \text{as} \ b \rightarrow \infty, 
	\end{equation*}
	where the constant  $\lambda_{0}=\frac{\pi}{2^{1/3}3^{1/6}|\Gamma(1/3)|^2} \approx 0.29$ with $\Gamma(y)=\int_{0}^{\infty}t^{y-1}e^{-t}\rmd t$ for $y>0$ being the Gamma function. 
	Here $\lambda_1\sim\lambda_2$ means that there exist constants $C, \tilde{C}>0$ such that $C \lambda_1<\lambda_2<\tilde{C} \lambda_1$. 
\end{prop}

We next assert that the sample measure $\mu_{\omega}$ is an SRB measure when $b$ exceeds a critical shear level. 
\begin{prop}\label{Prop:SRB}
	Under conditions of Proposition \ref{Prop:positive lyp}, there exists a constant $b_{0}>0$ such that the sample measure $\mu_{\omega}$ of $\varphi$ is an SRB measure whenever $b>b_{0}$. 
\end{prop}
\begin{proof}
	Following \cite[Theorem C.2]{Chekroun2011}, to prove that $\mu_{\omega}$ is an SRB measure, it suffices to verify that (i) $\varphi$ possesses a random attractor, (ii) H\"ormander’s condition is satisfied, and (iii) there exists a positive Lyapunov exponent associated with \eqref{compact SDE}. We now proceed to verify these conditions. 
	
	First, we have already demonstrated the existence of the random attractor $A(\omega)$ in Proposition \ref{Prop:attractor}. 
	Second, H\"ormander’s condition is satisfied since the driven noise of \eqref{compact SDE} is non-degenerate (i.e., $\sigma>0$) and the nonlinear term $F$ is a $2$-dimensional $C^{\infty}$ vector field of $\mathbb{R}^2$. 
	Third, according to Proposition \ref{Prop:positive lyp}, there exists a constant $b_{0}>0$, such that the top Lyapunov exponent $\lambda>0$ whenever $b>b_{0}$. 
	In summary, the sample measure of $\varphi$ is an SRB measure whenever $b>b_{0}$. 
\end{proof}

We end this section by providing a useful lemma concerning $\lambda$. Introduce the polar coordinates $X_t = \gamma_t (\cos\phi_t,\sin\phi_t)^{\top}$ with $\gamma_t = \|X_t\|$ and $U_t= \|U_t\| (\cos \xi_t,\sin \xi_t)^{\top}$. By It\^o's formula applied to \eqref{variational SDE} we can derive 
\begin{equation}\label{logU}
	\rmd \log \|U_t\| = \hat{Q}(X_t,\xi_t) \rmd t, 
\end{equation}
where the function $\hat{Q}:\mathbb{R}^{2}\times\mathbb{R}/(\pi\mathbb{Z})\rightarrow\mathbb{R}$ is given by 
\begin{equation}\label{hatQ}
	\hat{Q}(x, y, \xi) = \alpha - 2a (x^2+y^2) - (2b\cos2\xi+2a\sin2\xi) xy + (b\sin2\xi-a\cos2\xi) (x^2-y^2).  
\end{equation}
For $p\geq2$, we define the Lyapunov function $\Gamma: \mathbb{R}^+ \times \mathbb{R}/(\pi\mathbb{Z}) \rightarrow [1, \infty)$ as 
\begin{equation}\label{Lyapunov function}
	\Gamma(\gamma, \psi) := \gamma^{2p} + \gamma^2\psi^{2p} + M, 
\end{equation} 
where $M\geq1$ is an arbitrary constant. 
Then we have the following lemma. 
\begin{lemma}
	There exist $\iota\in(0,1)$ and $\kappa\in(0,\infty)$ such that for any $t\geq0$, 
	\begin{equation}\label{ergodic2}
		\left| \mathbb{E}^{X_{0},\xi_{0}}[\hat{Q}(X_t,\xi_t)] - \lambda \right| 
		\leq \kappa \iota^t \Gamma(\gamma_{0},\psi_{0}), 
	\end{equation}
	where $\gamma_{0}=\|X_{0}\|$ and $\psi_{0}=\xi_{0}-\phi_{0}$. 
\end{lemma}
\begin{proof}
	Recall that $X_t = \gamma_t (\cos\phi_t,\sin\phi_t)^{\top}$. 
	Then It\^o's formula applied to \eqref{compact SDE} gives 
	\begin{equation}\label{gamma&phi}
		\begin{cases}
			\rmd \gamma_t = \left(  \alpha\gamma_t-a\gamma_t^3+\frac{\sigma^2}{2\gamma_t}  \right) \rmd t + \sigma\rmd W_t^\gamma, \\
		\rmd \phi_t = (\beta+b\gamma_t^2) \rmd t + \frac{\sigma}{\gamma_t} \rmd W_t^\phi, 
		\end{cases}
	\end{equation}
	where $W_t^\gamma$ and $W_t^\phi$ are two independent standard Brownian motions. 
	Furthermore, $\xi_t \in \mathbb{R}/(\pi\mathbb{Z})$ satisfies 
	\begin{equation}\label{xi}
		\rmd \xi_t 
		= \left( \beta+2b\gamma_t^2 + \gamma_t^2 \left( a\sin(2\xi_t-2\phi_t)+b\cos(2\xi_t-2\phi_t) \right)  \right) \rmd t. 
	\end{equation}
	By defining $\psi_t := \xi_t - \phi_t$, it follows from \eqref{gamma&phi} and \eqref{xi} that 
	\begin{equation*}
			\rmd \psi_t =  2\gamma_t^2 \cos\psi_t \left(a\sin\psi_t+b\cos\psi_t\right)  \rmd t - \frac{\sigma}{\gamma_t} \rmd W_t^\phi.
	\end{equation*}
Integrating \eqref{logU} from $0$ to $t$ yields 
\begin{equation}\label{lyp2}
	\frac{1}{t} \log \|U_t\| = \frac{1}{t} \log \|U_0\| + \frac{1}{t} \int_{0}^{t} \hat{Q}(X_s,\xi_s) \rmd s
	= \frac{1}{t} \log \|U_0\| + \frac{1}{t} \int_{0}^{t} Q(\gamma_s, \psi_s) \rmd s, 
\end{equation} 
where the function $Q:\mathbb{R}^{+}\times\mathbb{R}/(\pi\mathbb{Z})\rightarrow\mathbb{R}$, in view of \eqref{hatQ}, satisfies 
\begin{equation}\label{hatQ and Q}
	\begin{aligned}
		\hat{Q}(X_t,\xi_t) &= \alpha - 2a \|X_t\|^2 - (2b\cos2\xi_t+2a\sin2\xi_t) x_t y_t + (b\sin2\xi_t-a\cos2\xi_t) (x_t^2-y_t^2) \\
		&= \alpha - 2a\gamma_t^2 + \gamma_t^2(b\sin2\psi_t-a\cos2\psi_t) =: Q(\gamma_t, \psi_t). 
	\end{aligned}
\end{equation}

We next prove the ergodicity of the process $\{(\gamma_t,\psi_t): t\geq0\}$. 
Following \cite{Mattingly2002}, it suffices to verify the Lyapunov condition and the minorization condition.
Denote by $\mathcal{L}$ the corresponding generator, i.e.,  
\begin{align*}
	\mathcal{L} = \left(  \alpha\gamma-a\gamma^3+\frac{\sigma^2}{2\gamma}  \right) \frac{\partial}{\partial \gamma} + \frac{\sigma^2}{2} \frac{\partial^2}{\partial \gamma^2} + 2\gamma^2 \cos\psi (a\sin\psi+b\cos\psi) \frac{\partial}{\partial \psi} + \frac{\sigma^2}{2\gamma^2} \frac{\partial^2}{\partial \psi^2}. 
\end{align*}
For the Lyapunov function $\Gamma$ defined by \eqref{Lyapunov function}, it is evident that  $\lim_{(\gamma, \psi)\rightarrow \infty} \Gamma(\gamma, \psi) = \infty$. 
By Young's inequality and $|\psi|\leq\pi$, we have 
\begin{align*}
	\mathcal{L}\Gamma(\gamma, \psi) 
	&= -2pa\gamma^{2p+2} + 2p\alpha\gamma^{2p} + 2p^2\sigma^2\gamma^{2p-2} 
	+ 4p\gamma^{4}\psi^{2p-1}\cos\psi (a\sin\psi+b\cos\psi)  \\
	&\quad - 2a\gamma^{4}\psi^{2p} + 2\alpha\gamma^{2}\psi^{2p} + 2\sigma^2 \psi^{2p} + p(2p-1)\sigma^2\psi^{2p-2} \\
	&\leq -a\Gamma(\gamma, \psi) + C(M, p, a, b, \alpha,  \sigma). 
\end{align*}
This implies that the Lyapunov condition holds for the process $\{(\gamma_t,\psi_t): t\geq0\}$. 
On the other hand, the driven noise of the process $\{(\gamma_t,\psi_t): t\geq0\}$ is non-degenerate, which implies the minorization condition. 
Consequently, 
%by using a result of Theorem 2.5 in \cite{Mattingly2002}, 
there exists a unique stationary measure $\nu$ for $\{(\gamma_t,\psi_t): t\geq0\}$ with the marginal $\tilde{\nu}$ for $\{\gamma_t: t\geq0\}$. 
Additionally, there exist $\iota\in(0,1)$ and $\kappa\in(0,\infty)$, such that for all measurable functions $f$ satisfying $|f|\leq\Gamma$, 
\begin{equation}\label{ergodic1}
	\left| \mathbb{E}^{\gamma_{0},\psi_{0}}[f(\gamma_t,\psi_t)] - \nu(f) \right| \leq \kappa \iota^t \Gamma(\gamma_{0},\psi_{0}). 
\end{equation}

Next we show that $\lambda=\nu(Q)$ when $f$ is taken to be the function $Q$, where $Q$ is given by \eqref{hatQ and Q}. 
Since $\gamma^2$ belongs to $L^1(\tilde{\nu})$ (see  \cite[Appendix A]{Baxendale2024}), we have $Q(\gamma,\psi)\in L^1(\nu)$. By \eqref{lyp1}, \eqref{lyp2}, and Birkhoff's almost-sure ergodic theorem, we obtain the Furstenberg--Khasminskii formula 
	\begin{equation}\label{FK formula}
		\lambda = \lim\limits_{t \rightarrow \infty} \frac{1}{t} \log \|U_t\| = \lim\limits_{t \rightarrow \infty} \frac{1}{t} \int_{0}^{t} Q(\gamma_s, \psi_s) \rmd s
		= \int_{\mathbb{R}^+ \times \mathbb{R}/(\pi\mathbb{Z})} Q(\gamma,\psi) \nu(\rmd \gamma,\rmd \psi) =: \nu(Q). 
	\end{equation}
	Furthermore, since $|Q|\leq\Gamma$, it follows from \eqref{hatQ and Q}, \eqref{ergodic1}, and \eqref{FK formula} that 
	\begin{equation*}
		\left| \mathbb{E}^{X_{0},\xi_{0}}[\hat{Q}(X_t,\xi_t)] - \lambda \right| 
		= \left| \mathbb{E}^{\gamma_{0},\psi_{0}}[Q(\gamma_t,\psi_t)] - \nu(Q) \right| \leq \kappa \iota^t \Gamma(\gamma_{0},\psi_{0}), 
	\end{equation*}
	which completes the proof. 
\end{proof}

\section{Random attractor and SRB measure for the discrete RDS}\label{Sec:discrete RDS}
In this section we adopt the backward Euler method to discretize the original SDE \eqref{compact SDE}. 
We demonstrate that the numerical approximation can induce a discrete RDS,  and moreover, establish the existence of a random attractor and an SRB measure for this system.

Given a step size $\tau>0$, the backward Euler method applied to SDE \eqref{compact SDE} computes approximations $X^{\tau}_k \approx 
X_{t_k}$, where $t_{k}=k\tau$, by forming $X^{\tau}_{0}=X_{t_{0}}$ and 
%\begin{equation}\label{BEM}
%	X^{\tau}_{k+1}  = X^{\tau}_{k} + \tau \left( \begin{bmatrix}
%		\alpha & -\beta \\[.2em] \beta & \alpha 
%	\end{bmatrix} X^{\tau}_{k+1} 
%	+ \|X^{\tau}_{k+1}\|^2 \begin{bmatrix}
%		-a & -b \\[.2em] b & -a 
%	\end{bmatrix} X^{\tau}_{k+1} \right)  + \sigma \Delta W_{k+1}, \quad \text{for} \ k\geq0. 
%\end{equation}
\begin{equation}\label{BEM}
	X^{\tau}_{k+1} = X^{\tau}_{k} + \tau F(X^{\tau}_{k+1}) + \sigma \Delta W_{k+1}, \qquad \text{for} \ k\geq0. 
\end{equation}
Here $F$ is given by \eqref{nonlinearity} and $\Delta W_{k+1} = W_{t_{k+1}}-W_{t_{k}}$. 
By Appendix \ref{Appendix:solvable of IES}, it follows that when $\tau<\min\{\frac{1}{1+4|\alpha|},\frac{a}{1+|a\alpha|+|b\beta|}\}$, the implicit equation \eqref{BEM} can be uniquely solved for any $k\geq0$, i.e., there exists a function $F_{\tau}:\mathbb{R}^2\rightarrow\mathbb{R}^2$ such that 
\begin{equation}\label{BEM2}
	X^{\tau}_{k+1} = F_{\tau} \left(  X^{\tau}_{k} + \sigma \Delta W_{k+1}  \right), \qquad \text{for} \ k\geq0. 
\end{equation}

\subsection{Discrete RDS generated by the numerical approximation}
Define the mapping $\varphi_{\tau}: \mathbb{Z} \times \Omega \times \mathbb{R}^2 \rightarrow \mathbb{R}^2$ as 
\begin{equation*}
	\varphi_{\tau}(k, \omega, \bm{x}) := F_{\tau} \left( \cdot + \sigma \Delta W_{k} \right) \circ  F_{\tau} \left( \cdot + \sigma \Delta W_{k-1} \right) \circ \cdots \circ  F_{\tau} \left( \bm{x} + \sigma \Delta W_{1} \right) 
\end{equation*}
for $k\geq2$ with 
\begin{align*}
	\varphi_{\tau}(0, \omega, \bm{x}) := \bm{x}, \quad \varphi_{\tau}(1, \omega, \bm{x}) := F_{\tau} \left( \bm{x} + \sigma \Delta W_{1} \right). 
\end{align*}
We show later in Theorem \ref{Thm:discrete RDS} that $\varphi_{\tau}$ forms a discrete RDS by demonstrating its conjugacy relation. 
To this end, we introduce  
\begin{equation*}
	\tilde{Z}^*_{k}(\omega) = \int_{t_{k-1}}^{t_{k}} Z^*(\theta_s \omega) \rmd s, 
\end{equation*}
where $Z^*(\theta_s \omega)$ is given by \eqref{OU2}. It can be verified that 
\begin{equation}\label{OU3}
	Z^*(\theta_{t_{k+1}} \omega) - Z^*(\theta_{t_{k}} \omega) = - \gamma \tilde{Z}^*_{k+1}(\omega) + \Delta W_{k+1}. 
\end{equation}
Then defining the transformation $\hat{X}^{\tau}_{k} := X^{\tau}_{k}-\sigma Z^*(\theta_{t_{k}} \omega)$, we can conjugate \eqref{BEM} to a discrete random differential equation 
\begin{equation}\label{RDE2}
	\hat{X}^{\tau}_{k+1} =  \hat{X}^{\tau}_{k} + \tau F\left( \hat{X}^{\tau}_{k+1}+\sigma Z^*(\theta_{t_{k+1}}\omega) \right) + \gamma\sigma \tilde{Z}^*_{k+1}(\omega). 
\end{equation}
It follows from \eqref{BEM2} and \eqref{OU3} that 
\begin{align*}
	&\hat{X}^{\tau}_{k+1} = X^{\tau}_{k+1}-\sigma Z^*(\theta_{t_{k+1}} \omega)
	= F_{\tau} \left(  X^{\tau}_{k} + \sigma \Delta W_{k+1}  \right) - \sigma Z^*(\theta_{t_{k+1}}\omega) \\
	&= F_{\tau} \left( \hat{X}^{\tau}_{k} + \sigma Z^*(\theta_{t_{k}} \omega) + \sigma\Delta W_{k+1} \right) - \sigma Z^*(\theta_{t_{k+1}}\omega) \\
	&= F_{\tau} \left( \hat{X}^{\tau}_{k} + \sigma Z^*(\theta_{t_{k+1}}\omega) + \gamma\sigma\tilde{Z}^*_{k+1}(\omega) \right) - \sigma Z^*(\theta_{t_{k+1}}\omega). 
\end{align*}
This induces a mapping $\hat{\varphi}_{\tau}: \mathbb{Z} \times \Omega \times \mathbb{R}^2 \rightarrow \mathbb{R}^2$ defined by
\begin{equation}\label{varphitau}
	\begin{aligned}
		\hat{\varphi}_{\tau}(k, \omega, \bm{y}) &:= \left( F_\tau\left( \cdot + \sigma Z^*(\theta_{t_{k}}\omega)+ \gamma\sigma \tilde{Z}^*_{k}(\omega) \right) - \sigma Z^*(\theta_{t_{k}}\omega) \right)
		\circ \cdots \\
		&\qquad \circ \left( F_\tau\left( \bm{y} + \sigma Z^*(\theta_{t_{1}}\omega)+ \gamma\sigma \tilde{Z}^*_{1}(\omega) \right) - \sigma Z^*(\theta_{t_{1}}\omega) \right)
	\end{aligned}
\end{equation}
for $k\geq2$ with 
\begin{equation}\label{varphitau01}
	\hat{\varphi}_{\tau}(0, \omega, \bm{y}) := \bm{y}, \quad \hat{\varphi}_{\tau}(1, \omega, \bm{y}) := F_\tau\left( \bm{y} + \sigma Z^*(\theta_{t_{1}}\omega)+ \gamma\sigma \tilde{Z}^*_{1}(\omega) \right) - \sigma Z^*(\theta_{t_{1}}\omega). 
\end{equation}
The following theorem describes the generation of a discrete RDS. 
\begin{thm}\label{Thm:discrete RDS}
	Suppose that $\tau<\min\{\frac{1}{1+4|\alpha|},\frac{a}{1+|a\alpha|+|b\beta|}\}$. Then the following conjugacy relation holds for any $k\geq0$ and $\omega\in\Omega$:   
	\begin{equation}\label{hatphi and phi}
		\varphi_{\tau}(k, \omega, \bm{x}) = \mathcal{T}^{-1}(\theta_{t_k}\omega, \cdot) \circ \hat{\varphi}_{\tau}(k, \omega, \cdot) \circ \mathcal{T}(\omega,\bm{x}), 
	\end{equation}
	where $\mathcal{T}$ is defined by \eqref{mathcalT}. 
	In addition, the mapping $\varphi_{\tau}$ induced by \eqref{BEM} forms a discrete RDS over $(\Omega,\mathcal{F},\mathbb{P},\theta)$. 
\end{thm}
\begin{proof}
	We first prove the cocycle property of the mapping $\hat{\varphi}_{\tau}$. By observing 
	\begin{equation*}
		\tilde{Z}^*_{k+l}(\omega) = 	\int_{t_{k+l-1}}^{t_{k+l}} Z^*(\theta_s \omega) \rmd s
		= \int_{t_{k-1}}^{t_{k}} Z^*(\theta_{s+t_l} \omega) \rmd s
		= \tilde{Z}^*_{k}(\theta_{t_{l}}\omega), 
	\end{equation*}
	it can be verified that 
	\begin{align*}
		&\hat{\varphi}_{\tau}(k+l, \omega, \bm{y}) \\
		= \ & \left( F_\tau\left( \cdot + \sigma Z^*(\theta_{t_{k+l}}\omega)+ \gamma\sigma \tilde{Z}^*_{k+l}(\omega) \right) - \sigma Z^*(\theta_{t_{k+l}}\omega) \right)
		\circ \cdots \\
		& \circ \left( F_\tau\left( \cdot + \sigma Z^*(\theta_{t_{l+1}}\omega)+ \gamma\sigma \tilde{Z}^*_{l+1}(\omega) \right) - \sigma Z^*(\theta_{t_{l+1}}\omega) \right)   \\ 
		& \circ \left( F_\tau\left( \cdot + \sigma Z^*(\theta_{t_{l}}\omega)+ \gamma\sigma \tilde{Z}^*_{l}(\omega) \right) - \sigma Z^*(\theta_{t_{l}}\omega) \right) \circ \cdots  \\ 
		& \circ \left( F_\tau\left( \bm{y} + \sigma Z^*(\theta_{t_{1}}\omega)+ \gamma\sigma \tilde{Z}^*_{1}(\omega) \right) - \sigma Z^*(\theta_{t_{1}}\omega) \right) \\
		= \ & \left( F_\tau\left( \cdot + \sigma Z^*(\theta_{t_{k}}\theta_{t_{l}}\omega)+ \gamma\sigma \tilde{Z}^*_{k}(\theta_{t_{l}}\omega) \right) - \sigma Z^*(\theta_{t_{k}}\theta_{t_{l}}\omega) \right)
		\circ \cdots \\
		& \circ \left( F_\tau\left( \cdot + \sigma Z^*(\theta_{t_{1}}\theta_{t_{l}}\omega)+ \gamma\sigma \tilde{Z}^*_{1}(\theta_{t_{l}}\omega) \right) - \sigma Z^*(\theta_{t_{1}}\theta_{t_{l}}\omega) \right) \circ \hat{\varphi}_{\tau}(l, \omega, \bm{y})   \\ 
		= \ & \hat{\varphi}_{\tau}(k, \theta_{t_{l}}\omega, \cdot)  \circ \hat{\varphi}_{\tau}(l, \omega, \bm{y}), 
	\end{align*}
	which means that $\hat{\varphi}_{\tau}$ is a cocycle. 
	
	Next we prove \eqref{hatphi and phi} by induction on $k$. By the definition of $\mathcal{T}$, we have 
	\begin{equation}\label{conjugacy1}
		\mathcal{T}^{-1}(\theta_{t_k}\omega, \cdot) \circ \hat{\varphi}_{\tau}(k, \omega, \cdot) \circ \mathcal{T}(\omega,\bm{x}) =  \hat{\varphi}_{\tau}(k, \omega, \bm{x} - \sigma Z^*(\omega)) + \sigma Z^*(\theta_{t_k}\omega), \quad \forall \, k\geq0. 
	\end{equation}
	When $k=0$, \eqref{hatphi and phi} obviously holds. 
	We assume that \eqref{hatphi and phi} holds for $k=m$, i.e., 
	\begin{equation*}
		\varphi_{\tau}(m, \omega, \bm{x}) = \hat{\varphi}_{\tau}(m, \omega, \bm{x}-\sigma Z^*(\omega)) + \sigma Z^*(\theta_{t_m}\omega).  
	\end{equation*}
	Then for $k=m+1$, by the cocycle property of $\hat{\varphi}_{\tau}$, \eqref{conjugacy1}, \eqref{varphitau01}, and \eqref{OU3}, 
	\begin{align*}
		&\mathcal{T}^{-1}(\theta_{t_{m+1}}\omega, \cdot) \circ \hat{\varphi}_{\tau}(m+1, \omega, \cdot) \circ \mathcal{T}(\omega,\bm{x}) \\
		= \ &\hat{\varphi}_{\tau}(m+1, \omega, \bm{x} - \sigma Z^*(\omega)) + \sigma Z^*(\theta_{t_{m+1}}\omega) \\
		= \ &\hat{\varphi}_{\tau}(1, \theta_{t_m}\omega, \cdot) \circ \hat{\varphi}_{\tau}(m, \omega, \bm{x} - \sigma Z^*(\omega)) + \sigma Z^*(\theta_{t_{m+1}}\omega) \\
		= \ &\hat{\varphi}_{\tau}(1, \theta_{t_m}\omega, \cdot) \circ \left(\varphi_{\tau}(m, \omega, \bm{x})-\sigma Z^*(\theta_{t_{m}}\omega)\right) + \sigma Z^*(\theta_{t_{m+1}}\omega) \\
		= \ &F_\tau\left( \varphi_{\tau}(m, \omega, \bm{x})-\sigma Z^*(\theta_{t_{m}}\omega) + \sigma Z^*(\theta_{t_{m+1}}\omega)+ \gamma\sigma \tilde{Z}^*_{m+1}(\omega) \right) \\
		& - \sigma Z^*(\theta_{t_{1}}\theta_{t_{m}}\omega) + \sigma Z^*(\theta_{t_{m+1}}\omega)   \\
		= \ &F_\tau\left( \varphi_{\tau}(m, \omega, \bm{x}) + \sigma\Delta W_{m+1} \right) \\
		= \ &F_\tau\left( \cdot + \sigma\Delta W_{m+1} \right) \circ  \varphi_{\tau}(m, \omega, \bm{x})
		= \varphi_{\tau}(m+1, \omega, \bm{x}), 
	\end{align*}
	which means that \eqref{hatphi and phi} holds for $k=m+1$. This verifies \eqref{hatphi and phi} for all $k\geq0$. 
	
	Finally, it follows from \eqref{hatphi and phi}, the cocycle property of $\hat{\varphi}_{\tau}$, and $\mathcal{T}(\omega,\cdot)\circ \mathcal{T}^{-1}(\omega,\cdot) = \text{id}_{\mathbb{R}^2}(\cdot)$ for all $\omega\in\Omega$ that 
	\begin{align*}
		&\varphi_{\tau}(k+l, \omega, \bm{x}) 
		= \mathcal{T}^{-1}(\theta_{t_{k+l}}\omega, \cdot) \circ \hat{\varphi}_{\tau}(k+l, \omega, \cdot) \circ \mathcal{T}(\omega,\bm{x}) \\
		&= \mathcal{T}^{-1}(\theta_{t_{k+l}}\omega, \cdot) \circ \hat{\varphi}_{\tau}(k, \theta_{t_l}\omega, \cdot) \circ \left( \mathcal{T}(\theta_{t_{l}}\omega, \cdot) \circ \mathcal{T}^{-1}(\theta_{t_{l}}\omega, \cdot) \right) \circ \hat{\varphi}_{\tau}(l, \omega, \cdot) \circ \mathcal{T}(\omega,\bm{x}) \\ 
		&= \varphi_{\tau}(k, \theta_{t_l}\omega, \cdot) \circ \varphi_{\tau}(l, \omega, \bm{x}), 
	\end{align*}
	which means that $\varphi_{\tau}$ is a cocycle. Thus $\varphi_{\tau}$ induced by \eqref{BEM} forms a discrete RDS. 
\end{proof}

\subsection{Random attractor for the discrete RDS}
In this subsection we establish the existence of a random attractor for $\varphi_{\tau}$ via a pathwise analysis of the \textit{a priori} estimate of $\hat{X}^{\tau}_{k}$ determined by \eqref{RDE2}. 
\begin{thm}\label{Thm:discrete attractor}
	Under conditions of Theorem \ref{Thm:discrete RDS}, there exists a random attractor $A^{\tau}(\omega)$ for the discrete RDS $\varphi_{\tau}$, i.e., for every bounded deterministic set $B\subset \mathbb{R}^{2}$, 
	\begin{equation*}
		\lim\limits_{k\rightarrow\infty} d\left(\varphi_{\tau}(k,\theta_{-t_{k}}\omega,B) , A^{\tau}(\omega)\right) = 0, \quad \text{for all} \ \omega \in \Omega. 
	\end{equation*} 
\end{thm}
\begin{proof}
	We first give the \textit{a priori} estimate of $\hat{X}^{\tau}_{k}$. Recall
	$\hat{X}^{\tau}_{k} = X^{\tau}_{k}-\sigma Z^*(\theta_{t_{k}} \omega)$ and denote
	\begin{equation*}
		\hat{X}^{\tau}_{k} = \begin{bmatrix}
			\bar{x}^{\tau}_{k} \\ \bar{y}^{\tau}_{k}
		\end{bmatrix}, \quad 
		Z^*(\theta_{t_k}\omega) = Z^*_{k} = \begin{bmatrix}
			x^*_{k} \\ y^*_{k}
		\end{bmatrix}. 
	\end{equation*}
	Then from \eqref{RDE2} we obtain 
	\begin{align*}
		&\|\hat{X}^{\tau}_{k+1}\|^2 - \hat{X}^{\tau}_{k} \cdot \hat{X}^{\tau}_{k+1} \\ 
		= \ &\tau \Big(  \alpha \|\hat{X}^{\tau}_{k+1}\|^2 + \sigma x^*_{k+1}\left(\alpha\bar{x}^{\tau}_{k+1}+\beta \bar{y}^{\tau}_{k+1}\right)  
		+ \sigma y^*_{k+1}\left(\alpha\bar{y}^{\tau}_{k+1}-\beta \bar{x}^{\tau}_{k+1}\right)  \Big)
		+ \gamma\sigma \hat{X}^{\tau}_{k+1} \cdot \tilde{Z}^*_{k+1}(\omega)  \\
		&- \tau \| \hat{X}^{\tau}_{k+1}+\sigma Z^*_{k+1}\|^2 
		\left(  a\|\hat{X}^{\tau}_{k+1}\|^2 + \sigma x^*_{k+1} (a\bar{x}^{\tau}_{k+1}-b\bar{y}^{\tau}_{k+1}) + \sigma y^*_{k+1} (b\bar{x}^{\tau}_{k+1}+a\bar{y}^{\tau}_{k+1}) \right). 
	\end{align*}
	It follows from the inequality 
	\begin{equation*}
		\max\left\{ \alpha\bar{x}^{\tau}_{k+1}+\beta \bar{y}^{\tau}_{k+1} , \ 
		\alpha\bar{y}^{\tau}_{k+1}-\beta \bar{x}^{\tau}_{k+1} \right\}
		\leq \sqrt{(\alpha^2+\beta^2)}\|\hat{X}^{\tau}_{k+1}\|
	\end{equation*}
	that 
	\begin{align*}
		&\sigma x^*_{k+1}\left(\alpha\bar{x}_{k+1}+\beta \bar{y}_{k+1}\right)  
		+ \sigma y^*_{k+1}\left(\alpha\bar{y}_{k+1}-\beta \bar{x}_{k+1}\right) 
		+ \frac{\gamma\sigma}{\tau} \hat{X}^{\tau}_{k+1} \cdot \tilde{Z}^*_{k+1}(\omega)  \\
		\leq \ &\sigma \sqrt{(\alpha^2+\beta^2)} \|\hat{X}^{\tau}_{k+1}\| \left(|x^*_{k+1}| + |y^*_{k+1}|\right) + \frac{\gamma\sigma}{\tau} \|\hat{X}^{\tau}_{k+1}\| \|\tilde{Z}^*_{k+1}(\omega)\|  \\
		\leq \ &\sigma \sqrt{2(\alpha^2+\beta^2)}\|\hat{X}^{\tau}_{k+1}\| \|Z_{k+1}^*\| + \frac{\gamma\sigma}{\tau}\|\hat{X}^{\tau}_{k+1}\| \|\tilde{Z}^*_{k+1}(\omega)\|. 
	\end{align*} 
	On the other hand, from 
	\begin{equation*}
		\left|  \|\hat{X}^{\tau}_{k+1}+\sigma Z^*_{k+1}\|^2 - \|\hat{X}^{\tau}_{k+1}\|^2 - \sigma^2 \|Z_{k+1}^*\|^2 \right|
		\leq 2\sigma \|\hat{X}^{\tau}_{k+1}\| \|Z_{k+1}^*\|, 
	\end{equation*}
	we know that 
	\begin{equation*}
		a\|\hat{X}^{\tau}_{k+1}\|^2  \|\hat{X}^{\tau}_{k+1}+\sigma Z^*_{k+1}\|^2 
		\geq a\|\hat{X}^{\tau}_{k+1}\|^4 -2a\sigma \|Z_{k+1}^*\| \|\hat{X}^{\tau}_{k+1}\|^3  +  a\sigma^2 \|Z_{k+1}^*\|^2 \|\hat{X}^{\tau}_{k+1}\|^2. 
	\end{equation*}
	By using the following inequality
	\begin{align*}
		\left| x^*_{k+1} (a\bar{x}^{\tau}_{k+1}-b\bar{y}^{\tau}_{k+1}) +  y^*_{k+1} (b\bar{x}^{\tau}_{k+1}+a\bar{y}^{\tau}_{k+1}) \right| 
		\leq 2\sqrt{a^2+b^2} \|Z_{k+1}^*\| \|\hat{X}^{\tau}_{k+1}\|, 
	\end{align*}
	we derive that 
	\begin{align*}
		&\left| \sigma x^*_{k+1} (a\bar{x}^{\tau}_{k+1}-b\bar{y}^{\tau}_{k+1}) +  \sigma y^*_{k+1} (b\bar{x}^{\tau}_{k+1}+a\bar{y}^{\tau}_{k+1}) \right| \|\hat{X}^{\tau}_{k+1}+\sigma Z^*_{k+1}\|^2 \\
		\leq \ &2\sigma\sqrt{a^2+b^2} \|Z_{k+1}^*\| \|\hat{X}^{\tau}_{k+1}\| \left( \|\hat{X}^{\tau}_{k+1}\|^2 + \sigma^2 \|Z_{k+1}^*\|^2 + 2\sigma \|\hat{X}^{\tau}_{k+1}\|  \|Z_{k+1}^*\| \right). 
	\end{align*}
	Consequently, by Young's inequality we have 
	\begin{align*}
		\frac{1}{2}\|\hat{X}^{\tau}_{k+1}\|^2 &\leq \frac{1}{2}\|\hat{X}^{\tau}_{k}\|^2 + \tau \alpha \|\hat{X}^{\tau}_{k+1}\|^2 + \tau\sigma \sqrt{2(\alpha^2+\beta^2)}\|\hat{X}^{\tau}_{k+1}\| \|Z_{k+1}^*\| + \gamma\sigma\|\hat{X}^{\tau}_{k+1}\| \|\tilde{Z}^*_{k+1}(\omega)\| \\ 
		&\quad -\tau \left( a\|\hat{X}^{\tau}_{k+1}\|^4 -2a\sigma \|Z_{k+1}^*\| \|\hat{X}^{\tau}_{k+1}\|^3 +  a\sigma^2 \|Z_{k+1}^*\|^2 \|\hat{X}^{\tau}_{k+1}\|^2 \right) \\
		&\quad + \tau\sqrt{a^2+b^2} \left( 2\sigma\|Z_{k+1}^*\| \|\hat{X}^{\tau}_{k+1}\|^3 + 4\sigma^2 \|Z_{k+1}^*\|^2\|\hat{X}^{\tau}_{k+1}\|^2 + 2\sigma^3 \|Z_{k+1}^*\|^3 \|\hat{X}^{\tau}_{k+1}\| \right) \\ 
		&\leq \frac{1}{2}\|\hat{X}^{\tau}_{k}\|^2 - \frac{a\tau}{2} \|\hat{X}^{\tau}_{k+1}\|^2 + C\tau  \left(  \|Z_{k+1}^*\|^4 + \tau^{-\frac{4}{3}} \|\tilde{Z}^*_{k+1}(\omega)\|^{\frac{4}{3}} + 1  \right). 
	\end{align*}
	It follows that 
	\begin{align*}
		\|\hat{X}^{\tau}_{k+1}\|^2 &\leq \frac{1}{1+a\tau} \|\hat{X}^{\tau}_{k}\|^2 + \frac{C\tau}{1+a\tau} (\|Z_{k+1}^*\|^4 + 1 ) +  \frac{C\tau^{-\frac{1}{3}}}{1+a\tau} \|\tilde{Z}^*_{k+1}(\omega)\|^{\frac{4}{3}} \\
		&\leq \frac{\|\hat{X}^{\tau}_{0}\|^2}{(1+a\tau)^{k+1}}  + C\tau \sum_{i=1}^{k+1} \frac{\|Z_{i}^*\|^4 + 1}{(1+a\tau)^{k+2-i}} 
		+  C\tau^{-\frac{1}{3}} \sum_{i=1}^{k+1} \frac{\|\tilde{Z}_{i}^*(\omega)\|^\frac{4}{3}}{(1+a\tau)^{k+2-i}} \\
		&\leq \frac{\|\hat{X}^{\tau}_{0}\|^2}{(1+a\tau)^{k+1}}  + \frac{C}{a} +  C\tau \sum_{i=0}^{k} \frac{\|Z_{k+1-i}^*\|^4 + \tau^{-\frac{4}{3}} \|\tilde{Z}_{k+1-i}^*(\omega)\|^{\frac{4}{3}}}{(1+a\tau)^{i+1}}. 
	\end{align*}
	
	We next show that $\varphi_{\tau}$ has a random absorbing set, thereby concluding the existence of random attractor. 
	Denote 
	\begin{align*}
		\hat{\mathcal{A}}_{\tau}(\omega) =  \left(  1 + \frac{C}{a} + C\tau \sum_{i=0}^{\infty} \frac{\|Z_{k+1-i}^*\|^4 + \tau^{-\frac{4}{3}} \|\tilde{Z}_{k+1-i}^*(\omega)\|^{\frac{4}{3}}}{(1+a\tau)^{i+1}}  \right)^{\frac{1}{2}}. 
	\end{align*}
	Recalling the definition of $\hat{\varphi}_{\tau}$ (i.e., \eqref{varphitau}), it is evident that $\hat{\varphi}_{\tau}(k,\omega,\bm{y})$ is equal to $\hat{X}^{\tau}_k$ with the initial value  $\hat{X}^{\tau}_0=\bm{y} \in \mathbb{R}^2$. Then it follows that  
	\begin{align*}
		\|\hat{\varphi}_{\tau}(k,\omega,\bm{y})\|^2 
		\leq \left(\frac{1}{1+a\tau}\right)^{k} \|\bm{y}\|^2 + \hat{\mathcal{A}}_{\tau}(\omega), 
	\end{align*}
	which, together with \eqref{hatphi and phi}, means that 
	\begin{equation}\label{hatphi1}
		\|\varphi_{\tau}(k, \omega, \bm{x})\|^2 
		=  \left\|\hat{\varphi}_{\tau}(k,\omega,\bm{x}-\sigma Z^{*}(\omega)) + \sigma Z^*_{k} \right\|^2
		\leq \frac{2 \|\bm{x}-\sigma Z^*(\omega)\|^2}{(1+a\tau)^{k}} + 2\hat{\mathcal{A}}_{\tau}(\omega) + 2\sigma^2\|Z^*_{k}\|^2. 
	\end{equation}
	Substituting $\omega$ with $\theta_{-t_{k}} \omega$ into \eqref{hatphi1}, we have 
	\begin{equation*}
		\|\varphi_{\tau}(k, \theta_{-t_{k}} \omega, \bm{x})\|^2 
		\leq \frac{2\|\bm{x}-\sigma Z^*(\theta_{-t_{k}} \omega)\|^2}{(1+a\tau)^{k}}  + 2\hat{\mathcal{A}}_{\tau}(\theta_{-t_{k}}\omega) + 2\sigma^2\|Z^*(\omega)\|^2, 
	\end{equation*}
	where 
	\begin{equation*}
		\hat{\mathcal{A}}_{\tau}(\theta_{-t_{k}}\omega) =  \left(  1 + \frac{C}{a} +  C\tau \sum_{i=0}^{\infty} \frac{\|Z^*_{1-i}\|^4 + \tau^{-\frac{4}{3}} \|\tilde{Z}^*_{1-i}(\omega)\|^{\frac{4}{3}}}{(1+a\tau)^{i+1}}   \right)^{\frac{1}{2}}. 
	\end{equation*}
	Finally, define 
	\begin{equation*}
		\mathcal{A}_\tau(\omega) := 2\hat{\mathcal{A}}_{\tau}(\theta_{-t_{k}}\omega) + 2\sigma^2\|Z^*(\omega)\|^2 + 1, 
	\end{equation*}
	and 
	\begin{equation}\label{absorb set}
		K_{\tau}(\omega) := \left\{ \bm{x}\in\mathbb{R}^2: \|\bm{x}\|^2 \leq \mathcal{A}_\tau(\omega) \right\}. 
	\end{equation}
	By $\|Z^*(\omega)\|<\infty$ for all $\omega\in\Omega$ and \eqref{Z*}, we have 
	\begin{align*}
		|\mathcal{A}_\tau(\omega)| 
		&\leq 2\sigma^2\|Z^*(\omega)\|^2 + C\left( 1+\left(  \tau \sum_{i=0}^{\infty} \frac{\|Z^*_{1-i}\|^4 + \tau^{-\frac{4}{3}} \|\tilde{Z}^*_{1-i}(\omega)\|^{\frac{4}{3}}}{(1+a\tau)^{i+1}}   \right)^{\frac{1}{2}} \right) \\
		&\leq 2\sigma^2\|Z^*(\omega)\|^2 + C\left( 1+\left(  \tau \sum_{i=0}^{\infty} \frac{(L_{\delta,\omega}+\tilde{L}_{\delta,\omega}|t_{i-1}|^{\delta})^4 + 1}{(1+a\tau)^{i+1}}   \right)^{\frac{1}{2}} \right) 
		< \infty. 
	\end{align*}
	Consequently, $K_{\tau}$ defined by \eqref{absorb set} is a random absorbing set for $\varphi_{\tau}$. 
	By the result of \cite[Theorem B.2]{Doan2018}, there exists a random attractor $A^{\tau}(\omega)$ for $\varphi_{\tau}$. 
\end{proof}

\subsection{SRB measure for the discrete RDS} 
The sample measure defined by $\mu^{\tau}_{\omega}:=\delta_{A^{\tau}(\omega)}$ is invariant under $\varphi_{\tau}$, following directly from the invariance of the random attractor $A^{\tau}(\omega)$. 
Moreover, it is also an SRB measure under certain conditions, as shown in the following theorem. 

\begin{thm}\label{Thm:SRB} 
	There exist constants $b_{0}\in(0,\infty)$ and $\tau_{0}\in(0,1)$, such that the sample measure $\mu^{\tau}_{\omega}$ for the discrete RDS $\varphi_{\tau}$ is an SRB measure whenever $b>b_{0}$ and $\tau<\tau_{0}$. 
\end{thm}
\begin{remark}
	Since the sample measure associated with $\varphi_{\tau}$ is the SRB measure, it is also a physically relevant invariant measure (i.e., \eqref{physical measure} holds) and can thus be computed by simply flowing a large set of initial points for a fixed noise realization; this is exactly how figures in Section \ref{Sec:numerical example} are obtained. 
\end{remark}
The proof of Theorem \ref{Thm:SRB} relies on several auxiliary propositions. We will therefore defer the proof at the end of this subsection and first introduce these necessary propositions. 
Linearizing \eqref{BEM} along a trajectory $X^{\tau}_{k}$ gives the process $U^{\tau}_{k}\in\mathbb{R}^2$ satisfying 
\begin{equation}\label{numerical variational SDE}
	U^{\tau}_{k+1}  = U^{\tau}_{k} + \tau \left( \begin{bmatrix}
		\alpha & -\beta \\ \beta & \alpha 
	\end{bmatrix} U^{\tau}_{k+1} 
	+ \|X^{\tau}_{k+1}\|^2 \begin{bmatrix}
		-a & -b \\ b & -a 
	\end{bmatrix} U^{\tau}_{k+1} + 2 X^{\tau}_{k+1} \cdot U^{\tau}_{k+1}  \begin{bmatrix}
		-a & -b \\ b & -a 
	\end{bmatrix} X^{\tau}_{k+1} \right). 
\end{equation}
As shown in Appendix \ref{Appendix:solvable of IES}, $U^{\tau}_{k}$, $k\geq0$ are well-defined whenever $\tau<\frac{1}{1+|\alpha|}$. 
Similar to the continuous case, we introduce the polar coordinate $U^{\tau}_{k} = \|U^{\tau}_{k}\| (\cos \xi^{\tau}_k, \sin \xi^{\tau}_k)^{\top}$ and define the \textit{(top) numerical Lyapunov exponent} $\lambda^{\tau}$ as 
\begin{equation}\label{numerical lyp}
	\lambda^{\tau} := \lim\limits_{N \rightarrow \infty} \frac{1}{N\tau} \log \|U^{\tau}_{N}\|. 
\end{equation}
In order to demonstrate the positivity of $\lambda^{\tau}$, we need the ergodicity of the process $\{(X^{\tau}_k,\xi^{\tau}_k): k\geq0\}$. To this end, introduce the Lyapunov function $\bar{\Gamma}: \mathbb{R}^2\times\mathbb{R}/(\pi\mathbb{Z}) \rightarrow [1, \infty)$ defined by 
\begin{equation}\label{barGamma}
	\bar{\Gamma}(X,\xi) := \|X\|^2+\xi^2+M, 
\end{equation}
where $M\geq1$ is an arbitrary constant. Then we have the following proposition, whose proof is postponed to Section \ref{Sec:proof}. 
\begin{prop}\label{Prop:ergodic}
	Suppose that $b>4a$ and $\tau<\min\{\frac{a}{1+4b|\beta|},\frac{1}{1+12|\beta|},\frac{1}{1+4|\alpha|},\frac{4}{a}\}$. Then there exists a unique stationary measure $\nu^{\tau}$ for $\{(X^{\tau}_k,\xi^{\tau}_k): k\geq0\}$, and a  unique stationary measure $\tilde{\nu}^{\tau}$ for $\{X^{\tau}_k: k\geq0\}$. 
	Furthermore, there exist $\bar{\iota}\in(0,1)$ and $\bar{\kappa}\in(0,\infty)$, such that 
	\begin{equation}\label{ergodic3}
		\left| \mathbb{E}^{X^{\tau}_{0},\xi^{\tau}_{0}}[\hat{Q}(X^{\tau}_k,\xi^{\tau}_k)] - \nu^{\tau}(\hat{Q}) \right| 
		\leq \bar{\kappa} \bar{\iota}^{k\tau} \bar{\Gamma}(X^{\tau}_{0},\xi^{\tau}_{0}), 
	\end{equation}
	where $\hat{Q}$ is given by \eqref{hatQ}. 
\end{prop}

In the following proposition we present a critical inequality linking $\lambda^{\tau}$ to $\lambda$, which is used to ensure the positivity of $\lambda^{\tau}$. 
\begin{prop}\label{Prop:lambda and lambda tau}
	Suppose that $b>4a$ and $\tau<\min\{\frac{a}{1+4b|\beta|},\frac{1}{1+12|\beta|},\frac{1}{1+4|\alpha|},\frac{4}{a}\}$. Then there exist $\iota,\bar{\iota}\in(0,1)$ and $\kappa,\bar{\kappa}\in(0,\infty)$, such that 
	\begin{equation}\label{lambda tau}
		\begin{aligned}
			\lambda &\leq \lambda^{\tau}+  \left| \mathbb{E}^{X^{\tau}_{0},\xi^{\tau}_{0}}[\hat{Q}(X^{\tau}_{k},\xi^{\tau}_{k})] - \mathbb{E}^{X_{0},\xi_{0}}[\hat{Q}(X_{t_k},\xi_{t_k})] \right| \\
			&\quad +\kappa \iota^{t_k} \Gamma(\gamma_{0},\psi_{0}) + \bar{\kappa} \bar{\iota}^{k\tau} \bar{\Gamma}(X^{\tau}_{0},\xi^{\tau}_{0}) + \frac{\tau}{2} \tilde{\nu}^{\tau}(\mathcal{H}), 
		\end{aligned}
	\end{equation}
	where the function $\mathcal{H}:\mathbb{R}^2\rightarrow\mathbb{R}$ is given by $\mathcal{H}(Y)= 3(\alpha^2+\beta^2)+15(a^2+b^2)\|Y\|^4$. 
\end{prop}
\begin{proof}
%	Similar to the continuous case, we introduce the polar coordinates $U^{\tau}_{k} = \|U^{\tau}_{k}\| \begin{bmatrix}
%		\cos \xi^{\tau}_k \\[.2em] \sin \xi^{\tau}_k
%	\end{bmatrix}$, 
	By Taylor's expansion at $U^{\tau}_{k+1}$, we have 
	\begin{align*}
		\|U^{\tau}_{k}\|^2 &= \|U^{\tau}_{k+1}\|^2 + 2U^{\tau}_{k+1}\cdot \left(U^{\tau}_{k}-U^{\tau}_{k+1}\right) + \|U^{\tau}_{k}-U^{\tau}_{k+1}\|^2. 
	\end{align*}
	It follows from \eqref{numerical variational SDE} that  
	\begin{equation*}
		\begin{aligned}
			\|U^{\tau}_{k}-U^{\tau}_{k+1}\|^2
			%	&\leq 3 \tau^2  (\alpha u^{\tau}_{k+1} - \beta v^{\tau}_{k+1})^2 + 3\tau^2  ((x^{\tau}_{k+1})^2+(y^{\tau}_{k+1})^2)^2(au^{\tau}_{k+1}+bv^{\tau}_{k+1})^2    \\
			%	&\quad + 12\tau^2 (ax^{\tau}_{k+1}+by^{\tau}_{k+1})^2(x^{\tau}_{k+1}u^{\tau}_{k+1} + y^{\tau}_{k+1}v^{\tau}_{k+1})^2 \\
			%	&\quad + 3 \tau^2 (\beta u^{\tau}_{k+1} + \alpha v^{\tau}_{k+1})^2 +3 \tau^2 ((x^{\tau}_{k+1})^2+(y^{\tau}_{k+1})^2)^2(bu^{\tau}_{k+1}-av^{\tau}_{k+1})^2 \\[.2em]
			%	&\quad + 12\tau^2 (bx^{\tau}_{k+1}-ay^{\tau}_{k+1})^2(x^{\tau}_{k+1}u^{\tau}_{k+1}+y^{\tau}_{k+1}v^{\tau}_{k+1})^2 \\
			%	&\leq 3\tau^2 (\beta^2+\alpha^2)\|U^{\tau}_{k+1}\|^2 +  3\tau^2(a^2+b^2)  \|X^{\tau}_{k+1}\|^4 \|U^{\tau}_{k+1}\|^2  
			%	+12\tau^2 (a^2+b^2)\|X^{\tau}_{k+1}\|^4 \|U^{\tau}_{k+1}\|^2    \\
%			&\leq  \tau^2 \|U^{\tau}_{k+1}\|^2 \left(  3(\alpha^2+\beta^2)+15(a^2+b^2)\|X^{\tau}_{k+1}\|^4  \right)  \\
			\leq \tau^2 \|U^{\tau}_{k+1}\|^2 \mathcal{H}(X^{\tau}_{k+1}) 
		\end{aligned}
	\end{equation*}
	and 
	\begin{equation*}
		U^{\tau}_{k+1} \cdot \left(U^{\tau}_{k+1}-U^{\tau}_{k}\right) 
		%	&= \tau \|U^{\tau}_{k+1}\|^2 \Big(  \alpha - 2a\|X^{\tau}_{k+1}\|^2 + \left((x^{\tau}_{k+1})^2-(y^{\tau}_{k+1})^2\right)(b\sin2\xi^{\tau}_{k+1}-a\cos2\xi^{\tau}_{k+1}) \\
		%	&\qquad \qquad - 2x^{\tau}_{k+1}y^{\tau}_{k+1}(a\sin2\xi^{\tau}_{k+1}+b\cos2\xi^{\tau}_{k+1}) \Big)\\
		= \tau \|U^{\tau}_{k+1}\|^2 \hat{Q}(X^{\tau}_{k+1},\xi^{\tau}_{k+1}), 
	\end{equation*}
	where $\hat{Q}$ is given by \eqref{hatQ}. 
	Hence we obtain 
	\begin{align*}
		\|U^{\tau}_{k}\|^2 
		&= \|U^{\tau}_{k+1}\|^2
		- 2\tau \|U^{\tau}_{k+1}\|^2 \hat{Q}(X^{\tau}_{k+1},\xi^{\tau}_{k+1})
		+ \|U^{\tau}_{k}-U^{\tau}_{k+1}\|^2 \\
		&\leq \|U^{\tau}_{k+1}\|^2 \left( 1 - 2\tau \hat{Q}(X^{\tau}_{k+1},\xi^{\tau}_{k+1}) + \tau^2 \mathcal{H}(X^{\tau}_{k+1}) \right), 
	\end{align*}
	which, together with $\log (1+x) \leq x$ for $x>-1$, means that 
	\begin{equation*}
		\log \|U^{\tau}_{k}\|^2 
		\leq \log \|U^{\tau}_{k+1}\|^2 - 2\tau \hat{Q}(X^{\tau}_{k+1},\xi^{\tau}_{k+1}) +  \tau^2 \mathcal{H}(X^{\tau}_{k+1}). 
	\end{equation*}
	Summing up $k$ from $0$ to $N-1$ in the above estimate yields 
	\begin{equation}\label{logUtau}
		\frac{1}{N} \log \|U^{\tau}_{0}\| \leq \frac{1}{N} \log \|U^{\tau}_{N}\|
		- \frac{\tau}{N}\sum_{k=1}^{N}  \hat{Q}(X^{\tau}_{k},\xi^{\tau}_{k}) + \frac{\tau^2}{2N}\sum_{k=1}^{N}\mathcal{H}(X^{\tau}_{k}). 
	\end{equation}
	Combining the definition of $\lambda^{\tau}$ and \eqref{logUtau}, we can deduce that 
	\begin{equation*}
		\lambda^{\tau} = \lim\limits_{N \rightarrow \infty} \frac{1}{N\tau} \log \|U^{\tau}_{N}\| \geq \lim\limits_{N \rightarrow \infty} \frac{1}{N}\sum_{k=1}^{N}  \hat{Q}(X^{\tau}_{k},\xi^{\tau}_{k}) - \lim\limits_{N \rightarrow \infty} \frac{\tau}{2N}\sum_{k=1}^{N}\mathcal{H}(X^{\tau}_{k}). 
	\end{equation*}
	In view of Proposition \ref{Prop:ergodic}, by Birkhoff's almost-sure ergodic theorem we have 
	\begin{align*}
		\lambda^{\tau} &\geq \nu^{\tau}(\hat{Q}) - \frac{\tau}{2} \tilde{\nu}^{\tau}(\mathcal{H}). 
	\end{align*} 
	On the other hand, by recalling $\lambda = \nu(Q)$ from \eqref{FK formula}, it follows from \eqref{ergodic2} and \eqref{ergodic3} that 
	\begin{align*}
		\lambda^{\tau} &\geq \nu^{\tau}(\hat{Q}) - \frac{\tau}{2} \tilde{\nu}^{\tau}(\mathcal{H}) \\
		&\geq \mathbb{E}^{X^{\tau}_{0},\xi^{\tau}_{0}}[\hat{Q}(X^{\tau}_{k},\xi^{\tau}_{k})] - \bar{\kappa} \bar{\iota}^{k\tau} \bar{\Gamma}(X^{\tau}_{0},\xi^{\tau}_{0}) - \frac{\tau}{2} \tilde{\nu}^{\tau}(\mathcal{H}) \\ 
		&\geq \mathbb{E}^{X^{\tau}_{0},\xi^{\tau}_{0}}[\hat{Q}(X^{\tau}_{k},\xi^{\tau}_{k})] - \bar{\kappa} \bar{\iota}^{k\tau} \bar{\Gamma}(X^{\tau}_{0},\xi^{\tau}_{0}) \\
		&\quad - \mathbb{E}^{X_{0},\xi_{0}}[\hat{Q}(X_{t_k},\xi_{t_k})] + \lambda - \kappa \iota^{t_k} \Gamma(\gamma_{0},\psi_{0}) - \frac{\tau}{2} \tilde{\nu}^{\tau}(\mathcal{H}), 
	\end{align*}
	which implies \eqref{lambda tau} and completes the proof. 
\end{proof}

We denote $C_{t} = (c_{t}, s_{t})^{\top} = (\cos\xi_t,\sin\xi_t)^{\top}$. By \eqref{xi} we can calculate that 
\begin{equation*}
	\rmd C_{t} = G(X_{t},C_{t}) \rmd t, 
\end{equation*}
where $G:\mathbb{R}^4\rightarrow\mathbb{R}^2$ is given by $G(x,y,c,s)=(G_{1}(x,y,c,s),G_{2}(x,y,c,s))^{\top}$ with 
\begin{equation}\label{G}
	\begin{aligned}
		G_{1} &= - \beta s-2bs(x^2+y^2) - (2ac s ^2+bc ^2s-bs^3) (x^2-y^2) - (4bc s^2 -2ac^2s +2as^3) xy  \\
		G_{2} &= \beta c+2bc(x^2+y^2) + (2ac^2 s+bc^3-bcs^2) (x^2-y^2) + (4bc^2 s -2ac^3 +2acs^2) xy. 
	\end{aligned}
\end{equation}
In addition, we denote $C^{\tau}_{k} = (c^{\tau}_k, s^{\tau}_k)^{\top}$, which is computed in Appendix \ref{Appendix:ck and sk} (i.e.,\eqref{ck} and \eqref{sk}). 
According to Proposition \ref{Prop:lambda and lambda tau}, to achieve the positivity of $\lambda^{\tau}$, we additionally need the strong convergence of the backward Euler method in finite time interval $[0,T]$ for given $T=N\tau$ with $N\in\mathbb{N}$. We denote by $t_{k}=k\tau$, $k=0,1,2,...,N$ the gridpoints. 
Then we have the following strong convergence result, whose proof is given in Section \ref{Sec:proof}. 
\begin{prop}\label{Prop:strong convergence}
	For any $p\geq2$, it holds that 
	\begin{equation}\label{strong convergence}
		\lim\limits_{\tau\rightarrow0} \sup\limits_{0\leq k \leq N}\mathbb{E}\left[ \|X_{t_{k}}-X^{\tau}_{k}\|^{p} + \|C_{t_{k}}-C^{\tau}_{k}\|^{p} \right] = 0. 
	\end{equation}
\end{prop}

Propositions \ref{Prop:lambda and lambda tau} and \ref{Prop:strong convergence} provide an approach to guarantee the positivity of $\lambda^{\tau}$, which
%constitutes a key result of this paper. More precisely, combining the strong convergence of backward Euler method, and the result of $\lambda>0$ for large shear strength $b$ shown in Proposition \ref{Prop:positive lyp}, the inequality \eqref{lambda tau} ensures the positive numerical Lyapunov exponent, i.e., $\lambda^{\tau}>0$; see the proof of Theorem \ref{Thm:SRB} below. 
%The positivity of $\lambda^{\tau}$ would 
in turn imply the existence of an SRB measure for the discrete RDS. 
With these results in hand, we are now in a position to prove Theorem \ref{Thm:SRB}. 

\begin{proof}[Proof of Theorem \ref{Thm:SRB}]
	By Proposition \ref{Prop:positive lyp} we know that there exists a constant $b_{0}>0$, such that $\lambda>0$ whenever $b>b_{0}$.  
	Since $\iota, \bar{\iota}\in(0,1)$, for such $\lambda>0$, there exists an integer $N>0$ such that $\kappa \iota^{T} \Gamma(\gamma_{0},\psi_{0}) + \bar{\kappa} \bar{\iota}^{T} \bar{\Gamma}(X^{\tau}_{0},\xi^{\tau}_{0})<\frac{\lambda}{4}$ with $T=N\tau$. 
	By Proposition \ref{Prop:strong convergence}, for fixed constants $b\in(b_0,\infty)$ and $T>0$, there exists a constant $\tau_{0}\in(0,1)$ such that when $\tau<\tau_{0}$, 
	\begin{equation*}
		\sup\limits_{0\leq k \leq N} \left| \mathbb{E}^{X^{\tau}_{0},\xi^{\tau}_{0}}[\hat{Q}(X^{\tau}_{k},\xi^{\tau}_{k})] - \mathbb{E}^{X_{0},\xi_{0}}[\hat{Q}(X_{t_k},\xi_{t_k})] \right| < \frac{\lambda}{4}, 
	\end{equation*}
	and $\frac{\tau}{2} \tilde{\nu}^{\tau}(\mathcal{H})<\frac{\lambda}{4}$, where $\mathcal{H}$ are given in Proposition \ref{Prop:lambda and lambda tau}. 
	Hence according to the inequality \eqref{lambda tau}, we can deduce that 
	\begin{align*}
		\lambda &\leq \lambda^{\tau} +  \left| \mathbb{E}^{X^{\tau}_{0},\xi^{\tau}_{0}}[\hat{Q}(X^{\tau}_{N},\xi^{\tau}_{N})] - \mathbb{E}^{X_{0},\xi_{0}}[\hat{Q}(X_{T},\xi_{T})] \right| \\
		&\quad +\kappa \iota^{T} \Gamma(\gamma_{0},\psi_{0}) + \bar{\kappa} \bar{\iota}^{T} \bar{\Gamma}(X^{\tau}_{0},\xi^{\tau}_{0}) + \frac{\tau}{2} \tilde{\nu}^{\tau}(\mathcal{H}) \\
		&< \lambda^{\tau} +\frac{3\lambda}{4}, 
	\end{align*}
	which ensures the positivity of numerical Lyapunov exponent, i.e., $\lambda^{\tau} > \frac{\lambda}{4} >0$. 
	It follows from \eqref{BEM} that 
	\begin{equation*}
		\mathbb{P}\left( X^{\tau}_{k+1}\in B | X^{\tau}_{k}=x \right) \leq \mathbb{P}\left( x+\sigma\Delta W_{k+1}\in G(B) \right), 
	\end{equation*}
	where $B\in\mathcal{B}(\mathbb{R}^2)$ and the function $G=\text{id}_{\mathbb{R}^2}-\tau F$. 
	This implies that the transition probability of $X^{\tau}_{k}$ is absolutely continuous with respect to the Lebesgue measure on $\mathbb{R}^2$ since the distribution function of $\Delta W_{k+1}$ is absolutely continuous. 
	In addition, as shown in Theorem \ref{Thm:discrete attractor}, $\varphi_{\tau}$ has a random attractor. 
	Consequently, by using the result of \cite[Theorem B]{Ledrappier1988}, the sample measure of $\varphi_{\tau}$ is an SRB measure. 
\end{proof}

\section{Proof of propositions in Section \ref{Sec:discrete RDS}}\label{Sec:proof}
In this section we present the proof of Propositions \ref{Prop:ergodic} and \ref{Prop:strong convergence} in Section \ref{Sec:discrete RDS}. 
\subsection{Proof of Proposition \ref{Prop:ergodic}}
The proof relies on the verification of the Lyapunov condition and the minorization condition. 
\begin{proof}[Proof of Proposition \ref{Prop:ergodic}]
	We first give the expression of $\xi^{\tau}_{k}$. To facilitate this, we denote $c^{\tau}_{k}=\cos \xi^{\tau}_{k}$ and $s^{\tau}_{k}=\sin \xi^{\tau}_{k}$, which are determined by \eqref{ck} and \eqref{sk}, respectively. Then 
	\begin{align*}
		&\sin(\xi^{\tau}_{k} - \xi^{\tau}_{k+1})  
		=  s^{\tau}_{k+1}(c^{\tau}_{k+1}-c^{\tau}_{k}) - c^{\tau}_{k+1}(s^{\tau}_{k+1}-s^{\tau}_{k}) \\
		&= - \tau \Big( \beta+2b\|X^{\tau}_{k+1}\|^2 + \left(a\sin2\xi^{\tau}_{k+1}+b\cos2\xi^{\tau}_{k+1}\right) \left(|x^{\tau}_{k+1}|^2-|y^{\tau}_{k+1}|^2\right) \\
		&\quad + (2b\sin2\xi^{\tau}_{k+1}-2a\cos2\xi^{\tau}_{k+1}) x^{\tau}_{k+1} y^{\tau}_{k+1}  \Big) +s^{\tau}_{k+1}\mathcal{R}_{k+1}^{c}-c^{\tau}_{k+1}\mathcal{R}_{k+1}^{s}, 
	\end{align*} 
	where $\mathcal{R}_{k+1}^{c}$ and $\mathcal{R}_{k+1}^{s}$ are given by \eqref{Rkc} and \eqref{Rks}. 
	This means that $\xi^{\tau}_{k}$ is determined by 
	\begin{equation}\label{numerical xi}
		\begin{aligned}
			\xi^{\tau}_{k+1} &= \xi^{\tau}_{k} + \arcsin \Big\{ \tau \Big( \beta+2b\|X^{\tau}_{k+1}\|^2 + \left(a\sin2\xi^{\tau}_{k+1}+b\cos2\xi^{\tau}_{k+1}\right) \left(|x^{\tau}_{k+1}|^2-|y^{\tau}_{k+1}|^2\right) \\
			&\quad + (2b\sin2\xi^{\tau}_{k+1}-2a\cos2\xi^{\tau}_{k+1}) x^{\tau}_{k+1} y^{\tau}_{k+1}  \Big) -s^{\tau}_{k+1}\mathcal{R}_{k+1}^{c}+c^{\tau}_{k+1}\mathcal{R}_{k+1}^{s} \Big\}. 
		\end{aligned}
	\end{equation}
	It follows from $\left|\arcsin(x)\right| \leq 2|x|$ for $x \in [-1, 1]$ and \eqref{numerical xi} that 
	\begin{equation}\label{estimate xi}
		\begin{aligned}
			|\xi^{\tau}_{k+1} - \xi^{\tau}_{k}| 
			&\leq 2\tau \Big|   \beta+2b\|X^{\tau}_{k+1}\|^2 + \left(a\sin2\xi^{\tau}_{k+1}+b\cos2\xi^{\tau}_{k+1}\right) \left(|x^{\tau}_{k+1}|^2-|y^{\tau}_{k+1}|^2\right) \\
			&\quad + (2b\sin2\xi^{\tau}_{k+1}-2a\cos2\xi^{\tau}_{k+1}) x^{\tau}_{k+1} y^{\tau}_{k+1} \Big|  + 2\left|s^{\tau}_{k+1}\mathcal{R}_{k+1}^{c}-c^{\tau}_{k+1}\mathcal{R}_{k+1}^{s}\right|. 
		\end{aligned}
	\end{equation} 
	%Introduce the Lyapunov function $\bar{\Gamma}: \mathbb{R}^2\times\mathbb{R}/(\pi\mathbb{Z}) \rightarrow [1, \infty)$ defined by 
	%\begin{equation}\label{barGamma}
	%	\bar{\Gamma}(X,\xi) := \|X\|^2+\xi^2+M
	%\end{equation}
	%for any $M\geq1$, 
	
	We next verify the Lyapunov condition with the Lyapunov function chosen as $\bar{\Gamma}^{p}$ for $p\geq1$, where $\bar{\Gamma}$ is defined by \eqref{barGamma}. 
	Combining \eqref{BEM} and \eqref{numerical xi} yields 
	\begin{equation}\label{discrete system}
		\begin{cases}
			X^{\tau}_{k} + \sigma \Delta W_{k+1} = 	X^{\tau}_{k+1} -  \tau \left( \begin{bmatrix}
				\alpha & -\beta \\ \beta & \alpha 
			\end{bmatrix} X^{\tau}_{k+1} 
			+ \|X^{\tau}_{k+1}\|^2 \begin{bmatrix}
				-a & -b \\ b & -a 
			\end{bmatrix} X^{\tau}_{k+1} \right),	\\[1em]
			\xi^{\tau}_{k} = \xi^{\tau}_{k+1} - \arcsin\Big\{ \tau \Big( \beta+2b\|X^{\tau}_{k+1}\|^2 + (2b\sin2\xi^{\tau}_{k+1}-2a\cos2\xi^{\tau}_{k+1}) x^{\tau}_{k+1} y^{\tau}_{k+1}  \\
			\qquad + \left(a\sin2\xi^{\tau}_{k+1}+b\cos2\xi^{\tau}_{k+1}\right) \left(|x^{\tau}_{k+1}|^2-|y^{\tau}_{k+1}|^2\right)   \Big) +s^{\tau}_{k+1}\mathcal{R}_{k+1}^{c}-c^{\tau}_{k+1}\mathcal{R}_{k+1}^{s} \Big\}. 
		\end{cases}
	\end{equation}
	Applying the function $\bar{\Gamma}^{p}$ to \eqref{discrete system}, then by Taylor's expansion, $|\xi^{\tau}_{k}|\leq\pi$, and \eqref{estimate xi}, we have 
	\begin{align*}
		&\bar{\Gamma}^{p}(X^{\tau}_{k} + \sigma \Delta W_{k+1},\xi^{\tau}_{k}) \\
		%	&\geq \bar{\Gamma}^{p}(x^{\tau}_{k+1}  , y^{\tau}_{k+1} ,\xi^{\tau}_{k+1}) + p\bar{\Gamma}^{p-1}(x^{\tau}_{k+1}  , y^{\tau}_{k+1} ,\xi^{\tau}_{k+1}) \Big\{ 2x^{\tau}_{k+1} (x^{\tau}_{k}+\sigma \Delta W_{k+1}^1-x^{\tau}_{k+1}) \\
		%	&\quad + 2y^{\tau}_{k+1} (y^{\tau}_{k}+\sigma \Delta W_{k+1}^2-y^{\tau}_{k+1}) + 2\xi^{\tau}_{k+1} (\xi^{\tau}_{k}-\xi^{\tau}_{k+1})  \Big\}  \\
		%	&\geq \bar{\Gamma}^{p}(x^{\tau}_{k+1}  , y^{\tau}_{k+1} ,\xi^{\tau}_{k+1}) + p\bar{\Gamma}^{p-1}(x^{\tau}_{k+1}  , y^{\tau}_{k+1} ,\xi^{\tau}_{k+1})  \left( - 2\alpha\tau\|X^{\tau}_{k+1}\|^2 
		%	+ 2a\tau\|X^{\tau}_{k+1}\|^4 
		%	-2\pi |\xi^{\tau}_{k+1}-\xi^{\tau}_{k}| \right)  \\
		\geq \ &  \bar{\Gamma}^{p}(X^{\tau}_{k+1},\xi^{\tau}_{k+1}) + p\bar{\Gamma}^{p-1}(X^{\tau}_{k+1},\xi^{\tau}_{k+1}) \Big( -2\alpha\tau\|X^{\tau}_{k+1}\|^2 + 2a\tau\|X^{\tau}_{k+1}\|^4  \\
		& -4\pi\beta\tau - 8\pi b \tau\|X^{\tau}_{k+1}\|^2 - 8\pi\tau(a+b)\|X^{\tau}_{k+1}\|^2
		- 4\pi (|\mathcal{R}_{k+1}^{c}|+|\mathcal{R}_{k+1}^{s}|) \Big) \\
		=: \ &  \bar{\Gamma}^{p}(X^{\tau}_{k+1},\xi^{\tau}_{k+1}) + \tilde{\Gamma}_{k+1}, 
	\end{align*}
	where, by Young's inequality, \eqref{estimate R}, and $|\xi^{\tau}_{k}|\leq\pi$, 
	\begin{align*}
		\tilde{\Gamma}_{k+1} 
		&\geq  p\bar{\Gamma}^{p-1}(X^{\tau}_{k+1},\xi^{\tau}_{k+1}) \left(  a\tau\|X^{\tau}_{k+1}\|^4 - \tau C(\alpha,\beta,a,b,\sigma)  \right)   \\ 
		&\geq  p\bar{\Gamma}^{p-1}(X^{\tau}_{k+1},\xi^{\tau}_{k+1}) \left(  \frac{a\tau}{2}\bar{\Gamma}^{2}(X^{\tau}_{k+1},\xi^{\tau}_{k+1}) - a\tau\left(|\xi^{\tau}_{k+1}|^2+M\right)^2 - \tau C(\alpha,\beta,a, b,\sigma)  \right)   \\ 
		%	&\geq  \frac{ap\tau}{2}\bar{\Gamma}^{p}(x^{\tau}_{k+1}  , y^{\tau}_{k+1} ,\xi^{\tau}_{k+1})  - \tau C(p, a, b, \alpha, \beta, \sigma) \bar{\Gamma}^{p-1}(x^{\tau}_{k+1}  , y^{\tau}_{k+1} ,\xi^{\tau}_{k+1})     \\ 
		&\geq  \frac{a\tau}{2}\bar{\Gamma}^{p}(X^{\tau}_{k+1},\xi^{\tau}_{k+1})  - \tau C(p,\alpha,\beta,a,b,\sigma). 
	\end{align*}
	Consequently, we obtain 
	\begin{equation}\label{Lyapunov condition 1}
		\left(1+\frac{a\tau}{2}\right) \bar{\Gamma}^{p}(X^{\tau}_{k+1},\xi^{\tau}_{k+1}) \leq  \bar{\Gamma}^{p}(X^{\tau}_{k} + \sigma \Delta W_{k+1},\xi^{\tau}_{k}) + \tau C(p,\alpha,\beta,a, b,\sigma) 
	\end{equation}
	for some deterministic positive real numbers $C$ uniformly with respect to $\tau$ and $k$. 
	
	By binomial expansion, $\mathbb{E}[\Delta W_{k+1}]=0$, and $\mathbb{E}[\|\Delta W_{k+1}\|^2]=2\tau$, for deterministic $Y\in\mathbb{R}^2$ and $\zeta \in \mathbb{R}$, we have that for $p\geq2$, 
	\begin{align*}
		&\mathbb{E}\left[\bar{\Gamma}^p(Y+\sigma\Delta W_{k+1},\zeta)\right] \\
		%	&= \mathbb{E} \left[  (x^2+y^2+\zeta^2+1+2\sigma x\Delta W_{k+1}^1+2\sigma y\Delta W_{k+1}^2+\sigma^2(\Delta W_{k+1}^1)^2+\sigma^2(\Delta W_{k+1}^2)^2)^p  \right] \\
		%	&= \mathbb{E} \left[ \sum_{i=0}^{p} \binom{p}{i} (x^2+y^2+\zeta^2+1)^{p-i} (2\sigma x\Delta W_{k+1}^1+2\sigma y\Delta W_{k+1}^2+\sigma^2(\Delta W_{k+1}^1)^2+\sigma^2(\Delta W_{k+1}^2)^2)^{i}  \right] \\
		= \ &\mathbb{E} \left[\bar{\Gamma}^{p}(Y,\zeta)\right] 
		+ p \mathbb{E} \left[\bar{\Gamma}^{p-1}(Y,\zeta) \left(2\sigma Y\cdot\Delta W_{k+1}+\sigma^2\|\Delta W_{k+1}\|^2\right)\right] \\ 
		&\quad + \mathbb{E} \left[ \sum_{i=2}^{p} \binom{p}{i} \bar{\Gamma}^{p-i}(Y,\zeta) \left(2\sigma Y\cdot\Delta W_{k+1}+\sigma^2\|\Delta W_{k+1}\|^2\right)^{i}  \right] \\
		\leq \ & \bar{\Gamma}^{p}(Y,\zeta) 
		+ 2\sigma^2p\tau  \bar{\Gamma}^{p-1}(Y,\zeta)  \\ 
		& + \sum_{i=2}^{p} \binom{p}{i} \bar{\Gamma}^{p-i}(Y,\zeta) 2^{i-1} \left(|2\sigma|^i \mathbb{E}\left[|Y\cdot\Delta W_{k+1}|^i\right]+|\sigma|^{2i} \mathbb{E}\left[\|\Delta W_{k+1}\|^{2i}\right] \right). 
	\end{align*}
	It follows from Cauchy--Schwarz's inequality, Young's inequality, and the general $p$th moment of Gaussian random variables that  
	\begin{equation*}
		\mathbb{E}\left[|Y\cdot\Delta W_{k+1}|^i\right] 
		\leq \epsilon\tau \|Y\|^{2i} + \frac{1}{4\epsilon\tau} \mathbb{E}\left[\|\Delta W_{k+1}\|^{2i}\right] 
		\leq \epsilon\tau \bar{\Gamma}^{i}(Y,\zeta) + \frac{2^{i}(2i-1)!!}{4\epsilon\tau} \tau^{i}, 
	\end{equation*}
	which leads to 
	\begin{align*}
		&\mathbb{E}\left[\bar{\Gamma}^p(Y+\sigma\Delta W_{k+1},\zeta)\right] \\
		\leq \ & \bar{\Gamma}^{p}(Y,\zeta) 
		+ 2\sigma^2p\tau  \bar{\Gamma}^{p-1}(Y,\zeta) 
		+ \sum_{i=2}^{p} \binom{p}{i} 2^{2i-1}\sigma^i \left( \epsilon\tau\bar{\Gamma}^{p}(Y,\zeta) + \frac{(2\tau)^{i-1}(2i-1)!!}{2\epsilon}  \bar{\Gamma}^{p-i}(Y,\zeta) \right) \\ 
		&+ \sum_{i=2}^{p} \binom{p}{i} 2^{2i-1} \sigma^{2i}(2i-1)!! \tau^{i} \bar{\Gamma}^{p-i}(Y,\zeta) \\ 
		\leq \ & (1+\epsilon\tau) \bar{\Gamma}^{p}(Y,\zeta) + C(p, \epsilon, \sigma)\tau. 
	\end{align*}
	Since $X^{\tau}_{k}$, $\xi^{\tau}_{k}$ are measurable with respect to $\mathcal{F}_{t_{k}}:=\sigma(W_{t}, t\leq t_{k})$ and $\Delta W_{k+1}$ are independent of $\mathcal{F}_{t_{k}}$, by \eqref{Lyapunov condition 1} and the property of the conditional expectation, we obtain 
	\begin{align*}
		&\left(1+\frac{a\tau}{2}\right) \mathbb{E}\left[\bar{\Gamma}^{p}(X^{\tau}_{k+1},\xi^{\tau}_{k+1}) | \mathcal{F}_{t_{k}}\right] \\
		\leq \ & \mathbb{E}\left[ \bar{\Gamma}^{p}(X^{\tau}_{k} + \sigma \Delta W_{k+1},\xi^{\tau}_{k}) | \mathcal{F}_{t_{k}}\right]  + \tau C(p,\alpha,\beta,a,b,\sigma) \\
		=\ &\mathbb{E}\left[\bar{\Gamma}^p(Y+\sigma\Delta W_{k+1},\zeta)\right]\Big|_{Y=X^{\tau}_{k},\zeta=\xi^{\tau}_{k}}  + \tau C(p,\alpha,\beta,a,b,\sigma) \\
		\leq \ &(1+\epsilon\tau) \bar{\Gamma}^{p}(X^{\tau}_{k},\xi^{\tau}_{k}) + \tau C(p,\epsilon,\alpha,\beta, a,b,\sigma). 
	\end{align*}
	By taking $\epsilon=\frac{a}{8}$, for $\tau<\frac{4}{a}$, we have 
	\begin{equation}\label{Lyapunov condition on xk}
		\mathbb{E}\left[\bar{\Gamma}^{p}(X^{\tau}_{k+1},\xi^{\tau}_{k+1}) | \mathcal{F}_{t_{k}}\right] 
		\leq \frac{1}{(1+\tfrac{a\tau}{4})} \bar{\Gamma}^{p}(X^{\tau}_{k},\xi^{\tau}_{k}) + \tau C(p,\alpha,\beta,a,b,\sigma). 
	\end{equation}
	In addition, it is evident that $\lim_{(X,\xi)\rightarrow \infty} \bar{\Gamma}^{p}(X,\xi) = \infty$. 
	Hence the Lyapunov condition holds for $\{(X^{\tau}_k,\xi^{\tau}_k): k\geq0\}$ when $\tau<\frac{4}{a}$. 
	By an analogous derivation one can verify the Lyapunov condition for $\{X^{\tau}_k: k\geq0\}$. 
	
	On the other hand, as shown in Appendix \ref{Appendix:minorization}, the process $\{(X^{\tau}_k,\xi^{\tau}_k): k\geq0\}$ satisfies the minorization condition when $b>4a$ and $\tau<\min\{\frac{a}{1+4b|\beta|},\frac{1}{1+12|\beta|},\frac{1}{1+4|\alpha|}\}$. 
	In addition, the minorization condition for $\{X^{\tau}_k: k\geq0\}$ follows from the fact that the driven noise for $X^{\tau}_k$ is non-degenerate when $\sigma>0$. 	
	Thus by \cite[Theorem 2.5]{Mattingly2002} we obtain the desired conclusion. 
%	 there exists a unique stationary measure $\nu^{\tau}$ for $\{(X^{\tau}_k,\xi^{\tau}_k): k\geq0\}$, with its marginal $\tilde{\nu}^{\tau}$ for $\{X^{\tau}_k: k\geq0\}$. 
\end{proof}

%We first give the expression of $\xi^{\tau}_{k}$. To this end, we introduce $c^{\tau}_{k}=\cos \xi^{\tau}_{k}$ and $s^{\tau}_{k}=\sin \xi^{\tau}_{k}$, which are determined by \eqref{ck} and \eqref{sk}, respectively. Then we have $\tan \xi^{\tau}_{k} = s^{\tau}_{k}/c^{\tau}_{k}$ and  
%\begin{align*}
%	&\sin(\xi^{\tau}_{k} - \xi^{\tau}_{k+1})  
%	=  s^{\tau}_{k+1}(c^{\tau}_{k+1}-c^{\tau}_{k}) - c^{\tau}_{k+1}(s^{\tau}_{k+1}-s^{\tau}_{k}) \\
%	&= - \tau \Big( \beta+2b\|X^{\tau}_{k+1}\|^2 + \left(a\sin2\xi^{\tau}_{k+1}+b\cos2\xi^{\tau}_{k+1}\right) \left(|x^{\tau}_{k+1}|^2-|y^{\tau}_{k+1}|^2\right) \\
%	&\quad + (2b\sin2\xi^{\tau}_{k+1}-2a\cos2\xi^{\tau}_{k+1}) x^{\tau}_{k+1} y^{\tau}_{k+1}  \Big) +s^{\tau}_{k+1}\mathcal{R}_{k+1}^{c}-c^{\tau}_{k+1}\mathcal{R}_{k+1}^{s}, 
%\end{align*} 
%where $\mathcal{R}_{k+1}^{c}$ and $\mathcal{R}_{k+1}^{s}$ are given by \eqref{Rkc} and \eqref{Rks}. 
%This means that $\xi^{\tau}_{k}$ are determined by 
%\begin{equation}\label{numerical xi}
%	\begin{aligned}
%	\xi^{\tau}_{k+1} &= \xi^{\tau}_{k} + \arcsin \Big\{ \tau \Big( \beta+2b\|X^{\tau}_{k+1}\|^2 + \left(a\sin2\xi^{\tau}_{k+1}+b\cos2\xi^{\tau}_{k+1}\right) \left(|x^{\tau}_{k+1}|^2-|y^{\tau}_{k+1}|^2\right) \\
%	&\quad + (2b\sin2\xi^{\tau}_{k+1}-2a\cos2\xi^{\tau}_{k+1}) x^{\tau}_{k+1} y^{\tau}_{k+1}  \Big) -s^{\tau}_{k+1}\mathcal{R}_{k+1}^{c}+c^{\tau}_{k+1}\mathcal{R}_{k+1}^{s} \Big\}. 
%\end{aligned}
%\end{equation}

\subsection{Proof of Proposition \ref{Prop:strong convergence}}
In this subsection we prove the strong convergence of the backward Euler method in finite time interval $[0,T]$. Recall that $t_{k}=k\tau$, $k=0,1,2,...,N$ are gridpoints. 
In the following we omit the dependence of constant $C$ on the coefficients of SDE \eqref{compact SDE}. 
We first give the moment boundedness of $X_t$ and $X^{\tau}_{k}$. 
\begin{lemma}
	For any $p\geq1$, it holds that 
	\begin{equation}\label{bound Xt} 
		\mathbb{E}\left[\sup\limits_{t\in[0,T]}\|X_t\|^{2p} \right]
		\leq  C(p, T). 
	\end{equation}  
\end{lemma}
\begin{proof}
	Applying It\^o's formula and Young's inequality, we obtain 
	\begin{equation}\label{moment bound 1}
		\begin{aligned}
			\rmd \|X_t\|^{2p} 
			&= p\|X_t\|^{2(p-1)} \left( 2X_t\cdot F(X_t)+2p\sigma^2 \right) \rmd t 
			+ 2p\sigma\|X_t\|^{2(p-1)} X_t\cdot\rmd W_t  \\
			&\leq \left(-2a \|X_t\|^{2p} +C(p)\right) \rmd t + 2p\sigma\|X_t\|^{2(p-1)} X_t\cdot\rmd W_t, 
		\end{aligned}
	\end{equation}
	which means that 
	\begin{equation*}
		\rmd (e^{2at}\|X_t\|^{2p}) \leq e^{2at}C(p) \rmd t + 2p\sigma e^{2at}\|X_t\|^{2(p-1)} X_t\cdot\rmd W_t. 
	\end{equation*}
	Hence for any $t\in[0,T]$, 
	\begin{equation}\label{moment bound 2}
		\mathbb{E}[\|X_t\|^{2p}] \leq e^{-2at} \mathbb{E}[\|X_0\|^{2p}] + \frac{C(p)}{2a}.  
	\end{equation}
	In addition, it follows from \eqref{moment bound 1} that 
	\begin{equation*} 
		\mathbb{E}\left[\sup\limits_{t\in[0,T]}\|X_t\|^{2p} \right]
		\leq \mathbb{E}[\|X_0\|^{2p}] + C(p)T
		+ \mathbb{E}\left[\sup\limits_{t\in[0,T]}\int_{0}^{t} 2p\sigma\|X_s\|^{2p-2} X_s \cdot \rmd W_s \right]. 
	\end{equation*}
	Adopting Burkholder--Davis--Gundy's inequality, we have 
	\begin{equation*}
		\mathbb{E}\left[\sup\limits_{t\in[0,T]}\int_{0}^{t} 2p\sigma\|X_s\|^{2p-2} X_s\cdot \rmd W_s \right]  
		\leq C(p) \left(\int_{0}^{T} \mathbb{E}\left[ \|X_s\|^{4p-2}\right] \rmd s \right)^{1/2}, 
	\end{equation*}
	which, together with \eqref{moment bound 2}, leads to the desired conclusion. 
\end{proof}

\begin{lemma}
	For any $p\geq1$ and $\tau<\frac{4}{a}$, it holds that 
	\begin{equation}\label{bound Xk}
		\mathbb{E}\left[\sup\limits_{0\leq k \leq N}\|X^{\tau}_{k}\|^{2p}\right] \leq C(p, T). 
	\end{equation}
\end{lemma}
\begin{proof}
	From \eqref{Lyapunov condition on xk} we know that  for $p\geq2$ and $\tau<\frac{4}{a}$, 
	\begin{equation*}
		\mathbb{E}\left[\bar{\Gamma}^{p}(X^{\tau}_{k+1},\xi^{\tau}_{k+1})\right] 
		\leq \frac{4}{(4+a\tau)} \mathbb{E}\left[\bar{\Gamma}^{p}(X^{\tau}_{k},\xi^{\tau}_{k})\right] + C(p)\tau, 
	\end{equation*}
	which means that  
	\begin{equation}\label{numerical bound1}
		\mathbb{E}[\|X^{\tau}_{k}\|^{2p}] \leq 
		\mathbb{E}\left[\bar{\Gamma}^{p}(X^{\tau}_{k},\xi^{\tau}_{k})\right] 
		\leq \left(\frac{4}{4+a\tau}\right)^{k} \mathbb{E}\left[\bar{\Gamma}^{p}(X^{\tau}_{0},\xi^{\tau}_{0})\right] + C(p). 
	\end{equation}
    By using 
	\begin{equation*}
		\frac{1}{2} \|X^{\tau}_{k} + \sigma \Delta W_{k+1}\|^2 + \frac{1}{2} \|X^{\tau}_{k+1}\|^2 \geq (X^{\tau}_{k} + \sigma \Delta W_{k+1}) \cdot X^{\tau}_{k+1}, 
	\end{equation*}
	Young's inequality, and \eqref{discrete system}, we know that  
	\begin{equation*}
			\|X^{\tau}_{k} + \sigma \Delta W_{k+1}\|^2
			\geq \|X^{\tau}_{k+1}\|^2 - 2\tau \left(  \alpha\|X^{\tau}_{k+1}\|^2  - a\|X^{\tau}_{k+1}\|^4 \right) 
			\geq (1+a\tau) \|X^{\tau}_{k+1}\|^2  - C\tau. 
	\end{equation*}
	Consequently, 
	\begin{equation*}
		\|X^{\tau}_{k}\|^{2} \leq \|X^{\tau}_{0}\|^{2} + 2\sigma \sum_{i=1}^{k} X^{\tau}_{i-1} \cdot \Delta W_{i} 
		+ \sigma^2 \sum_{i=1}^{k} \|\Delta W_{i}\|^2 + Ck\tau. 
	\end{equation*}
	Taking $p$-th power on both sides we obtain 
	\begin{align*}
		\|X^{\tau}_{k}\|^{2p} &\leq 4^{p-1} \left( \|X^{\tau}_{0}\|^{2p} + (2\sigma)^{p} \left|\sum_{i=1}^{k} X^{\tau}_{i-1} \cdot \Delta W_{i} \right|^{p} 
		+ \sigma^{2p} \left|\sum_{i=1}^{k} \|\Delta W_{i}\|^2 \right|^{p} + C(p)T^{p} \right). 
	\end{align*}
	By Burkholder--Davis--Gundy's inequality and \eqref{numerical bound1}, we have 
	\begin{align*}
		&\mathbb{E}\left[\sup\limits_{0\leq k \leq N}\left|\sum_{i=1}^{k} X^{\tau}_{i-1} \cdot \Delta W_{i}\right|^{p}\right] 
		\leq C(p) \mathbb{E}\left[ \left(\sum_{i=1}^{N} |x^{\tau}_{i-1}|^2 \tau  \right)^{\frac{p}{2}} + \left(\sum_{i=1}^{N} |y^{\tau}_{i-1}|^2 \tau  \right)^{\frac{p}{2}}\right] \\
		&\leq C(p) \tau^{\frac{p}{2}} N^{\frac{p}{2}-1}  \sum_{i=1}^{N}\mathbb{E}\left[ |x^{\tau}_{i-1}|^{p} + |y^{\tau}_{i-1}|^{p}\right] 
		\leq C(p) T^{\frac{p}{2}}. 
	\end{align*}
	Furthermore, for $j=1, 2$, 
	\begin{align*}
		\mathbb{E}\left[\sup\limits_{0\leq k \leq N}\left|\sum_{i=1}^{k} |\Delta W_{i}^j|^2\right|^{p}\right] 
		\leq N^{p-1}  \sum_{i=1}^{N}\mathbb{E}\left[ |\Delta W_{i}^j|^{2p}\right] 
		\leq C(p) N^{p-1}  \sum_{i=1}^{N} \tau^{p} 
		\leq C(p) T^{p}. 
	\end{align*}
	Thus \eqref{bound Xk} holds for $p\geq2$. 
	Finally, the case for $p\in[1,2)$ follows directly from Jensen's inequality.
\end{proof}

With these preparations in place, we are now ready to prove Proposition \ref{Prop:strong convergence}. 

\begin{proof}[Proof of Proposition \ref{Prop:strong convergence}]
	For $R>0$, we introduce two stopping times  
	\begin{equation*}
		\vartheta_R = \inf\{ t\geq0: \|X_t\| > R \}, \quad \vartheta^{\tau}_R = \inf\{ k\tau\geq0: \|X^{\tau}_{k} \| > R \}. 
	\end{equation*}
	By using Chebychev's inequality and \eqref{bound Xt}, we have  
	\begin{equation}\label{stopping time 1}
		\mathbb{P}\left( \vartheta_R\leq T \right) = \mathbb{E} \left[ \chi_{\{\vartheta_R\leq T\}} \frac{\|X_{\vartheta_R}\|^{p}}{R^{p}} \right] 
		\leq \frac{1}{R^{p}} \mathbb{E}\left[\sup\limits_{t\in[0,T]}\|X_t\|^{p} \right] 
		\leq \frac{C(p, T)}{R^{p}}. 
	\end{equation}  
	Similarly, in view of \eqref{bound Xk}, 
	\begin{equation}\label{stopping time 2}
		\mathbb{P}\left( \vartheta^{\tau}_R\leq T \right) \leq \frac{1}{R^{p}} \mathbb{E}\left[\sup\limits_{0\leq k \leq N}\|X^{\tau}_{k} \|^{p} \right] 
		\leq  \frac{C(p, T)}{R^{p}}. 
	\end{equation}
	In addition, it can be verified that the nonlinear term $F$ of SDE \eqref{compact SDE} satisfies 
	\begin{align}
		\|F(Y)\| &\leq C (1+\|Y\|^{3}), \quad \forall \,  Y \in \mathbb{R}^2,   \label{polynomial growth} \\ 
		\|F(Y)-F(\bar{Y})\| &\leq C(R) \|Y-\bar{Y}\|, \quad \forall \,  Y, \bar{Y} \in \mathbb{R}^2 \ \ \text{with}\ \ \|Y\| \vee \|\bar{Y}\| \leq R.  \label{local Lipschitz}
	\end{align} 
	
	For $s\in(t_{k},t_{k+1})$, by Taylor's theorem with integral form of the remainder, we obtain 
	\begin{equation*}
		F(X_{s}) = F(X_{t_{k+1}}) + \int_{0}^{1} \nabla F\left(\varsigma X_{t_{k+1}}+(1-\varsigma)X_{s}\right) (X_{s}-X_{t_{k+1}}) \rmd\varsigma 
		=: F(X_{t_{k+1}}) + \mathcal{E}(X_{s}, X_{t_{k+1}}), 
	\end{equation*}
	which means that 
	\begin{equation*}
		X_{t_{k+1}} 
		= X_{t_{0}} + \tau\sum_{i=1}^{k+1} F(X_{t_{i}}) + \sum_{i=0}^{k}\int_{t_{i}}^{t_{i+1}} \mathcal{E}(X_{s}, X_{t_{i+1}}) \rmd s + \sigma W_{t_{k+1}}. 
	\end{equation*}
	For $R>0$, introduce a sequence of events  
	\begin{equation*}
		\varrho_{k,R}=\left\{ \omega:  \sup\limits_{0\leq i \leq k} \|X_{t_{i}}\| + \sup\limits_{0\leq i \leq k} \|X^{\tau}_{i}\|\leq 2R \right\}. 
	\end{equation*}
	Obviously $\varrho_{k,R} \subset \varrho_{j,R}$ for $k\geq j$. 
	Since 
	\begin{equation*}
		X^{\tau}_{k+1} = X^{\tau}_{0} + \tau \sum_{i=1}^{k+1}F(X^{\tau}_{i}) + \sigma W_{t_{k+1}}, 
	\end{equation*}
	by \eqref{local Lipschitz} we have 
	\begin{align*}
		&\chi_{\varrho_{k,R}}\|X_{t_{k}} - X^{\tau}_{k}\| 
		\leq C(R)\tau \sum_{i=1}^{k} \chi_{\varrho_{i,R}}\|X_{t_{i}}-X^{\tau}_{i}\| + \sum_{i=0}^{k-1}\int_{t_{i}}^{t_{i+1}} \|\mathcal{E}(X_{s}, X_{t_{i+1}})\| \rmd s \\ 
		&\leq C(R)\tau\left(\|X_{t_{k}}\|+\|X^{\tau}_{k}\|\right) + C(R)\tau \sum_{i=1}^{k-1} \chi_{\varrho_{i,R}}\|X_{t_{i}}-X^{\tau}_{i}\| + \sum_{i=0}^{k-1}\int_{t_{i}}^{t_{i+1}} \|\mathcal{E}(X_{s}, X_{t_{i+1}})\| \rmd s. 
	\end{align*}
	By using \eqref{polynomial growth} and \eqref{bound Xt}, we deduce that 
	\begin{align*}
		&\mathbb{E}[\|\mathcal{E}(X_{s}, X_{t_{i+1}})\|^{p}] 
		\leq \mathbb{E}\left[ \int_{0}^{1} \left\| \nabla F\left(\varsigma X_{t_{i+1}}+(1-\varsigma)X_{s}\right) (X_{s}-X_{t_{i+1}}) \right\|^{p} \rmd\varsigma \right] \\
		&\leq C \left(\mathbb{E}\left[\sup\limits_{s\in[0,T]}\|X_{s}\|^{4p}\right]\right)^{\frac{1}{2}}  \left(\mathbb{E}\left[\|X_{s}-X_{t_{i+1}}\|^{2p}\right]\right)^{\frac{1}{2}}
		\leq C(p,T) \left(\mathbb{E}\left[\|X_{s}-X_{t_{i+1}}\|^{2p}\right]\right)^{\frac{1}{2}}, 
	\end{align*}
	where 
	\begin{align*}
		\mathbb{E}\left[\|X_{s}-X_{t_{i+1}}\|^{2p}\right] 
		&\leq (2\tau)^{2p-1} \int_{s}^{t_{i+1}} \mathbb{E}[\|F(X_{r})\|^{2p}] \rmd r + 2^{2p-1} \mathbb{E}\left[\sigma^{2p}\|W_{t_{i+1}}-W_{s}\|^{2p}\right]  \\
		&\leq C(p,T)\tau^{p}. 
	\end{align*}
	Consequently, 
	\begin{align*}
		&\mathbb{E}\left[\chi_{\varrho_{k,R}}\|X_{t_{k}} - X^{\tau}_{k}\|^{p} \right] \\
		\leq \ &C(R,p) \tau^{p} \mathbb{E}\left[\|X_{t_{k}}\|^{p}+\|X^{\tau}_{k}\|^{p}\right] + C(R,p,T) \tau \sum_{i=1}^{k-1} \mathbb{E}\left[\chi_{\varrho_{i,R}}\|X_{t_{i}}-X^{\tau}_{i}\|^{p}\right] + C(p,T) \tau^{\frac{p}{2}} \\
		\leq \ &C(R,p,T) \tau \sum_{i=1}^{k-1} \mathbb{E}\left[\chi_{\varrho_{i,R}}\|X_{t_{i}}-X^{\tau}_{i}\|^{p}\right] + C(R,p,T) \tau^{\frac{p}{2}}. 
	\end{align*}
	It follows from Gronwall's inequality that  
	\begin{equation}\label{convergence X tauR}
		\sup\limits_{0\leq k \leq N}\mathbb{E}\left[\chi_{\varrho_{k,R}}\|X_{t_{k}} - X^{\tau}_{k}\|^{p} \right]
		\leq  C(R,p,T) \tau^{\frac{p}{2}}. 
	\end{equation}
	Furthermore, by the definition of $C_{t}$ and $C^{\tau}_{k}$, we know that 
	\begin{equation}\label{bound C}
		\mathbb{E}\left[\sup\limits_{t\in[0,T]}\|C_t\|^{p} + \sup\limits_{0\leq k \leq N}\|C^{\tau}_{k}\|^{p}\right] = 2 < \infty. 
	\end{equation}
	By an analogous analysis, we can derive that 
	\begin{equation}\label{convergence C tauR}
		\sup\limits_{0\leq k \leq N}\mathbb{E}\left[\chi_{\varrho_{k,R}}\|C_{t_{k}} - C^{\tau}_{k}\|^{p} \right]
		\leq  C(R,p,T) \tau^{\frac{p}{2}}. 
	\end{equation}
	
	Let $e_{k} = X_{t_{k}}-X^{\tau}_{k} + C_{t_{k}}-C^{\tau}_{k}$. 
	By the definition of $\vartheta_R$, $\vartheta^{\tau}_R$ and $\varrho_{k,R}$, it can be inferred that 
	\begin{equation*}
		\left\{ \omega: \vartheta_R>T, \vartheta^{\tau}_R>T \right\} \subset \varrho_{N,R} \subset \varrho_{k,R}, \quad \forall \,  0\leq k \leq N. 
	\end{equation*}
	Applying the inequality $xy\leq\frac{\delta}{2} x^2 + \frac{1}{2\delta}y^2$, $\forall \,  x, y, \delta>0$, 
	we have 
	\begin{align*}
		\sup\limits_{0\leq k \leq N}\mathbb{E}\left[\|e_{k}\|^{p} \right] 
		&= \sup\limits_{0\leq k \leq N}\mathbb{E}\left[\|e_{k}\|^{p} \chi_{\{\vartheta_R>T, \vartheta^{\tau}_R>T\}}  \right] + \sup\limits_{0\leq k \leq N}\mathbb{E}\left[\|e_{k}\|^{p} \chi_{\{\vartheta_R\leq T \text{ or } \vartheta^{\tau}_R\leq T\}}  \right] \\ 
		&\leq  
		\sup\limits_{0\leq k \leq N}\mathbb{E}\left[\|e_{k}\|^{p} \chi_{\varrho_{k,R}}  \right] + \frac{\delta}{2}  \sup\limits_{0\leq k \leq N} \mathbb{E}[\|e_{k}\|^{2p}]
		+ \frac{1}{2\delta} \mathbb{P}\left( \vartheta_R\leq T \text{ or } \vartheta^{\tau}_R\leq T \right). 
	\end{align*}
	Given any $\epsilon>0$, by \eqref{bound Xt}, \eqref{bound Xk}, and \eqref{bound C},  we can choose $\delta$ so that  
	\begin{align*}
		\frac{\delta}{2} \sup\limits_{0\leq k \leq N} \mathbb{E}[\|e_{k}\|^{2p}]
		\leq \frac{4^{p-1}\delta}{2} \sup\limits_{0\leq k \leq N} \mathbb{E}\left[ \|X_{t_{k}}\|^{2p} + \|X^{\tau}_{k}\|^{2p} + \|C_{t_{k}}\|^{2p} + \|C^{\tau}_{k}\|^{2p}\right] < \frac{\epsilon}{3}. 
	\end{align*}
	Then for such $\delta>0$, by \eqref{stopping time 1} and \eqref{stopping time 2}, we can choose $R$ so that 
	\begin{equation*}
		\frac{1}{2\delta} \mathbb{P}\left( \vartheta_R\leq T \text{ or } \vartheta^{\tau}_R\leq T \right) \leq \frac{C(p, T)}{\delta R^{p}} < \frac{\epsilon}{3}. 
	\end{equation*}
	In addition, by \eqref{convergence X tauR} and \eqref{convergence C tauR} we know that 
	\begin{equation*}
		\sup\limits_{0\leq k \leq N}\mathbb{E}\left[\|e_{k}\|^{p} \chi_{\varrho_{k,R}}  \right]
		\leq 2^{p-1}
		\sup\limits_{0\leq k \leq N}\mathbb{E}\left[\chi_{\varrho_{k,R}}\|X_{t_{k}} - X^{\tau}_{k}\|^{p} + \chi_{\varrho_{k,R}}\|C_{t_{k}} - C^{\tau}_{k}\|^{p} \right]
		\leq  C(R,p,T) \tau^{\frac{p}{2}}. 
	\end{equation*} 
	Finally, choose $\tau$ sufficiently small such that 
	\begin{equation*}
		\sup\limits_{0\leq k \leq N}\mathbb{E}\left[\|e_{k}\|^{p} \chi_{\varrho_{k,R}}  \right] < \frac{\epsilon}{3}, 
	\end{equation*}  
	which completes the proof. 
\end{proof}

\section{Numerical experiments}\label{Sec:numerical example}
In this section we present several numerical experiments to validate the theoretical results. The simulations are based on the backward Euler method \eqref{BEM}. For the computation of numerical Lyapunov exponents, we use \eqref{numerical variational SDE} and \eqref{numerical lyp}. 

We begin by illustrating the SDE \eqref{SDE} under various parameter settings, offering an immediate graphical insights. 
For the deterministic case ($\sigma=0$), the equation \eqref{SDE} serves as a supercritical Hopf bifurcation (see \cite{Doan2018}): when $\alpha\leq0$ the system has a globally attracting equilibrium at $(x,y)^{\top}=(0,0)^{\top}$ which is exponentially stable until $\alpha=0$; when $\alpha>0$, the system has a limit cycle at $\left\{ (x,y)\in\mathbb{R}^2: x^2+y^2=\frac{\alpha}{a} \right\}$ which is globally attracting on $\mathbb{R}^2/\{0\}$ (see Fig. \ref{Fig:Hopf} (A)--(B) for the limit cycle with $\alpha=1$). 

In the presence of noise ($\sigma>0$), two distinct dynamical behaviors emerge: either the evolution of all initial points converges to a fixed random point (see Fig. \ref{Fig:Hopf} (C)--(D)), or it converges to a rather complicated object (see Fig. \ref{Fig:Hopf} (E)--(F)). The former is indicative of the phenomenon of synchronisation (see e.g. \cite{Flandoli2017}), namely, convergence of all trajectories to a single random equilibrium point, while the latter suggests the presence of a random attractor, which supports the SRB measure \cite{Ledrappier1988}. 

The differences between two types of dynamical phases, i.e., synchronisation and chaotic phases, can also be inferred from the Lyapunov exponents. In general, random attractors with negative Lyapunov exponents are associated with synchronisation, where the attractors are singletons; 
conversely, random attractors with a positive Lyapunov exponent impedes synchronisation and exhibits complex dynamic behaviors. 
As demonstrated in Theorem \ref{Thm:SRB}, the large shear strength $b$ implies the positivity of numerical Lyapunov exponent of the discrete RDS. 
Accordingly, we observe the single random equilibrium point in Fig. \ref{Fig:Hopf} (C)--(D) for $b=2$ and observe the random strange attractor in Fig. \ref{Fig:Hopf} (E)--(F) for $b=10$. 

Simulations of numerical Lyapunov exponents are plotted in Fig. \ref{Fig:lyp} for various parameters $\alpha$ and $b$, revealing a positive correlation between the shear strength $b$ and the numerical Lyapunov exponent $\lambda^{\tau}$ for a fixed $\alpha$. 
Specifically, a positive Lyapunov exponent is evident for large shear strength $b$, as indicated by the right-hand side of the red border line in Fig. \ref{Fig:lyp}.  
To investigate further, we plot numerical Lyapunov exponents for different $b$ with fixed $\alpha=1$ in Fig. \ref{Fig:plotb}. 
It appears that a bifurcation may occur among $b\in(5,6)$. Namely, the top numerical Lyapunov exponent changes sign from negative to positive within the interval $b\in(5,6)$. 
To elucidate the relationship between the sign of $\lambda^{\tau}$ and the nature of the random attractor, we conduct tests shown in Fig. \ref{Fig:plotb}, observing singletons when $b=4$, $5$ (see snapshot (A)), and random strange attractors when $b=6$, $7$ (see snapshots (B)--(C)). 

From Theorem \ref{Thm:SRB} and Figs. \ref{Fig:lyp}, \ref{Fig:plotb}, it is evident that a large shear strength $b$ induces a positive numerical Lyapunov exponent $\lambda^{\tau}$, thereby giving rise to the random attractor and associated SRB measure. 
Consequently, we take $b=15$ to examine the behavior of the random attractor $A^{\tau}(\omega)$ and the SRB measure $\mu^{\tau}_{\omega}$. 
We initiate our simulations with one million initial points from a two-dimensional normal distribution. After a brief period of evolution, these points still appear to follow a normal distribution (Fig. \ref{Fig:SRB} (A)). 
While after a long period of evolution, the evolution of all initial points converges to a strange attractor, characterized by its ``chaotic'' nature, Cantor-set-type structure, no fuzziness and very intense filamentation; see Fig. \ref{Fig:SRB} (B) for a particular realization of the SRB measure $\mu^{\tau}_{\theta_T\omega}$ at $T=500$. 
Given the randomness of the random attractor and the SRB measure, they continuously evolve and move in the phase space of the system, making them distinct at each moment in time. 
To capture this dynamical behavior, we have provided successive snapshots of sample measures $\mu^{\tau}_{\theta_t\omega}$ at times $t=t_0+0.1k$ for $t_0=500$ and $k=1, 2, 3, 4$, as shown in Fig. \ref{Fig:SRB} (C)--(F).

\begin{figure}
	\centering
	\subfloat[$b=10$, $\sigma=0$, $T=4$]{
		\begin{minipage}[t]{0.45\linewidth}
			\centering
			\includegraphics[height=5cm,width=5cm]{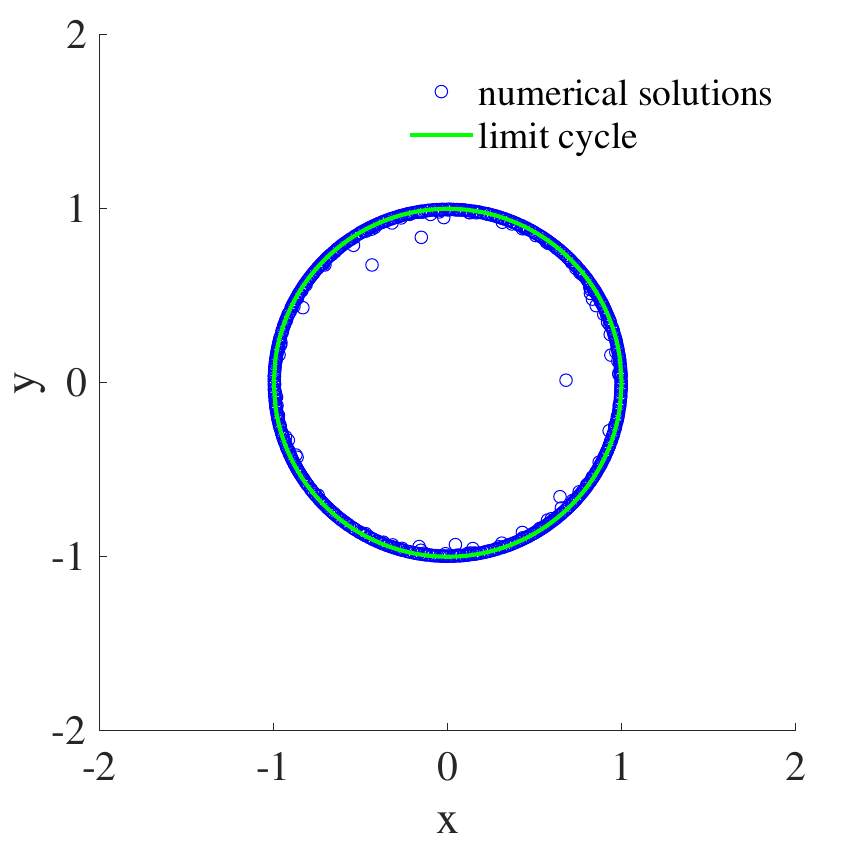} 
		\end{minipage}
	}
	\subfloat[$b=10$, $\sigma=0$, $T=40$]{
		\begin{minipage}[t]{0.45\linewidth}
			\centering
			\includegraphics[height=5cm,width=5cm]{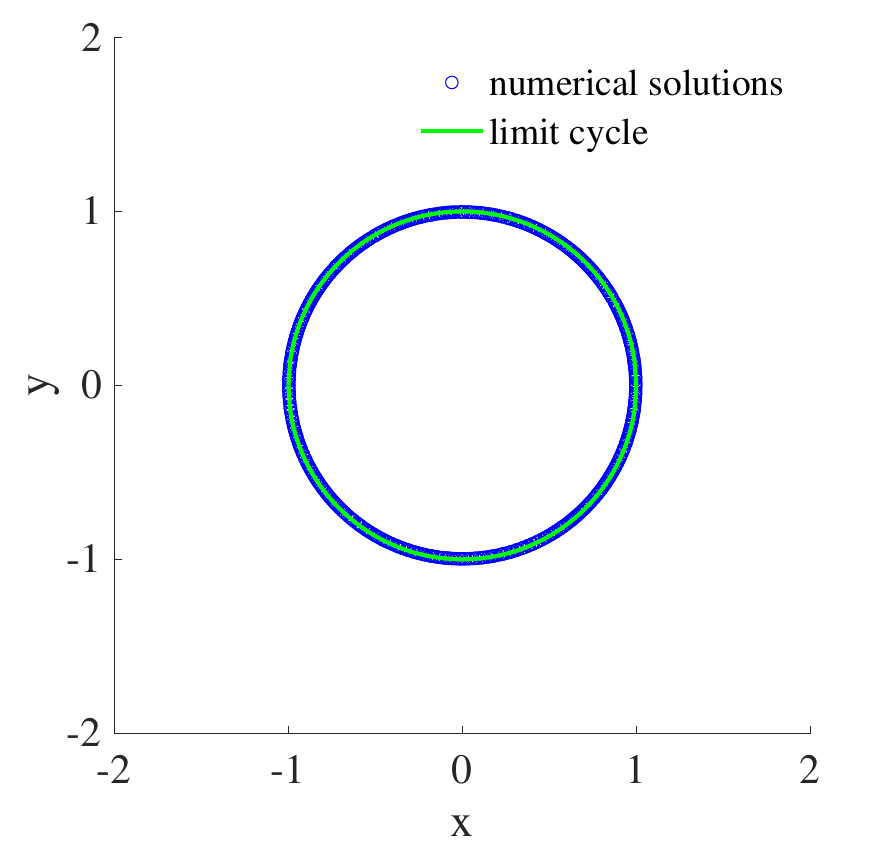} 
		\end{minipage}
	}
	\\	
	\subfloat[$b=2$, $\sigma=1$, $T=4$]{
		\begin{minipage}[t]{0.45\linewidth}
			\centering
			\includegraphics[height=5cm,width=5cm]{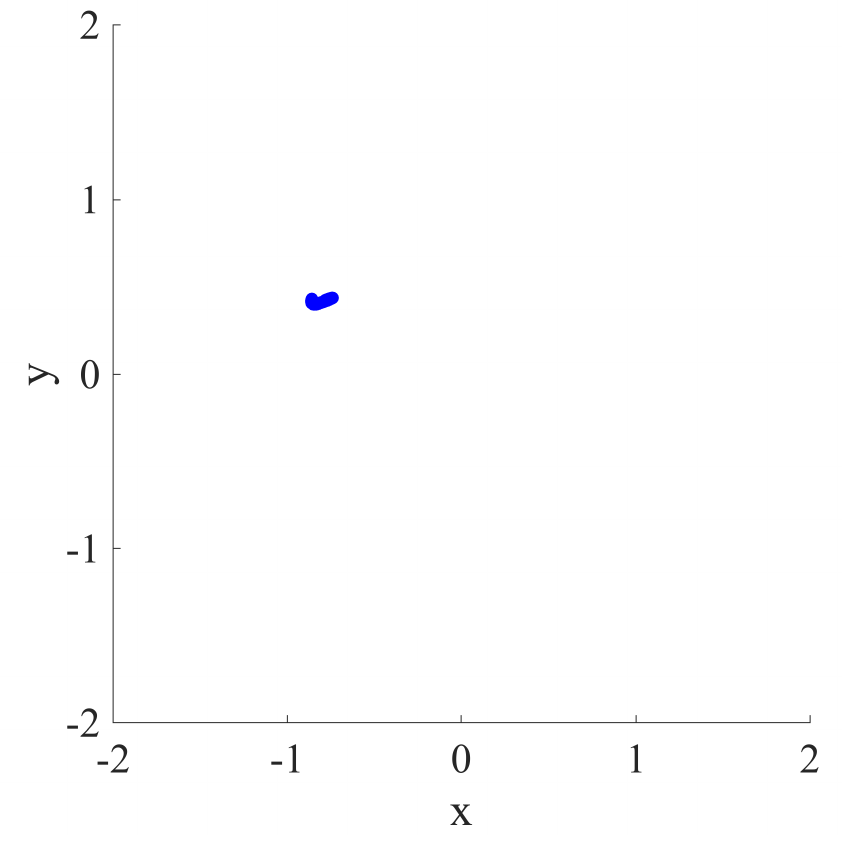} 
		\end{minipage}
	}
	\subfloat[$b=2$, $\sigma=1$, $T=30, 40, 50$]{
		\begin{minipage}[t]{0.45\linewidth}
			\centering
			\includegraphics[height=5cm,width=5cm]{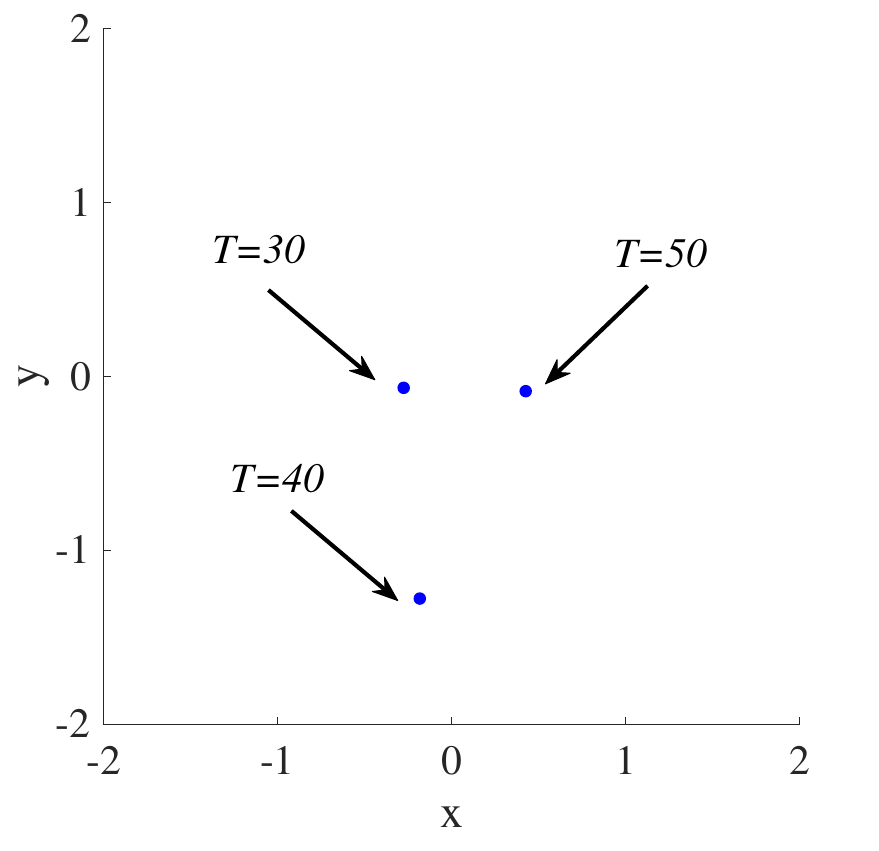} 
	\end{minipage}
	}
	\\
	\subfloat[$b=10$, $\sigma=1$, $T=4$]{
		\begin{minipage}[t]{0.45\linewidth}
			\centering
			\includegraphics[height=5cm,width=5cm]{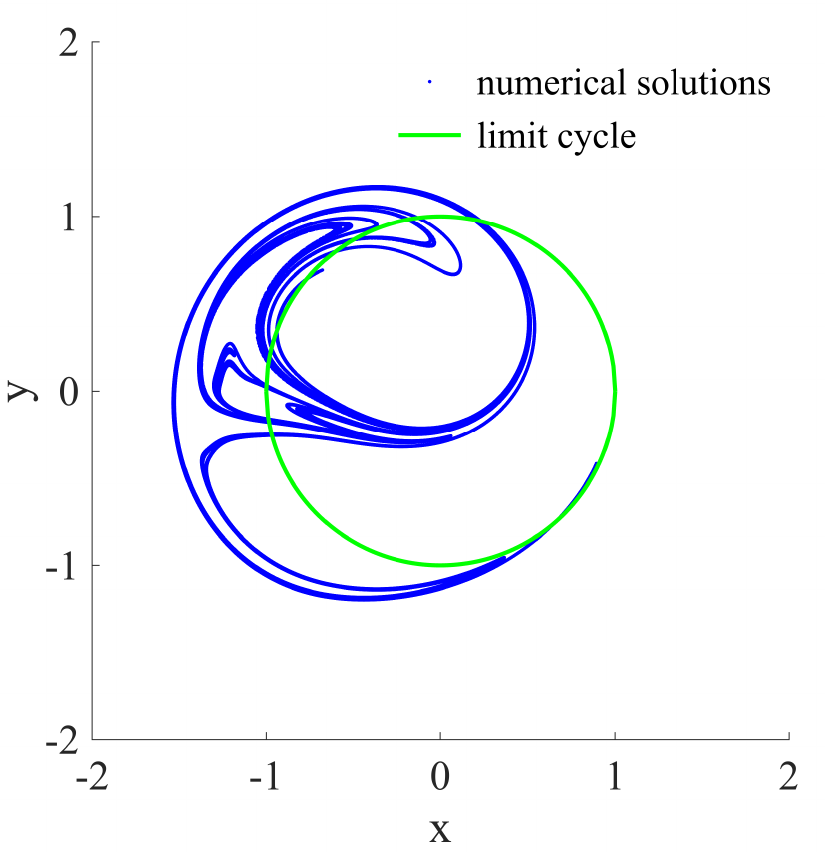} 
	\end{minipage}
}
\subfloat[$b=10$, $\sigma=1$, $T=40$]{
	\begin{minipage}[t]{0.45\linewidth}
		\centering
		\includegraphics[height=5cm,width=5cm]{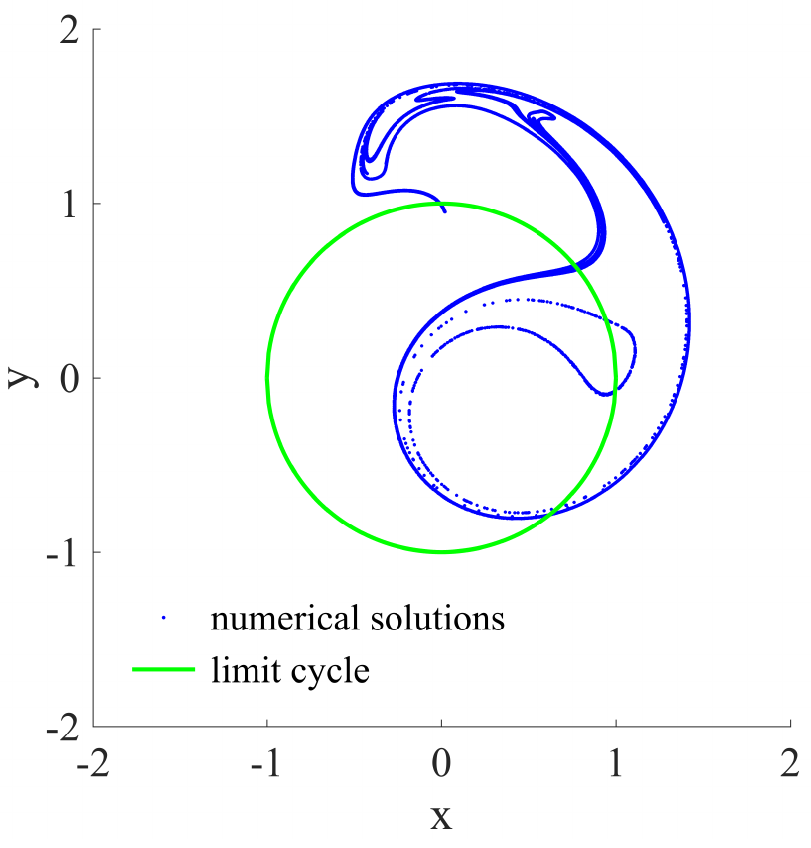}
\end{minipage}
}

\centering
\caption{Phase portraits of numerical solutions for a fixed noise realization $\omega$. One million initial points in phase space are chosen, and are computed to $T=4$ and $40$. The parameters $\alpha=\beta=a=1$ are used. In (A)--(B), in absence of noise, the evolution seems to converge to either a fixed point or a limit cycle, depending on the choice of $\alpha$. In (C)--(D) in the presence of noise and $b=2$ we observe synchronisation, i.e. convergence of all trajectories to a single point. In (E)--(F), for $b=10$ there is no synchronisation but convergence to a more complicated random attractor. }
\label{Fig:Hopf}
\end{figure}

\begin{figure}[tbhp]
	\centering  
	\includegraphics[scale = 0.5]{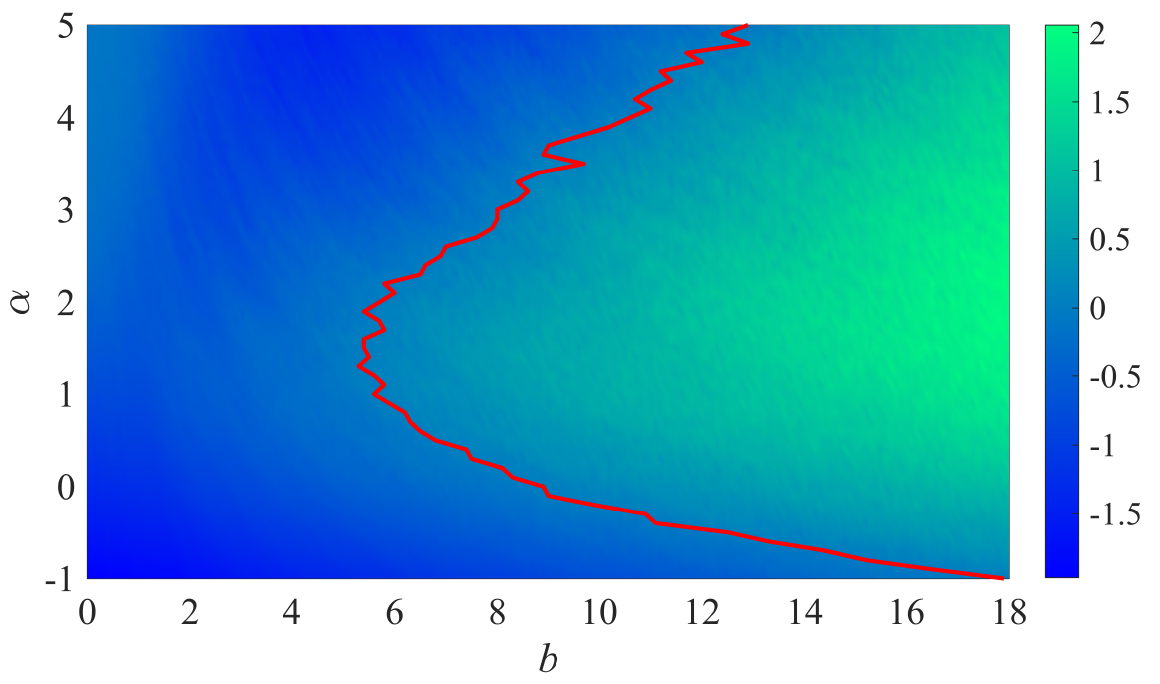} 
	\caption{Numerical Lyapunov exponents $\lambda^{\tau}$ for different parameters $\alpha$ and $b$ of SDE \eqref{SDE} computed to $T=1000$ with $\beta=a=\sigma=1$. 
	The red line corresponds to the border $\lambda^{\tau}=0$.}
	\label{Fig:lyp}
\end{figure}

\begin{figure}[tbhp]
	\centering
	\subfloat[\centering Numerical Lyapunov exponents for different $b$]{
		\begin{minipage}[t]{0.7\linewidth}
			\centering 
			\includegraphics[scale = 0.45]{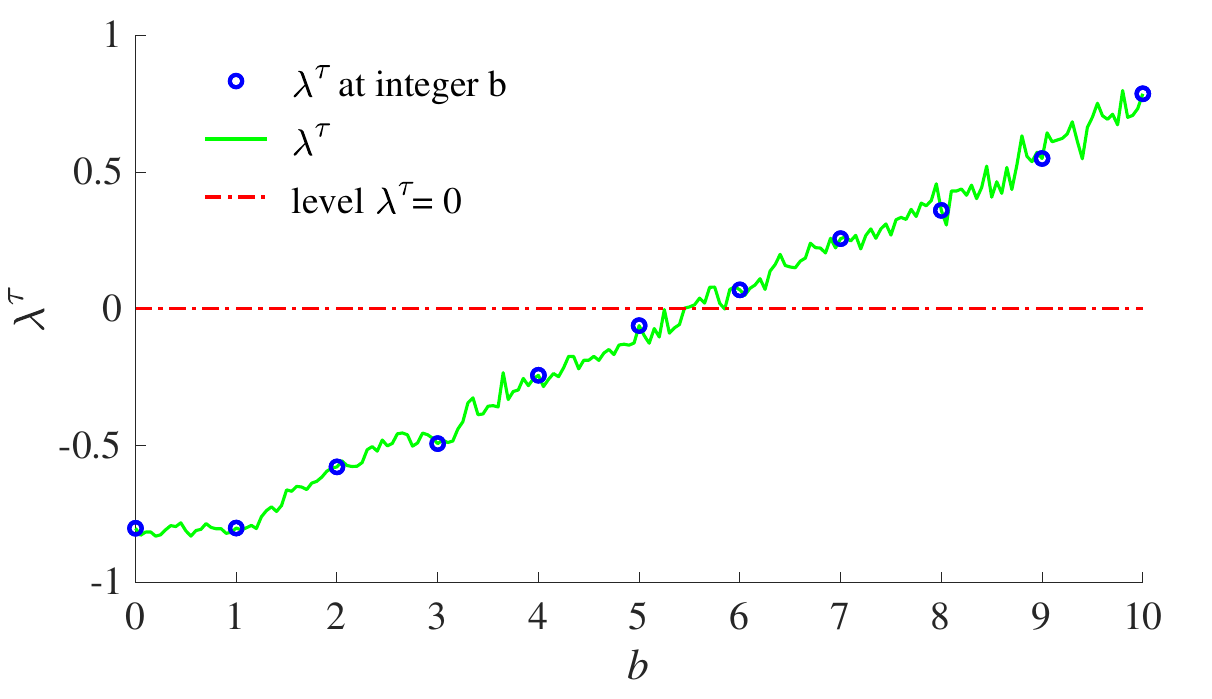}
		\end{minipage}
	}
	\\
	\subfloat[\centering $b=4$ and $5$]{
		\begin{minipage}[t]{0.3\linewidth}
			\centering
			\includegraphics[scale = 0.32]{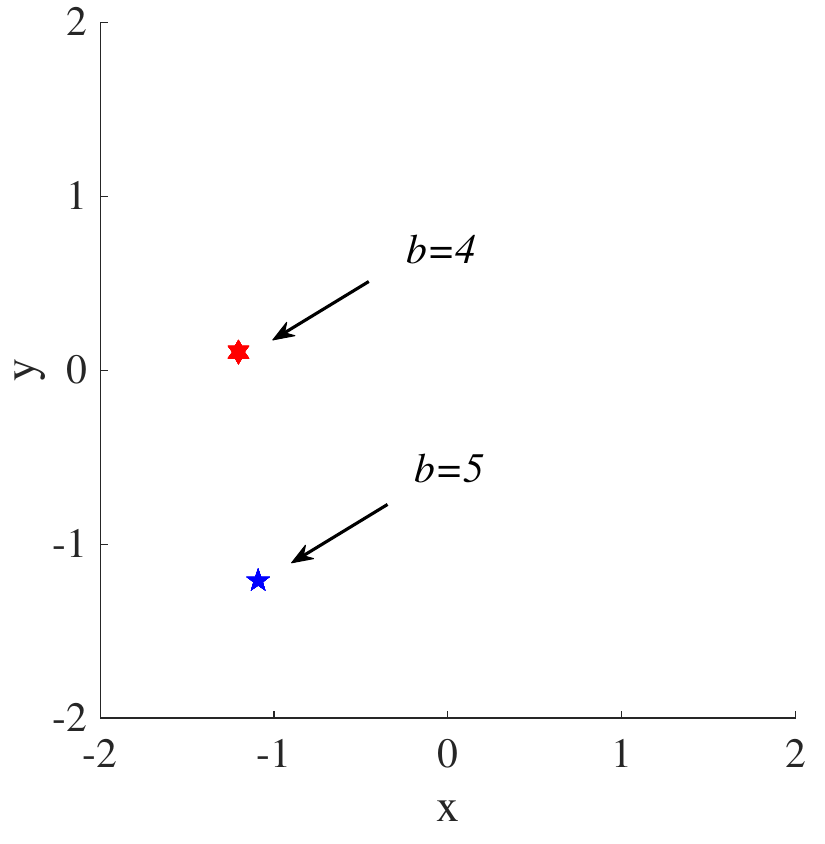} 
		\end{minipage}	} 
	\subfloat[\centering $b=6$]{
		\begin{minipage}[t]{0.3\linewidth}
			\centering
			\includegraphics[scale = 0.32]{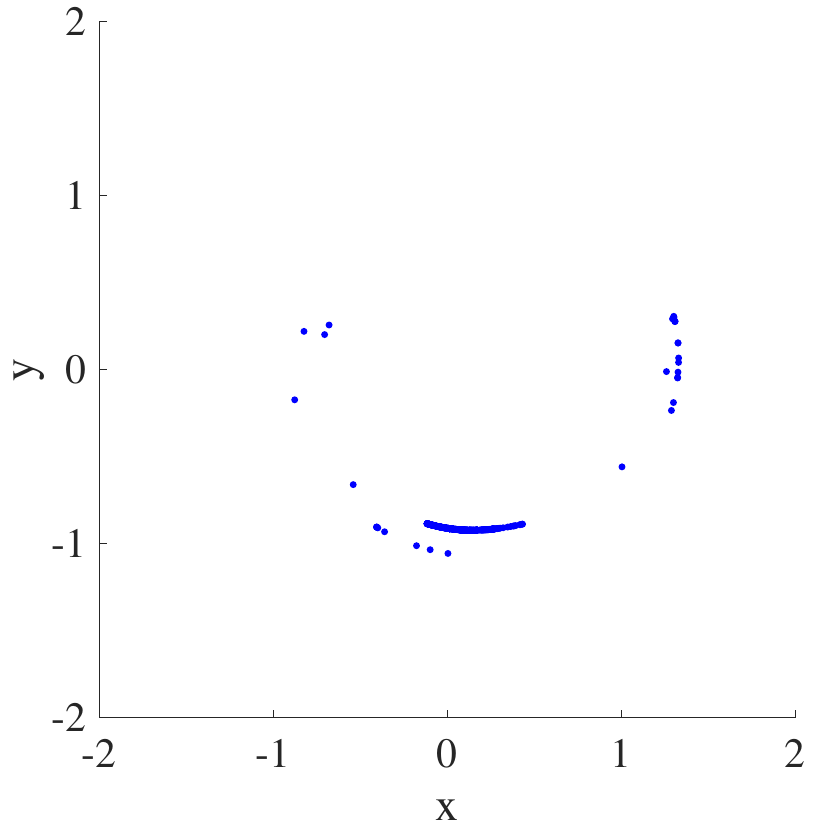} 
	\end{minipage}	} 
	\subfloat[\centering $b=7$]{
		\begin{minipage}[t]{0.3\linewidth}
			\centering
			\includegraphics[scale = 0.32]{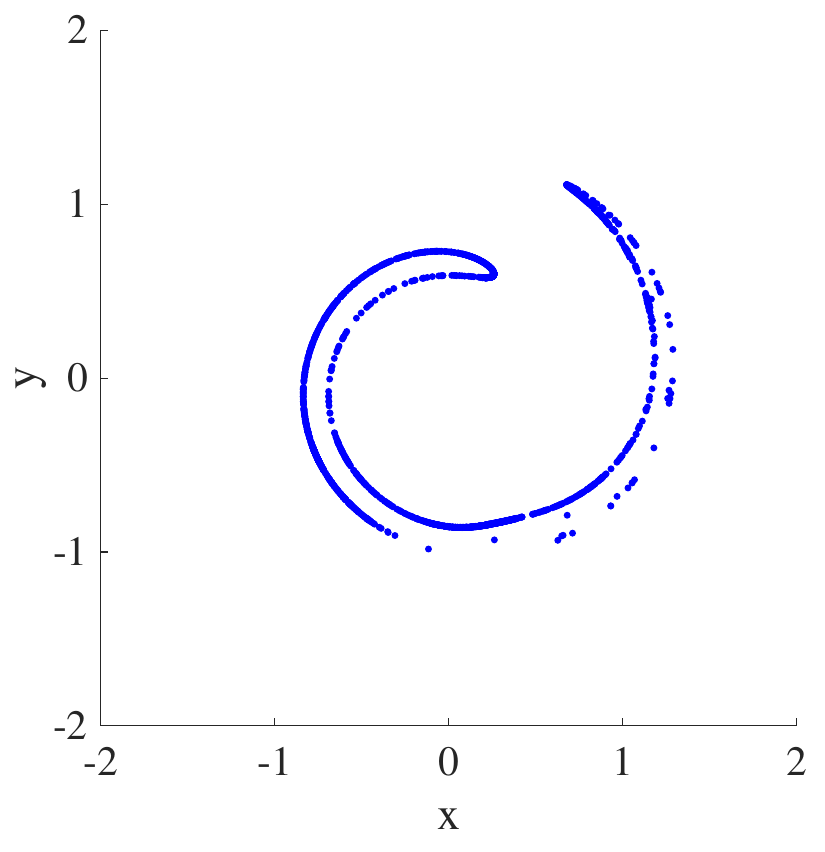} 
	\end{minipage}	}  
	\caption{In (A) we depict $\lambda^{\tau}$ as a function of $b$ for fixed $\alpha=\beta=a=\sigma=1$ computed to $T=1000$ with step size $\tau=10^{-5}$. The red dashed line demarcates the level $\lambda^{\tau}=0$, approximately situated between $b=5$ and $b=6$. We also present the phase portraits for different parameters $b$ in (B)--(D). It seems that the bifurcation may occur within the interval $b\in(5,6)$, which supports the observation in (A).}
	\label{Fig:plotb}
\end{figure}

\begin{figure}[tbhp]
	\centering
	\subfloat[\centering $T=0.1$]{
		\begin{minipage}[t]{0.3\linewidth}
			\centering
			\includegraphics[scale=0.3]{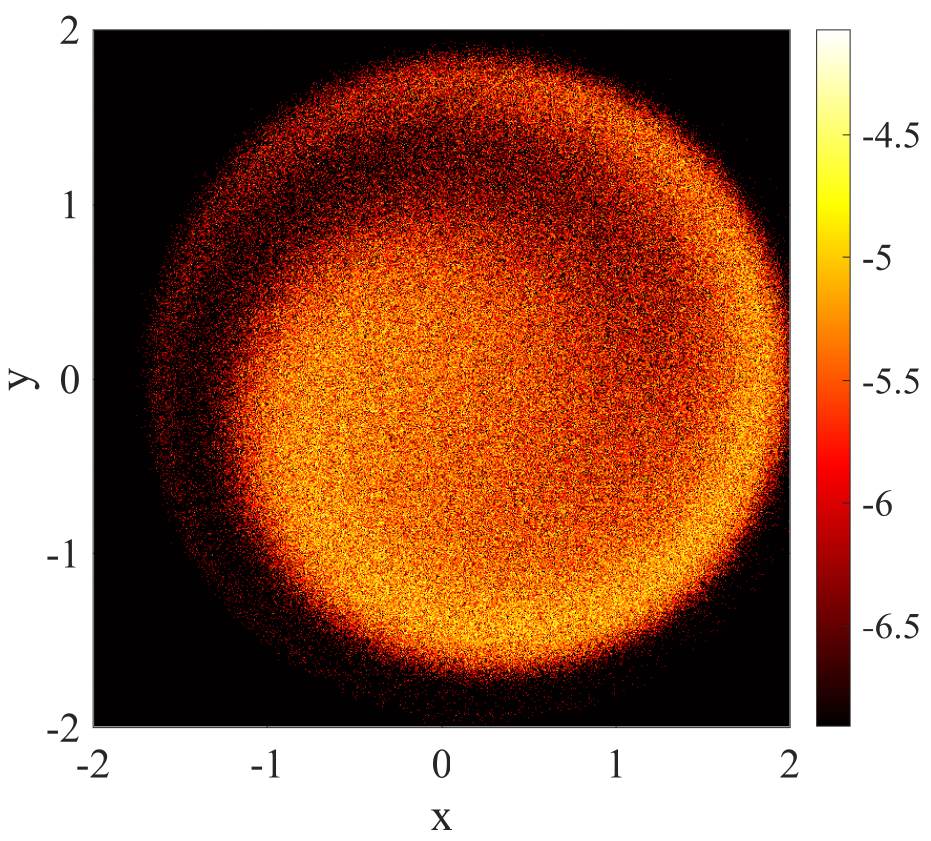} 
		\end{minipage}	} 
	\subfloat[\centering $T=500$]{
		\begin{minipage}[t]{0.3\linewidth}
			\centering
			\includegraphics[scale=0.3]{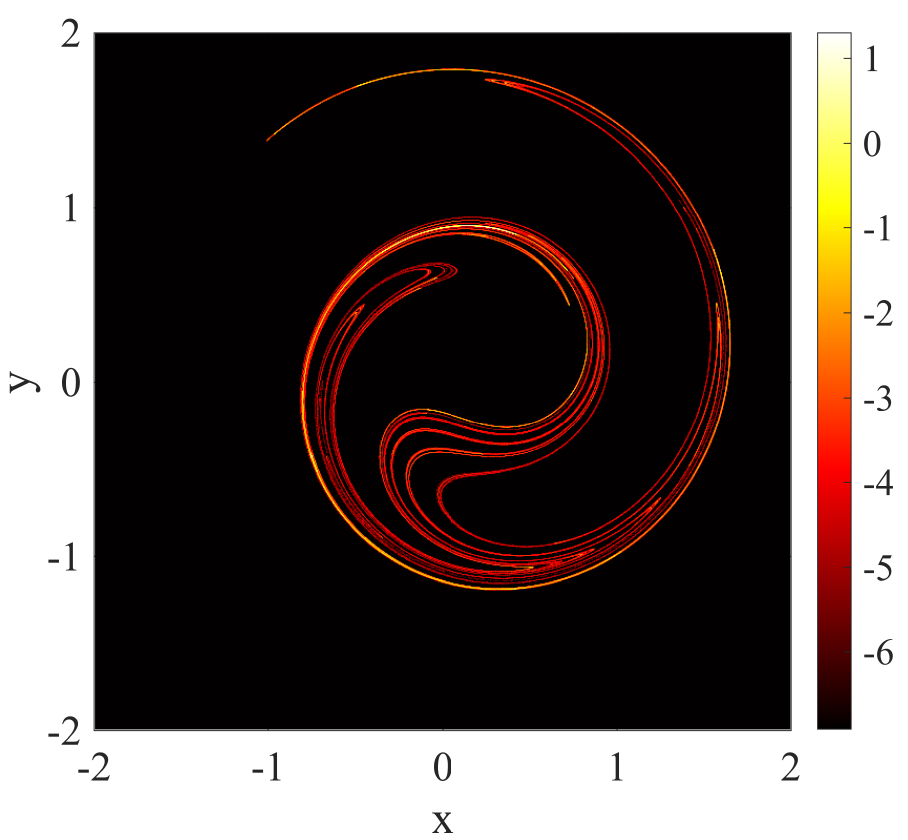} 
		\end{minipage}	} 
	\subfloat[\centering $T=500.1$]{
		\begin{minipage}[t]{0.3\linewidth}
			\centering
			\includegraphics[scale=0.3]{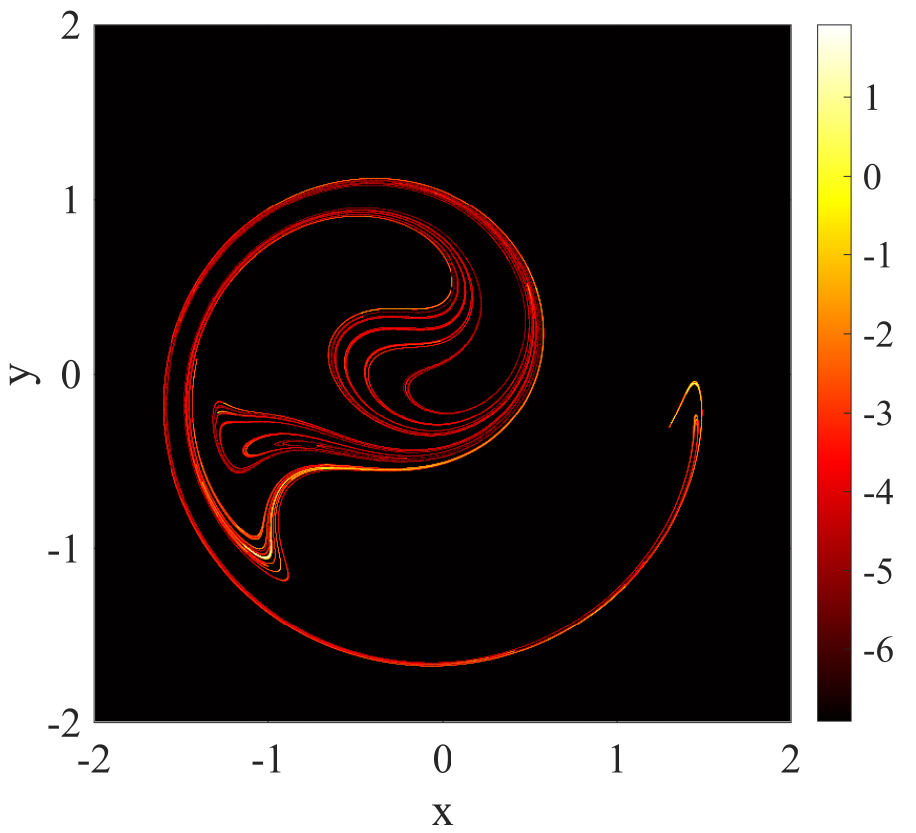} 
	\end{minipage}	} 
\\
	\subfloat[\centering $T=500.2$]{
		\begin{minipage}[t]{0.3\linewidth}
			\centering
			\includegraphics[scale=0.3]{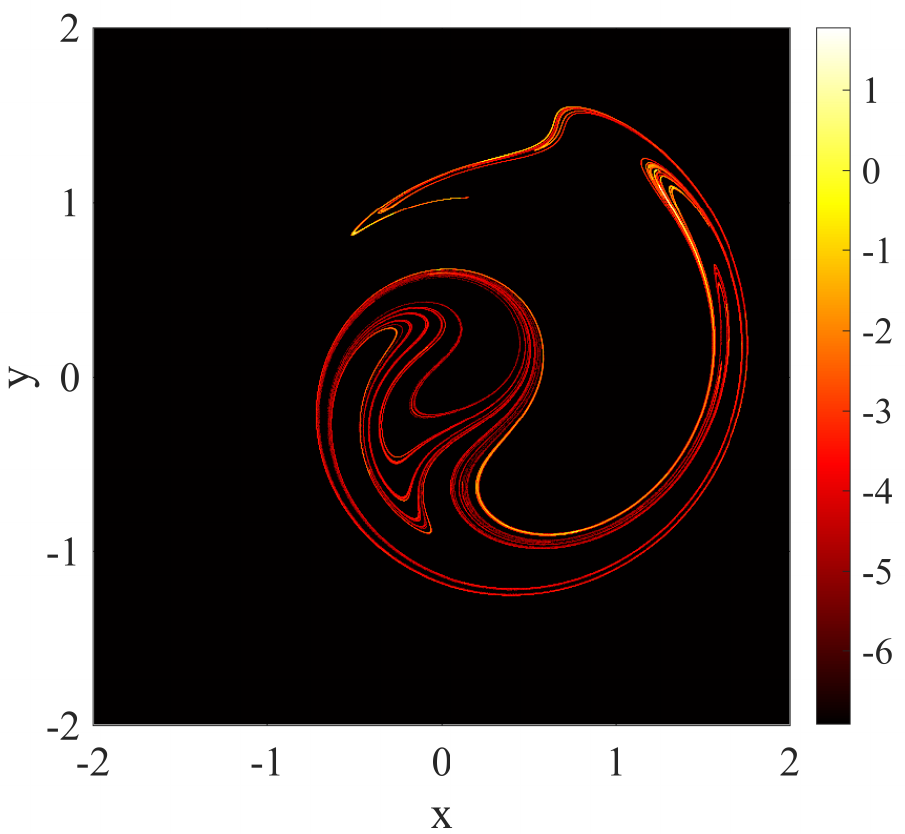} 
	\end{minipage}	} 
\subfloat[\centering $T=500.3$]{
	\begin{minipage}[t]{0.3\linewidth}
		\centering
		\includegraphics[scale=0.3]{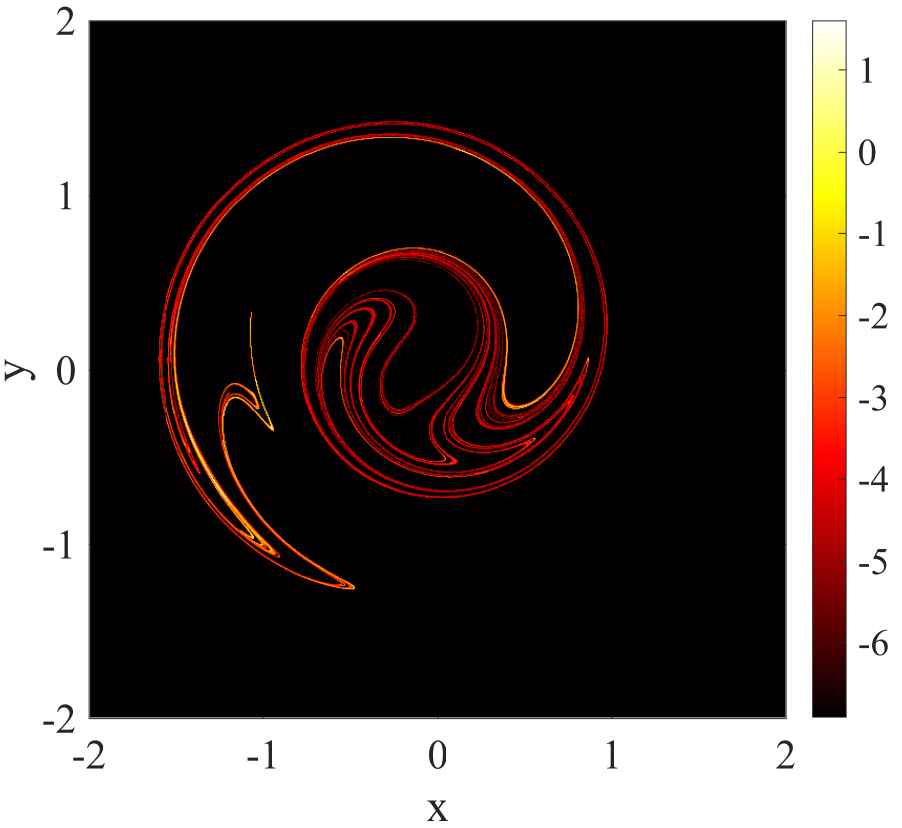} 
\end{minipage}	} 
\subfloat[\centering $T=500.4$]{
	\begin{minipage}[t]{0.3\linewidth}
		\centering
		\includegraphics[scale=0.3]{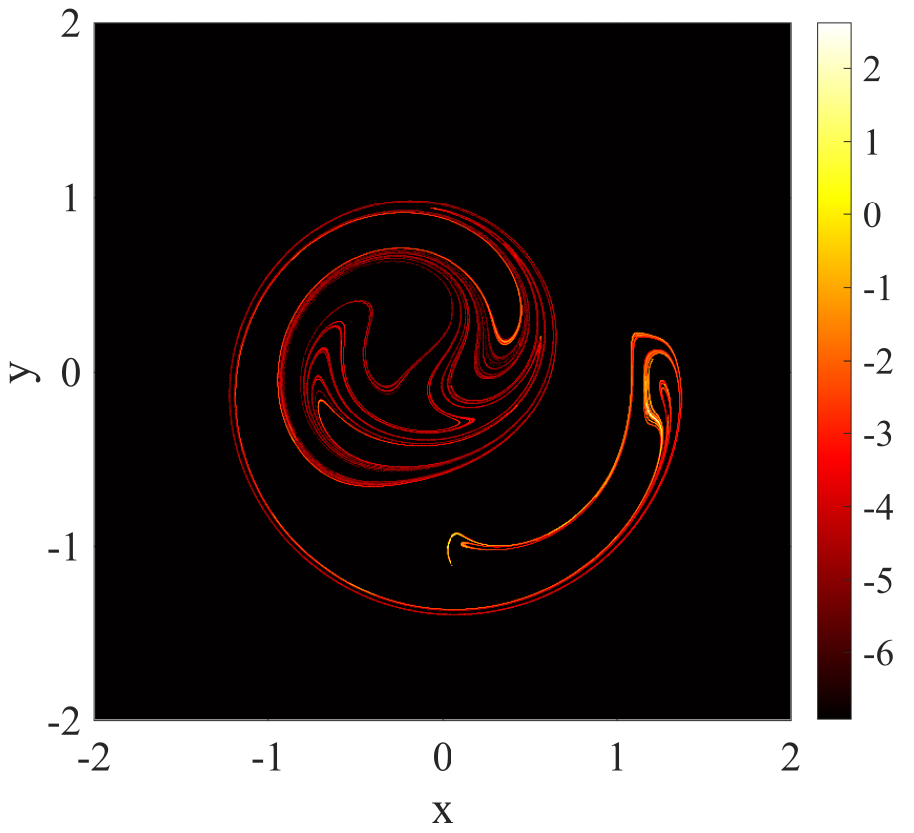} 
\end{minipage}	} 
	\caption{Random attractors $A^{\tau}(\theta_{T}\omega)$ and the associated sample measures $\mu^{\tau}_{\theta_{T}\omega}$ for a fixed realization $\omega$. One million initial points chosen from a two-dimensional normal distribution are used  and the dynamics is computed for different $T$. The parameters are chosen as $b=15$ and $\alpha=\beta=a=\sigma=1$, and the step size $\tau=10^{-4}$. The colorbar to the right is on a log-scale and quantifies the probability to end up in a particular region of phase space. }
	\label{Fig:SRB}
\end{figure}

\appendix
\section{Solvability of the backward Euler method}\label{Appendix:solvable of IES}
We first present the solvablility of the implicit equation \eqref{BEM}. 
Let 
\begin{align*}
	\begin{cases}
		F_1(\bar{x},\bar{y},w_1,w_2,x,y) = x-\bar{x}-\tau\left( \alpha x - \beta y - (ax+by)(x^2+y^2) \right) - \sigma w_1, \\
		F_2(\bar{x},\bar{y},w_1,w_2,x,y) = y-\bar{y}-\tau\left( \beta x + \alpha y + (bx-ay)(x^2+y^2) \right) - \sigma w_2. 
	\end{cases}
\end{align*}
The backward Euler method \eqref{BEM} can be written as 
\begin{equation}\label{solvable 1}
	\begin{cases}
		F_{1}(x^{\tau}_{k}, y^{\tau}_{k}, \Delta W_{k+1}^1, \Delta W_{k+1}^2, x^{\tau}_{k+1}, y^{\tau}_{k+1}) = 0, \\[0.2em]
		F_{2}(x^{\tau}_{k}, y^{\tau}_{k}, \Delta W_{k+1}^1, \Delta W_{k+1}^2, x^{\tau}_{k+1}, y^{\tau}_{k+1}) = 0, 
	\end{cases}
\end{equation}
where $(\Delta W_{k+1}^1 , \Delta W_{k+1}^2)^{\top}=\Delta W_{k+1}$. 
It can be verified that when $\tau<\frac{1}{1+|\alpha|}$, 
\begin{align*}
	\begin{bmatrix}
		x \\[.2em] y 
	\end{bmatrix} \cdot 
	\begin{bmatrix}
		F_1(\bar{x},\bar{y},w_1,w_2,x,y) + (\bar{x}+\sigma w_1) \\[.2em]
		F_2(\bar{x},\bar{y},w_1,w_2,x,y) + (\bar{y}+\sigma w_2)
	\end{bmatrix}
	= (1-\alpha\tau)(x^2+y^2) + a\tau(x^2+y^2)^2 > 0. 
\end{align*}
Thus \eqref{solvable 1} has a solution for all $k\geq0$ whenever $\tau<\frac{1}{1+|\alpha|}$ (see, e.g., \cite[Theorem 6.3.4]{Ortega2000}). 

The Jacobi matrix of $F_i(\bar{x},\bar{y},w_1,w_2,x,y)$, $i=1, 2$ reads
\begin{equation*}
	\frac{\partial (F_1, F_2)}{\partial (x, y)} = \begin{bmatrix}
		1-\alpha\tau  + \tau (3ax^2+ay^2+2bxy) & \beta\tau  + \tau (bx^2+3by^2+2axy) \\[0.4em]
		-\beta\tau  - \tau (3bx^2+by^2-2axy) & 1-\alpha\tau  + \tau (ax^2+3ay^2-2bxy)
	\end{bmatrix}. 
\end{equation*}
The determinant of $\partial (F_1, F_2) / \partial (x, y)$ is 
\begin{align*}
	\det \left(\frac{\partial (F_1, F_2)}{\partial (x, y)}\right) 
	&= 
	1-2\alpha\tau + (\beta^2+\alpha^2)\tau^2 + (4a\tau-4a\alpha\tau^2+4b\beta\tau^2)  (x^2+y^2) \\
	&\quad + \tau^2 \left( (3a^2+3b^2)(x^4+y^4) + (6a^2+6b^2)x^2y^2 \right). 
\end{align*}
When $\tau<\min\{\frac{1}{1+4|\alpha|},\frac{a}{1+|a\alpha|+|b\beta|}\}$, we can deduce that the determinant of $\partial (F_1, F_2) / \partial (x, y)$ is bounded from below by a strictly positive constant uniform with respect to $\tau$.  
Hence by the implicit function theorem, we conclude that $X^{\tau}_{k}$, $k\geq0$ are well-defined whenever $\tau<\min\{\frac{1}{1+4|\alpha|},\frac{a}{1+|a\alpha|+|b\beta|}\}$.

We next consider the solvability of $U^{\tau}_{k+1}$ determined by \eqref{numerical variational SDE}, which can be written as a system of linear equations, i.e., 
\begin{equation}\label{MU}
	M(x^{\tau}_{k+1},y^{\tau}_{k+1}) U^{\tau}_{k+1}
	= U^{\tau}_{k}, 
\end{equation}
where the coefficient matrix $M(x,y)$ is given by
\begin{equation}\label{matrix M}
	M(x,y) = \begin{bmatrix}
		1-\alpha\tau+\tau\left(3ax^2+2bxy+ay^2\right) & \beta\tau+\tau\left(bx^2+2axy+3by^2\right) \\[.6em]
		-\beta\tau-\tau\left(3bx^2-2axy+by^2\right) & 1-\alpha\tau+\tau\left(ax^2-2bxy+3ay^2\right)
	\end{bmatrix}. 
\end{equation}
Direct calculation gives the determinant of $M(x,y)$, 
\begin{align*}
	\det M(x,y) &= (1-\alpha\tau)^2 + (\tau-\alpha\tau^2)(4ax^2+4by^2) + \tau^2 \beta^2 \\
	&\quad + \tau^2 \left( \beta(4bx^2+4ay^2) +(3a^2+3b^2)(x^4+y^4) + (6a^2+6b^2)x^2y^2 \right). 
\end{align*} 
It can be seen that if $\tau<\frac{1}{1+|\alpha|}$, then $\det M(x^{\tau}_{k+1},y^{\tau}_{k+1})>0$ for all $x^{\tau}_{k+1}, y^{\tau}_{k+1}\in\mathbb{R}$. 
Thus $U^{\tau}_{k}$, $k\geq0$ are well-defined whenever $\tau<\frac{1}{1+|\alpha|}$.

\section{Calculation of $c^{\tau}_{k} $ and $s^{\tau}_{k}$}\label{Appendix:ck and sk}
Denote $\begin{bmatrix}
	u^{\tau}_k \\[.2em] v^{\tau}_k
\end{bmatrix} = U^{\tau}_{k} = r_{k} \begin{bmatrix}
c^{\tau}_{k} \\ s^{\tau}_{k}
\end{bmatrix}$ with $r_{k}=\|U^{\tau}_{k}\|$, $c^{\tau}_{k} = \cos\xi^{\tau}_{k}$, and $s^{\tau}_{k} = \sin\xi^{\tau}_{k}$. 
To facilitate subsequent analysis, we introduce $(A^{c}_{1},A^{s}_{1})^{\top} = r_{k+1}^{-1} (U^{\tau}_{k+1}-U^{\tau}_{k})$. It follows from  \eqref{numerical variational SDE} that 
\begin{equation}\label{A1}
	\begin{bmatrix}
		A^{c}_{1} \\ A^{s}_{1}
	\end{bmatrix} = \tau \begin{bmatrix}
	\alpha & -\beta \\ \beta & \alpha 
	\end{bmatrix} \begin{bmatrix}
	c^{\tau}_{k+1} \\ s^{\tau}_{k+1}
	\end{bmatrix} 
	+ \|X^{\tau}_{k+1}\|^2 \begin{bmatrix}
	-a & -b \\ b & -a 
	\end{bmatrix} \begin{bmatrix}
	c^{\tau}_{k+1} \\ s^{\tau}_{k+1}
	\end{bmatrix} + 2 X^{\tau}_{k+1} \cdot \begin{bmatrix}
	c^{\tau}_{k+1} \\ s^{\tau}_{k+1}
	\end{bmatrix}  \begin{bmatrix}
	-a & -b \\ b & -a 
	\end{bmatrix} X^{\tau}_{k+1}. 
\end{equation}
By direct calculation, we find that 
\begin{equation}\label{ck1}
	\begin{aligned}
	c^{\tau}_{k+1}-c^{\tau}_{k} 
%	&= \frac{u^{\tau}_{k+1}}{r_{k+1}} - \frac{u^{\tau}_{k}}{r_{k}} = \frac{r_{k}u^{\tau}_{k+1}-r_{k+1}u^{\tau}_{k}}{r_{k}r_{k+1}} 
%	= \frac{r_{k}(u^{\tau}_{k+1}-u^{\tau}_{k})-(r_{k+1}-r_{k})u^{\tau}_{k}}{r_{k}r_{k+1}} \\
%	&= \frac{u^{\tau}_{k+1}-u^{\tau}_{k}}{r_{k+1}} -  \frac{\left((u^{\tau}_{k+1})^2+(v^{\tau}_{k+1})^2-(u^{\tau}_{k})^2-(v^{\tau}_{k})^2\right)u^{\tau}_{k}}{r_{k}r_{k+1}(r_{k+1}+r_{k})} \\
	&= \frac{u^{\tau}_{k+1}-u^{\tau}_{k}}{r_{k+1}} -  \frac{(u^{\tau}_{k+1}+u^{\tau}_{k})(u^{\tau}_{k+1}-u^{\tau}_{k})u^{\tau}_{k}}{r_{k}r_{k+1}(r_{k+1}+r_{k})} - \frac{(v^{\tau}_{k+1}+v^{\tau}_{k})(v^{\tau}_{k+1}-v^{\tau}_{k})u^{\tau}_{k}}{r_{k}r_{k+1}(r_{k+1}+r_{k})} \\
	&=: A^{c}_{1} - A^{c}_{2} - A^{c}_{3}, 
\end{aligned}
\end{equation}
where 
\begin{align*}
	A^{c}_{2}
%	&=  \frac{(u^{\tau}_{k+1}+u^{\tau}_{k})c^{\tau}_{k}}{r_{k+1}+r_{k}} A^{c}_{1}
%	= c^{\tau}_{k}c^{\tau}_{k+1}A^{c}_{1} - c^{\tau}_{k}A^{c}_{1} \left( c^{\tau}_{k+1} -  \frac{u^{\tau}_{k+1}+u^{\tau}_{k}}{r_{k+1}+r_{k}} \right) \\
%	&= c^{\tau}_{k}c^{\tau}_{k+1}A^{c}_{1} - c^{\tau}_{k}A^{c}_{1}   \frac{c^{\tau}_{k+1}r_{k}-u^{\tau}_{k}}{r_{k+1}+r_{k}} \\
	&= (c^{\tau}_{k+1})^2A^{c}_{1} + c^{\tau}_{k+1}(c^{\tau}_{k}-c^{\tau}_{k+1})A^{c}_{1} - c^{\tau}_{k}A^{c}_{1}   \frac{c^{\tau}_{k+1}r_{k}-u^{\tau}_{k}}{r_{k+1}+r_{k}}, \\
	A^{c}_{3}
%	&=  \frac{(v^{\tau}_{k+1}+v^{\tau}_{k})c^{\tau}_{k}}{r_{k+1}+r_{k}} \frac{v^{\tau}_{k+1}-v^{\tau}_{k}}{r_{k+1}}
%	= c^{\tau}_{k}s^{\tau}_{k+1}\frac{v^{\tau}_{k+1}-v^{\tau}_{k}}{r_{k+1}} - c^{\tau}_{k}\frac{v^{\tau}_{k+1}-v^{\tau}_{k}}{r_{k+1}} \left( s^{\tau}_{k+1} -  \frac{v^{\tau}_{k+1}+v^{\tau}_{k}}{r_{k+1}+r_{k}} \right) \\
%	&= c^{\tau}_{k}s^{\tau}_{k+1}\frac{v^{\tau}_{k+1}-v^{\tau}_{k}}{r_{k+1}} - c^{\tau}_{k}\frac{v^{\tau}_{k+1}-v^{\tau}_{k}}{r_{k+1}}  \frac{s^{\tau}_{k+1}r_{k}-v^{\tau}_{k}}{r_{k+1}+r_{k}} \\
	&= c^{\tau}_{k+1}s^{\tau}_{k+1}A^{s}_{1} + s^{\tau}_{k+1}(c^{\tau}_{k}-c^{\tau}_{k+1})A^{s}_{1} - c^{\tau}_{k}A^{s}_{1}   \frac{s^{\tau}_{k+1}r_{k}-v^{\tau}_{k}}{r_{k+1}+r_{k}}. 
\end{align*}
Substituting the expressions for $A^{c}_{2}$ and $A^{c}_{3}$ into \eqref{ck1}, we obtain
\begin{equation}\label{ck}
	\begin{aligned}
	c^{\tau}_{k+1}-c^{\tau}_{k}
%	&= (s^{\tau}_{k+1})^2A^{c}_{1} - c^{\tau}_{k+1}s^{\tau}_{k+1}\frac{v^{\tau}_{k+1}-v^{\tau}_{k}}{r_{k+1}} 
%	+ \Big( -c^{\tau}_{k+1}(c^{\tau}_{k}-c^{\tau}_{k+1})A^{c}_{1} + c^{\tau}_{k}A^{c}_{1}   \frac{c^{\tau}_{k+1}r_{k}-u^{\tau}_{k}}{r_{k+1}+r_{k}}  \\
%	&\qquad \qquad - s^{\tau}_{k+1}(c^{\tau}_{k}-c^{\tau}_{k+1})\frac{v^{\tau}_{k+1}-v^{\tau}_{k}}{r_{k+1}} 
%	+ c^{\tau}_{k}\frac{v^{\tau}_{k+1}-v^{\tau}_{k}}{r_{k+1}}  \frac{s^{\tau}_{k+1}r_{k}-v^{\tau}_{k}}{r_{k+1}+r_{k}}  \Big)  \\
	&= \tau G_{1}(x^{\tau}_{k+1},y^{\tau}_{k+1},c^{\tau}_{k+1},s^{\tau}_{k+1}) + \mathcal{R}_{k+1}^{c}, 
\end{aligned}
\end{equation}
where $G_{1}$ is given by \eqref{G} and 
\begin{equation}\label{Rkc}
	\mathcal{R}_{k+1}^{c} =  (c^{\tau}_{k+1}A^{c}_{1}+s^{\tau}_{k+1}A^{s}_{1})(c^{\tau}_{k+1}-c^{\tau}_{k}) + c^{\tau}_{k}A^{c}_{1}   \frac{c^{\tau}_{k+1}r_{k}-u^{\tau}_{k}}{r_{k+1}+r_{k}}  + c^{\tau}_{k}A^{s}_{1}   \frac{s^{\tau}_{k+1}r_{k}-v^{\tau}_{k}}{r_{k+1}+r_{k}}. 
\end{equation}
By an analogous derivation, we have 
\begin{equation}\label{sk}
	s^{\tau}_{k+1}-s^{\tau}_{k} = \tau G_{2}(x^{\tau}_{k+1},y^{\tau}_{k+1},c^{\tau}_{k+1},s^{\tau}_{k+1}) + \mathcal{R}_{k+1}^{s}, 
\end{equation}
where $G_{2}$ is given by \eqref{G} and 
\begin{equation}\label{Rks}
		\mathcal{R}_{k+1}^{s} =  (s^{\tau}_{k+1}A^{s}_{1}+c^{\tau}_{k+1}A^{c}_{1})(s^{\tau}_{k+1}-s^{\tau}_{k}) + s^{\tau}_{k}A^{s}_{1}   \frac{s^{\tau}_{k+1}r_{k}-v^{\tau}_{k}}{r_{k+1}+r_{k}}  + s^{\tau}_{k}A^{c}_{1}  \frac{c^{\tau}_{k+1}r_{k}-u^{\tau}_{k}}{r_{k+1}+r_{k}}. 
\end{equation}

Next we provide the estimates of $\mathcal{R}_{k+1}^{c}$ and $\mathcal{R}_{k+1}^{s}$. We first present a rough estimate. It follows from \eqref{A1} and $\max\{|c^{\tau}_{k}|,|s^{\tau}_{k}|\}\leq1$ for all $k\geq0$ that 
\begin{align*}
	|A^{c}_{1}| &\leq  \tau  \left(\alpha +\beta+ (4a+2b)|x^{\tau}_{k+1}|^2 + (2a+4b) |y^{\tau}_{k+1}|^2 \right), \\
	|A^{s}_{1}| &\leq  \tau  \left(\alpha +\beta+ (2a+4b)|x^{\tau}_{k+1}|^2 + (4a+2b) |y^{\tau}_{k+1}|^2 \right). 
\end{align*}
By noting that $r_{k}\geq0$, we have 
\begin{equation}\label{Rcs}
	\begin{aligned}
	|\mathcal{R}_{k+1}^{c}| &\leq  2|c^{\tau}_{k+1}-c^{\tau}_{k}| |A^{c}_{1}| + |c^{\tau}_{k+1}-c^{\tau}_{k}| |A^{s}_{1}| +  |s^{\tau}_{k+1}-s^{\tau}_{k}| |A^{s}_{1}|, \\
	|\mathcal{R}_{k+1}^{s}| &\leq  2|s^{\tau}_{k+1}-s^{\tau}_{k}| |A^{s}_{1}| + |c^{\tau}_{k+1}-c^{\tau}_{k}| |A^{c}_{1}| + |s^{\tau}_{k+1}-s^{\tau}_{k}| |A^{c}_{1}|. 
\end{aligned}
\end{equation}
Subsequently, we refine the estimate \eqref{Rcs}. 
By \eqref{ck} and \eqref{sk}, we deduce that  
\begin{equation}\label{ck+1-ck}
	\begin{aligned}
	|c^{\tau}_{k+1} - c^{\tau}_{k}| &\leq \tau \left( \beta+(4a+6b)\|X^{\tau}_{k+1}\|^2 \right) + |\mathcal{R}_{k+1}^{c}| \\
	&\leq \tau \left( \beta+(4a+6b)\|X^{\tau}_{k+1}\|^2 \right) + 2|A^{c}_{1}| + 2|A^{s}_{1}| \\
	&\leq \tau \left( 4\alpha+5\beta+(16a+18b)\|X^{\tau}_{k+1}\|^2 \right), 
\end{aligned}
\end{equation}
and similarly, 
\begin{equation}\label{sk+1-sk}
	|s^{\tau}_{k+1} - s^{\tau}_{k}| \leq \tau \left( 4\alpha+5\beta+(16a+18b)\|X^{\tau}_{k+1}\|^2 \right). 
\end{equation}
Substituting \eqref{ck+1-ck} and \eqref{sk+1-sk} into \eqref{Rcs}, we obtain the refined estimate of $\mathcal{R}_{k+1}^{c}$ and $\mathcal{R}_{k+1}^{s}$, i.e., 
\begin{align*}
	&|\mathcal{R}_{k+1}^{c}| + |\mathcal{R}_{k+1}^{s}| \leq  \left(3|c^{\tau}_{k+1}-c^{\tau}_{k}|+|s^{\tau}_{k+1}-s^{\tau}_{k}|\right) |A^{c}_{1}| + \left(3|s^{\tau}_{k+1}-s^{\tau}_{k}|+|c^{\tau}_{k+1}-c^{\tau}_{k}|\right) |A^{s}_{1}| \\
	&\leq 4\tau^2  \left( 4\alpha+5\beta+(16a+18b)\|X^{\tau}_{k+1}\|^2 \right) \left(2\alpha +2\beta+ 6(a+b) \|X^{\tau}_{k+1}\|^2 \right). 
\end{align*} 
In addition, by using $\max \{|c^{\tau}_{k}|,|s^{\tau}_{k}|\}\leq1$ for all $k\geq0$ and \eqref{Rcs},  we have 
\begin{equation}\label{estimate R}
	\begin{aligned}
		|s^{\tau}_{k+1}\mathcal{R}_{k+1}^{c}-c^{\tau}_{k+1}\mathcal{R}_{k+1}^{s}| &\leq |\mathcal{R}_{k+1}^{c}|+|\mathcal{R}_{k+1}^{s}| \leq 4|A^{c}_{1}| + 4|A^{s}_{1}| \\
		&\leq 4\tau  \left(2\alpha +2\beta+ 6(a+b) \|X^{\tau}_{k+1}\|^2 \right). 
	\end{aligned} 
\end{equation}

\section{Minorization condition for $\{(X^{\tau}_k,\xi^{\tau}_k):k\geq0\}$}\label{Appendix:minorization}
Denote $\rho_{k} = \tan\xi^{\tau}_{k} = v^{\tau}_{k} / u^{\tau}_{k}$, where $u^{\tau}_{k}$ and $v^{\tau}_{k}$ are two components of $U^{\tau}_{k}$ determined by \eqref{numerical variational SDE}. It follows from \eqref{MU} that 
\begin{equation}\label{rho}
	\rho_{k+1} = \frac{-M_{21}(x^{\tau}_{k+1},y^{\tau}_{k+1}) + M_{11}(x^{\tau}_{k+1},y^{\tau}_{k+1})\rho_{k}}{M_{22}(x^{\tau}_{k+1},y^{\tau}_{k+1}) - M_{12}(x^{\tau}_{k+1},y^{\tau}_{k+1})\rho_{k}}, 
\end{equation}
where the matrix $M(x,y)=[M_{ij}(x,y)]_{i, j=1, 2}$ is given by \eqref{matrix M}. 
The verification of the minorization condition for process $\{(X^{\tau}_k,\xi^{\tau}_k): k\geq0\}$ is divided into three cases. 

\textit{Case 1. }
First, for any initial values $(x^{\tau}_{0},y^{\tau}_{0})^{\top} \in \mathbb{R}^2$ and $\xi^{\tau}_{0} \in (\tfrac{\pi}{4}, \tfrac{\pi}{2})$ (i.e., $\rho_0\geq1$), we take $y^{\tau}_{1}=1$, 
\begin{equation*}
	\rho^{*}_{1}=\max\{ 3\tan(\tfrac{\pi}{2}-\delta), 3\rho_{0} \}, \quad \tau < \min \left\{ \frac{1}{1+4|\beta|\rho^{*}_{1}}, \frac{1}{1+4|\alpha|} \right\} , \quad \text{for} \ \ 0<\delta<\frac{\pi}{4}, 
\end{equation*}
and 
\begin{equation*}
	x^{\tau}_{1} = \left( \frac{1}{\tau} \frac{(1-\alpha\tau)(\rho^{*}_{1}-\rho_{0})-\beta\tau\rho_{0}\rho^{*}_{1}-\beta\tau}{3b-a\rho^{*}_{1}+b\rho_{0}\rho^{*}_{1}+3a\rho_{0}} \right)^{\frac{1}{2}}. 
\end{equation*}
It is worth noting that the expression for $x^{\tau}_{1}$ is meaningful since both its numerator and denominator are strictly positive. Then by \eqref{rho} we have $\rho_{1}= \rho^{*}_{1}>\rho_{0}$. 
Second, by taking $x^{\tau}_{2} = y^{\tau}_{2} = 0$, we obtain 
\begin{equation*}
	\rho_2 = \frac{\beta\tau+(1-\alpha\tau)\rho_1}{1-\alpha\tau-\beta\tau\rho_1} > \frac{1}{3} \rho_{1}, 
\end{equation*}
which means that $\rho_{2}>\tan(\tfrac{\pi}{2}-\delta)$ and thus $\xi^{\tau}_{2} \in (\xi^{*}_{2}, \tfrac{\pi}{2})$ with $\xi^{*}_{2}\in(\tfrac{\pi}{2}-\delta,\tfrac{\pi}{2})$. 
In this stage, we have shown that for any given $(x^{\tau}_{0},y^{\tau}_{0},\xi^{\tau}_{0})^{\top} \in \mathbb{R}^2\times(\tfrac{\pi}{4},\tfrac{\pi}{2})$, we can adjust $x^{\tau}_{1}$ and $y^{\tau}_{1}$ to ensure that $(x^{\tau}_{2},y^{\tau}_{2},\xi^{\tau}_{2})^{\top}$ can be set to $(0,0,\xi^*_{2})^{\top}$. 

In the following we will demonstrate that $(x^{\tau}_{2},y^{\tau}_{2},\xi^{\tau}_{2})^{\top}$ can enter a region in $\mathbb{R}^3$ with positive probability. 
It follows from \eqref{rho} that the function 
\begin{equation*}
	\rho_{2} = \frac{-M_{21}(x^{\tau}_{2},y^{\tau}_{2}) + M_{11}(x^{\tau}_{2},y^{\tau}_{2})\rho_{1}}{M_{22}(x^{\tau}_{2},y^{\tau}_{2}) - M_{12}(x^{\tau}_{2},y^{\tau}_{2})\rho_{1}} =: f_{\rho_{2}}(x^{\tau}_{2},y^{\tau}_{2},\rho_{1}) 
\end{equation*}
is continuous within the small neighborhood of $(0,0,\rho^{*}_{1})^{\top}$, which, together with \eqref{BEM2}, implies that 
\begin{equation*}
	\rho_{2} = f_{\rho_{2}}(x^{\tau}_{2},y^{\tau}_{2},\rho_{1}) = f_{\rho_{2}}(\cdot,\cdot,\rho_{1})\circ F_{\tau}(x^{\tau}_{1}+\sigma \Delta W^{1}_{2},y^{\tau}_{1}+\sigma \Delta W^{2}_{2}). 
\end{equation*}
An analogous analysis yields 
\begin{equation*}
	\rho_{1} = f_{\rho1}(\cdot,\cdot,\rho_{0})\circ F_{\tau}(x^{\tau}_{0}+\sigma \Delta W^{1}_{1},y^{\tau}_{0}+\sigma \Delta W^{2}_{1}). 
\end{equation*}
Recalling $\xi^{\tau}_{k} = \arctan \rho_{k}$, we can deduce that there exists a continuous function $H:\mathbb{R}^{7}\rightarrow\mathbb{R}^{3}$, such that \begin{equation*}
	(x^{\tau}_{2},y^{\tau}_{2},\xi^{\tau}_{2})^{\top} = H(x^{\tau}_{0}, y^{\tau}_{0}, \xi^{\tau}_{0}, \Delta W^{1}_{1}, \Delta W^{2}_{1}, \Delta W^{1}_{2}, \Delta W^{2}_{2}).  
\end{equation*}
Specially, for any $(x^{\tau}_{0}, y^{\tau}_{0}, \xi^{\tau}_{0})^{\top} \in \mathbb{R}^2 \times (\tfrac{\pi}{4}, \tfrac{\pi}{2})$, we can choose $\{\Delta W^{j}_{i}\}_{i,j=1,2}$ to reach $(x^{\tau}_{2}, y^{\tau}_{2}, \xi^{\tau}_{2})^{\top}=(0,0,\xi^{*}_{2})^{\top}$ with $\xi^{*}_{2}\in(\tfrac{\pi}{2}-\delta,\tfrac{\pi}{2})$ for $\delta\in(0,\frac{\pi}{4})$. 
By the continuity of $H$, for any $\delta\in(0,\frac{\pi}{4})$, there exists a constant $\delta_{1}>0$, such that 
\begin{equation*}
	\mathbb{P}\left( (x^{\tau}_{2}, y^{\tau}_{2}, \xi^{\tau}_{2})^{\top} \in B_{\delta}(0,0)\times(\tfrac{\pi}{2}-\delta,\tfrac{\pi}{2}) \right) > \mathbb{P}\left( \{\Delta W^{j}_{i}\}_{i,j=1,2}\in\bm{C}_{\delta_{1}} \right) > 0, 
\end{equation*}
where $B_{\delta}(0,0)\subset\mathbb{R}^2$ is a ball centered at $(0,0)^{\top}$ with radius $\delta$,  $\bm{C}_{\delta_{1}} \subset \mathbb{R}^4$ is a cube depending on $\delta_{1}$, and we have used the fact that the probability of $\Delta W^{j}_{i}$ in such cube is positive. 

\textit{Case 2. }
For the case of $(x^{\tau}_{0},y^{\tau}_{0})^{\top} \in \mathbb{R}^2$ and $\xi^{\tau}_{0} \in (\tfrac{\pi}{2}, \tfrac{3\pi}{4})$, the analysis is analogous to Case 1 so we only sketch it. 
In this case we assume $b>4a$ and $\tau < \frac{a}{1+4b|\beta|}$. 
We first take $y^{\tau}_{1}=0$ and 
\begin{equation*}
	x^{\tau}_{1} = \left( \frac{1}{\tau} \frac{3a\beta\tau\rho_{0}-(1-\alpha\tau)(3a+b\rho_{0})-b\beta\tau}{3b^2+3a^2} \right)^{\frac{1}{2}}. 
\end{equation*}
Then from \eqref{rho} we obtain $\rho_{1}=-\tfrac{3a}{b}$. 
Secondly, we take $y^{\tau}_{2}=0$ and 
\begin{equation*}
	x^{\tau}_{2} = \left( \frac{1}{\tau} \frac{(1-\alpha\tau)(\sqrt{3}-\rho_{1})-\sqrt{3}\beta\tau\rho_{1}-\beta\tau}{3b-\sqrt{3}a+\sqrt{3}b\rho_{1}+3a\rho_{1}} \right)^{\frac{1}{2}}. 
\end{equation*}
Then from \eqref{rho} we obtain $\rho_{2}=\sqrt{3}$, which reduces to Case 1. 

\textit{Case 3. }
We consider the case of $(x^{\tau}_{0},y^{\tau}_{0})^{\top} \in \mathbb{R}^2$ and $\xi^{\tau}_{0} =\tfrac{\pi}{2}$, which corresponds to $\rho^{-1}_{0}=0$. Firstly we take $x^{\tau}_{0}=y^{\tau}_{0}=0$, then from \eqref{rho} we have 
\begin{equation*}
	\rho^{-1}_{1} = \frac{\rho^{-1}_{0}M_{22}(x^{\tau}_{1},y^{\tau}_{1}) - M_{12}(x^{\tau}_{1},y^{\tau}_{1})}{-\rho^{-1}_{0}M_{21}(x^{\tau}_{1},y^{\tau}_{1}) + M_{11}(x^{\tau}_{1},y^{\tau}_{1})}
	= \frac{-M_{12}(x^{\tau}_{1},y^{\tau}_{1})}{M_{11}(x^{\tau}_{1},y^{\tau}_{1})}
	= \frac{-\beta\tau}{1-\alpha\tau}. 
\end{equation*}
Secondly, we take $x^{\tau}_{1}=y^{\tau}_{1}=0$ and 
$\tau < \min \{ \frac{1}{8|\beta|\tan(\frac{\pi}{2}-\delta)}, \frac{1}{1+4|\alpha|}, \frac{1}{1+4|\beta|} \}$ for $\delta\in(0,\frac{\pi}{4})$, 
then again from \eqref{rho} we obtain 
\begin{equation*}
	\rho_{2} = \frac{-M_{21}(x^{\tau}_{2},y^{\tau}_{2}) + M_{11}(x^{\tau}_{2},y^{\tau}_{2})\rho_{1}}{M_{22}(x^{\tau}_{2},y^{\tau}_{2}) - M_{12}(x^{\tau}_{2},y^{\tau}_{2})\rho_{1}} 
	= \frac{\beta^2\tau^2-(1-\alpha\tau)^2}{2\beta\tau(1-\alpha\tau)}. 
\end{equation*}
This means that $\rho_{2}>\tan(\tfrac{\pi}{2}-\delta)$ for $\beta\leq0$ while $\rho_{2}<-\tan(\tfrac{\pi}{2}-\delta)$ for $\beta>0$. 
Consequently, we have  $\xi^{\tau}_{2}=\arctan\rho_{2}\in(\tfrac{\pi}{2}-\delta, \tfrac{\pi}{2}+\delta)$. 

Combining these three cases, for any $(x^{\tau}_{0},y^{\tau}_{0},\xi^{\tau}_{0})^{\top} \in \mathbb{R}^2\times(\tfrac{\pi}{4}, \tfrac{3\pi}{4})$ and $\delta\in(0,\tfrac{\pi}{4})$, when $\tau$ is sufficiently small, we can conclude that  
\begin{equation*}
	\mathbb{P}\left( (x^{\tau}_{4}, y^{\tau}_{4}, \xi^{\tau}_{4})^{\top} \in B_{\delta}(0,0,\tfrac{\pi}{2}) \right) > 0. 
\end{equation*}
%Furthermore, since the process $(x^{\tau}_{k}, y^{\tau}_{k}, \xi^{\tau}_{k})^{\top}$ are continuous with respect to $\{\Delta W^{j}_{i}\}_{i=1,2,...,k; j=1,2}$, and $\{\Delta W^{j}_{i}\}$ have $C^{\infty}$ density, we know that the transition probability for the process $(X^{\tau}_{k},\xi^{\tau}_{k})^{\top}$ has a $C^{\infty}$ density when $\tau$ is sufficiently small. 
This verifies the minorization condition.

\bibliographystyle{plain}
\bibliography{srb.bib}

\begin{thebibliography}{10}

\bibitem{Arnaudon2017}
A.~Arnaudon, A.L. De~Castro, and D.D. Holm.
\newblock Noise and dissipation in rigid body motion.
\newblock In {\em Stochastic geometric mechanics}, volume 202 of {\em Springer
  Proc. Math. Stat.}, pages 1--12. Springer, Cham, 2017.

\bibitem{Arnaudon2018}
A.~Arnaudon, A.L. De~Castro, and D.D. Holm.
\newblock Noise and dissipation on coadjoint orbits.
\newblock {\em J. Nonlinear Sci.}, 28(1):91--145, 2018.

\bibitem{arnoldRDS}
L.~Arnold.
\newblock {\em Random dynamical systems}.
\newblock Springer Monographs in Mathematics. Springer-Verlag, Berlin, 1998.

\bibitem{Arnold2000}
L.~Arnold and P.~Imkeller.
\newblock The {K}ramers oscillator revisited.
\newblock In {\em Stochastic Processes in Physics, Chemistry, and Biology},
  pages 280--291, Berlin, Heidelberg, 2000. Springer Berlin Heidelberg.

\bibitem{Baxendale1994}
Peter~H. Baxendale.
\newblock A stochastic {H}opf bifurcation.
\newblock {\em Probab. Theory Related Fields}, 99(4):581--616, 1994.

\bibitem{Baxendale2024}
P.H. Baxendale.
\newblock Lyapunov exponents and shear-induced chaos for a hopf bifurcation
  with additive noise.
\newblock {\em Probab. Theory Relat. Fields},
  https://doi.org/10.1007/s00440-024-01301-4, 2024.

\bibitem{Blumenthal2019}
A.~Blumenthal and L.-S. Young.
\newblock Equivalence of physical and {SRB} measures in random dynamical
  systems.
\newblock {\em Nonlinearity}, 32(4):1494--1524, 2019.

\bibitem{Chekroun2011}
M.~D. Chekroun, E.~Simonnet, and M.~Ghil.
\newblock Stochastic climate dynamics: random attractors and time-dependent
  invariant measures.
\newblock {\em Phys. D}, 240(21):1685--1700, 2011.

\bibitem{Chemnitz2023}
D.~Chemnitz and M.~Engel.
\newblock Positive {L}yapunov exponent in the {H}opf normal form with additive
  noise.
\newblock {\em Comm. Math. Phys.}, 402(2):1807--1843, 2023.

\bibitem{Chen20252}
C.C. Chen, T.H. Dang, J.L. Hong, and G.T. Song.
\newblock {On numerical discretizations that preserve probabilistic limit
  behaviors for time-homogeneous Markov processes}.
\newblock {\em Bernoulli}, 31(4):3139--3164, 2025.

\bibitem{Chen2025}
C.C. Chen, T.H. Dang, J.L. Hong, and F.S. Zhang.
\newblock A new class of splitting methods that preserve ergodicity and
  exponential integrability for the stochastic {L}angevin equation.
\newblock {\em SIAM J. Numer. Anal.}, 63(2):1000--1024, 2025.

\bibitem{Chen2023}
C.C. Chen, J.L. Hong, and Y.L. Lu.
\newblock Stochastic differential equation with piecewise continuous arguments:
  {M}arkov property, invariant measure and numerical approximation.
\newblock {\em Discrete Contin. Dyn. Syst. Ser. B}, 28(1):765--807, 2023.

\bibitem{Crauel1994}
H.~Crauel and F.~Flandoli.
\newblock Attractors for random dynamical systems.
\newblock {\em Probab. Theory Related Fields}, 100(3):365--393, 1994.

\bibitem{Deville2011}
R.~E.~L. Deville, N.~S. Namachchivaya, and Z.~Rapti.
\newblock Stability of a stochastic two-dimensional non-{H}amiltonian system.
\newblock {\em SIAM J. Appl. Math.}, 71(4):1458--1475, 2011.

\bibitem{Doan2018}
T.S. Doan, M.~Engel, J.~S.~W. Lamb, and M.~Rasmussen.
\newblock Hopf bifurcation with additive noise.
\newblock {\em Nonlinearity}, 31(10):4567--4601, 2018.

\bibitem{Duan2003}
J.Q. Duan, K.N. Lu, and B.~Schmalfuss.
\newblock Invariant manifolds for stochastic partial differential equations.
\newblock {\em Ann. Probab.}, 31(4):2109--2135, 2003.

\bibitem{Eckmann1985}
J.-P. Eckmann and D.~Ruelle.
\newblock Ergodic theory of chaos and strange attractors.
\newblock {\em Rev. Modern Phys.}, 57(3):617--656, 1985.

\bibitem{Flandoli2017}
F.~Flandoli, B.~Gess, and M.~Scheutzow.
\newblock Synchronization by noise.
\newblock {\em Probab. Theory Related Fields}, 168(3-4):511--556, 2017.

\bibitem{Higham2002}
D.~J. Higham, X.R. Mao, and A.~M. Stuart.
\newblock Strong convergence of {E}uler-type methods for nonlinear stochastic
  differential equations.
\newblock {\em SIAM J. Numer. Anal.}, 40(3):1041--1063, 2002.

\bibitem{Keller1999}
H.~Keller and G.~Ochs.
\newblock Numerical approximation of random attractors.
\newblock In {\em Stochastic dynamics ({B}remen, 1997)}, pages 93--115.
  Springer, New York, 1999.

\bibitem{Ledrappier1988}
F.~Ledrappier and L.-S. Young.
\newblock Entropy formula for random transformations.
\newblock {\em Probab. Theory Related Fields}, 80(2):217--240, 1988.

\bibitem{Lin2008}
K.~K. Lin and L.-S. Young.
\newblock Shear-induced chaos.
\newblock {\em Nonlinearity}, 21(5):899--922, 2008.

\bibitem{Mao2013}
X.R. Mao and L.~Szpruch.
\newblock Strong convergence rates for backward {E}uler--{M}aruyama method for
  non-linear dissipative-type stochastic differential equations with
  super-linear diffusion coefficients.
\newblock {\em Stochastics}, 85(1):144--171, 2013.

\bibitem{Mattingly2002}
J.~C. Mattingly, A.~M. Stuart, and D.~J. Higham.
\newblock Ergodicity for {SDE}s and approximations: locally {L}ipschitz vector
  fields and degenerate noise.
\newblock {\em Stochastic Process. Appl.}, 101(2):185--232, 2002.

\bibitem{Ochs2001}
G.~Ochs.
\newblock Random attractors: robustness, numerics and chaotic dynamics.
\newblock In {\em Ergodic theory, analysis, and efficient simulation of
  dynamical systems}, pages 1--30. Springer, Berlin, 2001.

\bibitem{Ortega2000}
J.~M. Ortega and W.~C. Rheinboldt.
\newblock {\em Iterative solution of nonlinear equations in several variables},
  volume~30 of {\em Classics in Applied Mathematics}.
\newblock Society for Industrial and Applied Mathematics (SIAM), Philadelphia,
  PA, 2000.
\newblock Reprint of the 1970 original.

\bibitem{Wieczorek2009}
S.~Wieczorek.
\newblock Stochastic bifurcation in noise-driven lasers and {H}opf oscillators.
\newblock {\em Phys. Rev. E (3)}, 79(3):036209, 10, 2009.

\bibitem{Young2002}
L.-S. Young.
\newblock What are {SRB} measures, and which dynamical systems have them?
\newblock volume 108, pages 733--754. 2002.
\newblock Dedicated to David Ruelle and Yasha Sinai on the occasion of their
  65th birthdays.

\bibitem{Young2017}
L.-S. Young.
\newblock Generalizations of {SRB} measures to nonautonomous, random, and
  infinite dimensional systems.
\newblock {\em J. Stat. Phys.}, 166(3-4):494--515, 2017.

\end{thebibliography}

\end{document}